\documentclass[a4paper,10pt]{article}
\usepackage[T1]{fontenc}
\usepackage{amsmath,amsfonts,amssymb,bbm,tikz}
\usepackage{amsthm}
\usepackage{algorithm}
\usepackage{algpseudocode}
\usepackage{cite}
\usepackage{enumerate}
\usepackage{pgfplots}
\usepackage{xcolor}
\usepackage{pgfplotstable}
\usepackage{booktabs}
\usepackage{geometry}
\usepackage[hidelinks]{hyperref}
\hypersetup{
	colorlinks   = true, 
	urlcolor     = blue, 
	linkcolor    = blue, 
	citecolor   = red 
}
\usepackage{geometry}
\usepackage{comment}
\usepgfplotslibrary{groupplots}
\pdfoutput=1

\usepackage{cleveref}

\def\bar{\overline}\def\tilde{\widetilde}

\def\rank{\mathrm{rank}}
\def\qsrank{\mathrm{qsrank}}

\newcommand{\norm}[1]{\left\lVert#1\right\rVert}
\newcommand{\R}{\mathbb{R}}

\newtheorem{definition}{Definition}[section]

\newtheorem{remark}{Remark}
\newtheorem{example}{Example}
\newtheorem{lemma}[definition]{Lemma}
\newtheorem{theorem}[definition]{Theorem}
\newtheorem{corollary}[definition]{Corollary}

\newcommand{\upd}[1]{{\color{black}#1}} 

\algdef{SE}[DOWHILE]{Do}{doWhile}{\algorithmicdo}[1]{\algorithmicwhile\ #1}%

\author{
	Stefano Massei\footnote{Department of Mathematics, University of Pisa,  Largo Bruno Pontecorvo, 5, 56127 Pisa, Italy.
		E-mail: stefano.massei@unipi.it. } 
	\and
	Luca Saluzzi\footnote{Centro di Ricerche Matematiche "Ennio De Giorgi", 
Scuola Normale Superiore, Piazza dei Cavalieri, 3,
56126 Pisa, Italy. E-mail: luca.saluzzi@sns.it. \\
		Both authors are members of the INdAM/GNCS research group. L. Saluzzi is part of the INdAM - GNCS Project ``Metodi di riduzione di modello ed approssimazioni di rango basso per problemi alto-dimensionali'' (CUP E53C23001670001) and ``titolare di borsa per l’estero dell'Istituto Nazionale di Alta Matematica”.}}

\title{On the data-sparsity of the solution of Riccati equations with  applications to feedback control}

\date{}

\pgfplotsset{compat=1.15}

\begin{document}
	\maketitle
	
\begin{abstract}
Solving large-scale continuous-time algebraic Riccati equations is a significant challenge in various control theory applications. This work demonstrates that when the matrix coefficients of the equation are quasiseparable, the solution also exhibits numerical quasiseparability. This property enables us to develop  two efficient Riccati solvers. The first solver is applicable to the general quasiseparable case, while the second is tailored to the particular case of banded coefficients. Numerical experiments confirm the effectiveness of the proposed algorithms on both synthetic examples and case studies from the control of partial differential equations and agent-based models.
\end{abstract}

	
	\section{Introduction}
The main focus of this work is on continuous-time algebraic Riccati equations (CARE), which are quadratic matrix equations of the form
\begin{equation}\label{eq:care}
\mathcal{R}(X) := A^\top X+XA- XFX+Q=0,	
\end{equation}
where $X\in\mathbb R^{n\times n}$ is the unknown solution, $F,Q\in\mathbb R^{n\times n}$ are symmetric positive semidefinite, and $A\in\mathbb R^{n\times n}$ is such that its \emph{numerical range} 
$$\mathcal W(A):=\{x^HAx:\ x\in\mathbb C^n,\ \norm{x}_2=1\},$$
satisfies $\mathcal W(A)\subset \mathbb C^{-}:=\{z\in\mathbb C:\ \Re(z)<0\}$.
Due to its nonlinearity, equation \eqref{eq:care} has many solutions but, under these assumptions, there exists a unique real symmetric positive semidefinite solution $X$. The latter is usually called the \emph{stabilizing solution} as it has the additional property that the so called \emph{closed-loop matrix} $A-FX$ is stable, i.e., it has eigenvalues with non positive real parts \cite[Theorem 2.18]{bini2011numerical}. Stabilizing solutions play a key role in a number of control theory applications such as linear-quadratic optimal regulators~\cite{sima, lancaster}, linear-quadratic Gaussian balancing~\cite{benner2009,breiten2015feedback}, and  $\mathcal H_2$/$\mathcal H_{\infty}$ feedback synthesis problems~\cite{alla2023,dolgov2023}. When $n$ is small the numerical computation of the stabilizing solution requires the eigen- or \upd{Schur} decomposition of the Hamiltonian matrix $\left[\begin{smallmatrix}
	A&-F\\ -Q&-A^\top 
\end{smallmatrix}\right]$~\cite{bini2011numerical}. For large $n$, the latter tasks are too expensive and the treatable scenarios are those where one can exploit additional structure in the coefficients, e.g. sparsity in $A$, and $F,Q$ of low-rank. See \cite{benner20} for an overview, and comparison of the large-scale methods for CAREs. 

In the first part of this work we  study the theoretical properties and the computation of the stabilizing solution of \eqref{eq:care}, in the case where the matrix coefficients have additional rank structures. More specifically, we consider the case when $n$ is possibly large, and $A,F,Q$ are \emph{quasiseparable matrices}~\cite{vandebril08}, i.e., all their offdiagonal blocks are low-rank. A particular case of this scenario is when all these matrices are banded. We remark that, we will often assume that $F$ is full rank, and this is a different situation than the usual setting of linear-quadratic optimal control, where the ranks of $F$, and $Q$ are small as they correspond to the number of inputs and outputs of the system, respectively. Within this setting, we provide decay bounds for the singular values of the offdiagonal blocks of $X$, justifying the approximability of $X$ by a quasiseparable matrix. To derive these bounds we relate a generic offdiagonal block with a Sylvester equation with low-rank right-hand side, and then exploit singular values upper bounds for the solution of this kind of equations~\cite{beckermann17}. Our results establish a link between the rate of decay of the offdiagonal singular values and certain rational approximation \upd{problems}, known as Zolotarev problems~\cite{beckermann17}, involving the set $\mathcal W(L^{-1}AL)$, where $F=LL^\top$ is a Cholesky factorization of $F$. To the best of our knowledge these are the first theoretical results showing the numerical quasiseparability of the CARE solution, in case of a full rank quadratic coefficient. In addition, in our Theorem~\ref{thm:quasi-zol} we improve the upper bound in \cite[Theorem 2.7]{massei18} that concerns the numerical quasiseparability of the solution of Sylvester equations with quasiseparable coefficients. As a byproduct of this analysis, we improve and enlarge the scope of existing upper bounds for the \emph{tensor train ranks} (TT ranks) of the value function $V(\mathbf y):=\mathbf y^\upd{\top} X\mathbf y$ associated with the solution of the CARE~\cite[Theorem 3.1]{dolgov21}. 

From the algorithmic view point, we propose two fast Riccati solvers: Algorithm~\ref{alg:dac_care} that applies to CAREs with quasiseparable coefficients, and Algorithm~\ref{alg_NK} that is specific to the banded case. The former method exploits the representation of $A,F,$ and $Q$, in the \emph{\upd{hierarchically} semiseparable format} (HSS)~\cite{xia10}, and is based on a divide-and-conquer scheme, similarly to other recent solvers for matrix equations with hierarchically low-rank coefficients~\cite{kressner19,kressner20}. Algorithms that provide an approximation of the CARE solution in a hierarchical matrix format were also proposed in \cite{grasedyck08,grasedyck03}. The method for the banded coefficients case aims at providing a sparse approximation of $X$, by means of an inexact Newton-Kleinman iteration (NK)~\cite{kleinman1968iterative} combined with a thresholding mechanism that keeps under control the level of sparsity of the iterates. Specifically, each Newton step is executed inexactly by conducting a limited number of GMRES iterations. Subsequently, using a greedy approach, we truncate diagonals while ensuring that the residual norm remains below a desired threshold. We establish conditions on the number of GMRES iterations and the truncation strategy that are sufficient to guarantee the convergence of Algorithm~\ref{alg_NK} to the stabilizing solution; the latter are built on existing results about the convergence of inexact NK iterations, i.e., \cite[Theorem 4.3]{feitzinger2009inexact}, and \cite[Theorem 10]{benner2016inexact}.  Under reasonable assumptions, both Algorithm~\ref{alg:dac_care}, and Algorithm~\ref{alg_NK}, have at most a linear-logarithmic complexity; when the NK iterate remains sufficiently banded, the second procedure provides a significant speed-up thanks to the use of sparse arithmetic.
We remark that an algorithm that provides banded approximants of $X$  was also proposed in \cite{haber18}; the latter procedure is based on an inexact NK iteration where the Newton step is rephrased as a Lyapunov equation, and approximated  with a truncated Taylor expansion of the integral representation of the solution. 

The last contribution of the paper is the application of the proposed methods to infinite horizon optimal control problems with quasiseparable structure, that appear in the context of optimal control for partial differential equations where the control operates over the entire spatial domain, and the objective is to minimize a cost related to the solution over the same domain. Another relevant example is the control of agent-based models; there, the dynamics pertain to the velocity and acceleration of the agents, with each agent possessing its own control, aimed at minimizing a global quantity, such as achieving consensus ~\cite{cucker2007emergent,bailo2018optimal}.  The numerical solution of large-scale instances of these problems is not doable with current state-of-the-art Riccati solvers, such as RADI~\cite{radi18}, and the rational Krylov subspace method~\cite{simoncini16b}, as they rely on the existence of accurate low-rank approximants. The next section provides more details about this application. 
\subsection{Infinite horizon problems via State-Dependent Riccati equations}\label{sec:intro-appl}

The infinite horizon problem, along with the associated synthesis of the feedback law, involves a dynamical system, which we assume to be in the following control affine form:
\begin{equation*}\label{eq}
	\left\{ \begin{array}{l}
		\dot{y}(s)=f(y(s)) + B(y(s))u(s), \;\; s\in(0,+\infty),\\
		y(0)=x\in\mathbb{R}^n.
	\end{array} \right.
\end{equation*}
Here, $y:[0,+\infty)\rightarrow\R^n$ denotes the state of the system, $u:[0,+\infty)\rightarrow\R^m$ represents the control signal,  $\mathcal{U}=L^\infty ([0,+\infty);U)$ is the set of admissible controls, with $U\subset \R^m$, $f:\R^n\rightarrow\R^n$ defines the system dynamics, and $B$ is a matrix-valued function of the state.

\noindent The goal is to determine a control that minimizes the \emph{infinite horizon cost functional} 
\begin{equation*}
	J_{\infty}(u;x) := \int\limits_0^{+\infty} y(t)^\top Q y(t) + u(t)^\top R u(t)\, dt \,,
	\label{quadratic_cost}
\end{equation*}
where $Q \in \mathbb{R}^{n\times n}$ is symmetric positive semidefinite and $ R \in \mathbb{R}^{m\times m}$ symmetric positive definite. Moreover, we want the optimal control to be given in feedback form, i.e., a control signal that depends only on the current state of the system. One approach involves solving the Hamilton-Jacobi-Bellman (HJB) equation, a first-order nonlinear PDE defined over $\R^n$~\cite{bardi1997optimal}. However, the fully non linear and, possibly, high dimensional nature of HJB equations can pose significant computational challenges making this approach not always viable in a large-scale scenario. 

Instead of directly addressing the HJB equations, we will use a faster but suboptimal alternative: the State-Dependent Riccati Equation (SDRE) \cite{ccimen2008state}.
This method involves iteratively solving linear-quadratic control problems arising from the gradual linearization of dynamics along a trajectory; this means that locally the system dynamics is expressed in semilinear form as:
\begin{align*}
	\dot{y}(t) & = A(y(t)) y(t) +B(y(t)) u(t) \\
	y(0) & = x\,.
	\label{semilinear}
\end{align*}

In the particular case where all the matrix coefficients are constant in the state, i.e.,  $A(y(t)) = A \in \mathbb{R}^{n\times n}$ and  $B(y(t)) = B \in \mathbb{R}^{n\times m}$,  we encounter the Linear Quadratic Regulator (LQR) problem; if the pair $(A,B)$ is stabilizable and the pair $(A,Q^{1/2})$ is detectable, the optimal feedback control for the LQR is given by the formula \cite{anderson2007optimal}:
\begin{equation*}
	u(y) = -R^{-1} B^{\top} Xy,
	\label{control_SDRE}
\end{equation*}
where $X\in\mathbb{R}^{n\times n}$ is the unique positive definite solution of the CARE
\begin{align*}
	A^\top X + X A
	-X BR^{-1}B^\top X +Q = 0\ .
\end{align*}
The SDRE technique extends this approach by introducing state dependence, resulting in:
\begin{equation}
	u(y) = -R^{-1} B^{\top}(y) X(y)y\,,
	\label{control_sdre_feed}
\end{equation}
where $X(y)$ now solves a State-Dependent Riccati Equation (SDRE)
\begin{align}
	A^\top(y) X(y) + X(y) A(y)-X(y)B(y)R^{-1}B^\top(y)X(y)+Q & = 0,
	\label{sdre}
\end{align}
and $A(y),B(y)$ are fixed at the state $y$.
This procedure iterates along the trajectory, solving \eqref{sdre} sequentially as the state $y(t)$ evolves over time.
Under appropriate stability assumptions, it can be shown that the closed-loop dynamics generated by the feedback law \eqref{control_sdre_feed} are locally asymptotically stable (for further details and the exact statement, see \cite{Banks_Lewis_Tran_2007}).

In this work, we aim at lowering the computational effort of the SDRE approach when solving  optimal control problems where the matrices $A(y),B(y),R$, and $Q$ are quasiseparable, and possibly full-rank, for all states $y$.


%
%
%
%
%

\subsection{Synopsis and notation}
The remainder of the paper is organized as follows. In Section~\ref{sec:quasi-riccati} we recall some
results on singular values decay for linear matrix equations and we present new decay bounds
for the offdiagonal singular values of the solution of CAREs with quasiseparable coefficients. Section~\ref{sec:tt} shows how our results also provide tighter estimates for the TT ranks of value functions. In Section~\ref{sec:quasi-care-solve} we propose and test a divide-and-conquer method (Algorithm~\ref{alg:dac_care}) for CAREs with quasiseparable coefficients. Section~\ref{sec:banded} focuses on the banded coefficients case. A truncated inexact Newton-Kleinman iteration (Algorithm~\ref{alg_NK}) is proposed and analysed in Section~\ref{sec:tink}, and some numerical tests, including a comparison with the divide-and-conquer solver, are reported in Section~\ref{sec:tink-tests}. In Section~\ref{sec:feedback_appl} we present numerical results concerning infinite horizon optimal control problems, and finally, in Section~\ref{sec:conclusion}
we draw some conclusions.

Sometimes we use Matlab-like notation to denote the submatrices of a certain matrix $M$, e.g., we write $M(I, J)$ for the submatrix corresponding to the subsets $I,J$ of row and column indices, and $M(i_{\mathrm{start}}:i_{\mathrm{end}},j_{\mathrm{start}}:j_{\mathrm{end}})$ to indicate submatrices corresponding to contiguous index sets.
Given a square matrix $M$, we denote by $\kappa(M)$ its 2-norm condition number and, if $M$ is diagonalizable, we indicate with $\kappa_{\mathrm{eig}}(M)$ the 2-norm condition number of an eigenvector matrix of $M$. For a symmetric matrix $M$ the symbols $\lambda_{\min}(M)$ and $\lambda_{\max}(M)$ refer to the minimum and maximum eigenvalues of $M$, respectively. Finally, we use $\Re(z)$, and $\Im(z)$ to denote the real and imaginary part of $z\in\mathbb C$.

\section{Riccati equation with quasiseparable coefficients}\label{sec:quasi-riccati}
In many applications, the matrices $A,F$, and $Q$ in \eqref{eq:care} are banded or have offdiagonal blocks with low-rank. Herein, we analyze when $X$ inherits such a structure from the coefficients of the CARE. More precisely, throughout this section we assume that the matrices $A,F,Q$ are \emph{quasiseparable}; let us recall the latter notion and some of its basic properties.
\begin{definition}
Let $M\in\mathbb R^{n\times n}$ and indicate with 
$$
\mathrm{Off}(M):=\{M(s+1:n,1:s):\ s=1,\dots n-1\}\ \cup\ \{M(1:s,s+1:n):\ s=1,\dots n-1\}, 
$$ 
the set of maximal submatrices of $M$ that are contained in the strictly lower or upper triangular  part of $M$. 
The matrix $M$ is said to be \emph{quasiseparable} of order $r$ (or quasiseparable of rank $r$) if $$\max_{\widetilde M\in\mathrm{Off(M)}}\rank(\widetilde M)=r,$$
and we write $\qsrank(M)=r$.
\end{definition} 

We remark that any $n\times n$ matrix is quasiseparable of order at most $\lfloor\frac{n}{2}\rfloor$, but quasiseparability is of interest, from the computational point of view, only when the order is significantly lower than $n$. There are many properties describing the behaviour of the $\qsrank$ under matrix operations; we recall the ones that will be used in this section and we refer to \cite{vandebril08} for their proofs.
\begin{itemize}
	\item The quasiseparable order is subadditive with respect to the matrix sum and product, i.e., $\qsrank(M_1+M_2)\le \qsrank(M_1)+\qsrank(M_2)$ and $\qsrank(M_1\cdot M_2)\le \qsrank(M_1)+\qsrank(M_2)$. Actually, a slightly stronger property holds; if we define the lower and upper quasiseparable orders as the maximum rank among the submatrices in the lower and upper triangular parts, respectively, then the lower and upper quasiseparable orders are subadditive with respect to the matrix sum and product.
	\item If the matrix $M$ is invertible, then $\qsrank(M^{-1})=\qsrank(M)$.
		\item If $M$ is a real symmetric positive definite matrix and $L$ is its Cholesky factor ($M=LL^\top$), then $\qsrank(M)=\qsrank(L)$.
\end{itemize}
In Section~\ref{sec:decay-care} we will show that when $A,F,Q$ are quasiseparable with a low order, the singular values of the matrices in $\mathrm{Off}(X)$ exhibit a rapid decay, and this ensures that $X$ is well approximated with a quasiseparable matrix of low rank. The theoretical argument, used in the proof of the latter property, is based on some results about the singular values of the solutions of Sylvester equations that we are going to recall in the next section.
\subsection{Singular values decay for linear matrix equations}\label{sec:decay-sylv}
Linear matrix equations of the form
\begin{equation}\label{eq:sylv}
AX+XB= Q, \qquad A,B,X,Q\in\mathbb C^{n\times n}
\end{equation}
are known in the literature as \emph{Sylvester equations}, and have been widely studied in view of their relevant role in many applications, \upd{see e.g.} \cite{simoncini16}. In the case where $Q$ is low-rank and the spectra of $A$ and $-B$ are disjoint, it is possible to show that the singular values of the solution $X$ of \eqref{eq:sylv}, decay rapidly towards zero. The rate of the latter decay can be estimated with the optimal value of the following rational approximation problem
$$
Z_h(E,F) := \min_{g(z)\in\mathcal R_{h,h}}\frac{\max_{z\in E} |g(z)|}{\min_{z\in F} |g(z)|}, 
$$
where $\mathcal R_{h,h}$ denotes the set of rational functions with both numerator and denominator of degree at most $h$, and $E,F$ are disjoint subsets of the complex plane containing either the numerical ranges or the spectra of $A$ and $-B$, respectively. More precisely, the following result links $Z_h(E, F)$ with the low-rank approximability of $X$.
\begin{theorem}[Theorem 2.1 in \cite{beckermann17}]\label{thm:zol}
Let $X$ be the solution of $AX+XB=Q$, with $Q$ of rank $t$, and let 
	$E$ and $F$ be disjoint compact sets in the complex plane.
	\begin{itemize}
		\item[$(i)$] If $E,F$ contain
		the numerical ranges of $A$ and $-B$, respectively, then
		\[
		\frac{\sigma_{th+1}(X)}{\norm{ X}_2}\leq  K_C\,\cdot Z_h(E,F),\qquad h=1,2,\dots,
		\]
		where $K_C = 1$ if $A,B$ are normal matrices and $1\le K_C \le (1+\sqrt{2})^2$ otherwise.
		\item[$(ii)$] If $A,B$ are diagonalizable and $E,F$ contain
		the spectra of $A$ and $-B$, respectively, then
		\[
		\frac{\sigma_{th+1}(X)}{\norm{ X}_2}\leq  \kappa_{\mathsf{eig}}(A)\cdot \kappa_{\mathsf{eig}}(B)\,\cdot Z_h(E,F),\qquad h=1,2,\dots\ .
		\]
	\end{itemize}
\end{theorem}

The quantity  $Z_h(E,F)$ is called the $h$-th \emph{Zolotarev number} for the sets $E$ and $F$. When the latter are well separated, e.g., separated by a line, $Z_h (E, F )$
decreases rapidly as $h$ increases. Exponentially decaying bounds for  $Z_h (E, F )$ have been established for various
configurations of $E$ and $F$, including the case of real intervals and discs, see \cite{beckermann17} and \cite[Section 2]{kressner23}.

Theorem~\ref{thm:zol} has been generalized to the setting $A,B,$ and $Q$ quasiseparable, by showing that, in the latter case, any offdiagonal block of $X$ can be written as the solution of a Sylvester equation with low-rank right-hand side \cite[Theorem 2.7]{massei18}. This implies that $X$ is numerically quasiseparable and its offdiagonal singular values decay accordingly to certain Zolotarev numbers. By following a different argument, similar to the one used in the proof of Theorem~\ref{thm:zol}, we can improve on the bound given in \cite[Theorem 2.7]{massei18}.
\begin{theorem}\label{thm:quasi-zol}
Let $X$ be the solution of $AX+XB=Q$, with $A,B$, and $Q$ quasiseparable matrices of orders  $r_a,r_b$, and $r_q$, respectively, and let 
$E$ and $F$ be disjoint compact sets in the complex plane. 
\begin{itemize}
	\item[$(i)$] If $E,F$ contain
	the numerical ranges of $A$ and $-B$, respectively, then the singular values $\sigma_i(M)$ of any $M\in\mathrm{Off}(X)$ satisfy
	\[
	\frac{\sigma_{(r_a+r_b)h+r_q+1}(M)}{\norm{ X}_2}\leq  K_C\,\cdot Z_h(E,F),\qquad h=1,2,\dots, 
	\]
	where $K_C = 1$ if $A,B$ are normal matrices and $1\le K_C \le (1+\sqrt{2})^2$ otherwise.
	\item[$(ii)$] If $A,B$ are diagonalizable and $E,F$ contain
	the spectra of $A$ and $-B$, respectively, then the singular values $\sigma_i(M)$ of any $M\in\mathrm{Off}(X)$ satisfy
	\[
	\frac{\sigma_{(r_a+r_b)h+r_q+1}(M)}{\norm{ X}_2}\leq  \kappa_{\mathsf{eig}}(A)\cdot \kappa_{\mathsf{eig}}(B)\,\cdot Z_h(E,F),\qquad h=1,2,\dots\ .
	\]
\end{itemize}	
\end{theorem}
\begin{proof}
By following the same steps of the proof of \cite[Theorem 2.1]{beckermann17}, we have that for any pair of polynomials $p(z),q(z)$ of degree at most $h$, it holds
\begin{align*}
p(A)Xq(B)-q(A)Xp(B)&=\sum_{i,j=0}^{h-1} c_{ij}A^iQB^j\\
\Rightarrow X - g(A)Xg(B)^{-1}&=\sum_{i,j=0}^{h-1} c_{ij}\frac{A^i}{p(A)}Q\frac{B^j}{q(B)}\\
&=\sum_{i,j=0}^{h-1} d_{ij}(p_i I -A)^{-1}Q(q_jI-B)^{-1}=:Y,
\end{align*}
where $p_i$ and $q_j$ are the roots of the polynomials $p(z)$, and $q(z)$, respectively, $g(z)=q(z)/p(z)$, and $d_{ij}$ are some scalar coefficients. Each offdiagonal block of $Y$ is a linear combination of the corresponding offdiagonal blocks in $(p_i I -A)^{-1}Q(q_jI-B)^{-1}$, for $i,j=0,\dots, h-1$. In particular, by considering the splittings
$$
Y=\begin{bmatrix}
	Y_{11}& Y_{12}\\
	Y_{21}&Y_{22}
\end{bmatrix},\quad 
Q=\begin{bmatrix}
	Q_{11}& Q_{12}\\
	Q_{21}&Q_{22}
\end{bmatrix},\quad (p_i I - A)^{-1}=\begin{bmatrix}
A^{(i)}_{11}& A^{(i)}_{12}\\
A^{(i)}_{21}&A^{(i)}_{22}
\end{bmatrix},\quad (q_jI-B)^{-1}=\begin{bmatrix}
B^{(j)}_{11}& B^{(j)}_{12}\\
B^{(j)}_{21}&B^{(j)}_{22}
\end{bmatrix},
$$ 
we have that
$$
Q(q_jI-B)^{-1}=\begin{bmatrix}
	\star &\star\\
Q_{21}B^{(j)}_{11}+ Q_{22}B^{(j)}_{21}	&\star
\end{bmatrix},
$$
implying that, for $j=0,\dots, h-1$, the $(2,1)$ block of $Q(q_jI-B)^{-1}$ belongs to the range of at most $r_q+h\cdot r_b$ columns. Combining the latter property with the expression
$$
(p_iI -A)^{-1}Q(q_jI-B)^{-1}=\begin{bmatrix}
	\star &\star\\
	A^{(i)}_{21}(Q_{11}B^{(j)}_{11}+ Q_{12}B^{(j)}_{21}) +A^{(i)}_{22}(Q_{21}B^{(j)}_{11}+ Q_{22}B^{(j)}_{21})	&\star
\end{bmatrix},
$$
we see that $Y_{21}$ has rank at most $h(r_a+r_b)+r_q$; since an analogous argument applies to $Y_{12}$, we have that $Y$ is quasiseparable of order at most $h(r_a+r_b)+r_q$. The claim follows by applying an Eckart-Young argument to the offdiagonal blocks of $X$, that yields $\sigma_{h(r_a+r_b)+r_q+1}(M)\le\norm{X-Y}_2\le \norm{g(A)}_2\norm{X}_2\norm{g(B)^{-1}}_2$, and minimizing the right-hand side with respect to the choice of the polynomials $p(z),q(z)$.
\end{proof}
Theorem~\ref{thm:quasi-zol} ensures a faster decay rate  for the offdiagonal singular values than \cite[Theorem 2.7]{massei18}; indeed, the latter result uses the same quantity to bound $\sigma_{(r_a+r_b+r_q)h+1}(M)$ (in place of $\sigma_{(r_a+r_b)h+r_q+1}(M)$). Moreover, Theorem~\ref{thm:quasi-zol} applies to the case of diagonalizable coefficients with disjoint spectra; to remove the assumptions on the numerical range of $A$ and $B$ in \cite[Theorem 2.7]{massei18} one should assume additional spectral properties for their principal submatrices.

In the next sections we will leverage Theorem~\ref{thm:zol}, and Theorem~\ref{thm:quasi-zol}, to link the singular values of the offdiagonal submatrices of a CARE solution to certain Zolotarev numbers.
\subsection{Offdiagonal singular values decay for Riccati equations}\label{sec:decay-care}
Let us assume that $A,F,Q$ are real quasiseparable matrices with orders $r_a, r_f$ and $r_q$, respectively,  $A$ be such that $\mathcal W(A)\subset\mathbb C^{-}$, and $F,Q$ symmetric and positive semidefinite. Then,  for a given $k\in\{1,\dots,n-1\}$ we split the coefficients as
$$
A=\underbrace{\left[\begin{smallmatrix}
		A_{11}\\ &A_{22}
\end{smallmatrix}\right]}_{A_0}+\underbrace{\left[\begin{smallmatrix}
		&A_{12}\\ A_{21}
\end{smallmatrix}\right]}_{\delta A}, \quad F=\underbrace{\left[\begin{smallmatrix}
		F_{11}\\ &F_{22}
\end{smallmatrix}\right]}_{F_0}+\underbrace{\left[\begin{smallmatrix}
		&F_{12}\\ F_{21}
\end{smallmatrix}\right]}_{\delta F},\quad Q=\underbrace{\left[\begin{smallmatrix}
		Q_{11}\\ &Q_{22}
\end{smallmatrix}\right]}_{Q_0}+\underbrace{\left[\begin{smallmatrix}
		&Q_{12}\\ Q_{21}
\end{smallmatrix}\right]}_{\delta Q},
$$
where the square diagonal blocks $A_{11},F_{11},Q_{11}$ have size $k\times k$ and $A_{22},F_{22},Q_{22}\in\mathbb R^{(n-k)\times (n-k)}$. We remark that, in this setting, the CARE that we want to solve, reads as
\begin{equation}\label{eq:riccati}
	(A_0+\delta A)^\top X+X(A_0 +\delta A)-X(F_0+\delta F)X+Q_0+\delta Q=0,
\end{equation}
where the matrices containing the offdiagonal blocks have rank bounded by $2$ times the quasiseparable ranks of the corresponding coefficient:
$$
\mathrm{rk}(\delta A)\le 2r_a,\qquad \mathrm{rk}(\delta F)\le 2r_f,\qquad \mathrm{rk}(\delta Q)\le 2r_q.
$$
Similarly to \cite{kressner20}, we study the influence of these low-rank parts of the coefficients on the solution; more precisely we compare $X$ with the symmetric and positive semidefinite solution $X_0$  of the Riccati equation associated with the block diagonal part of the coefficients:
\begin{equation}\label{eq:riccati_diag}
	A_0^\top X_0+X_0A_0-X_0F_0X_0+Q_0=0.
\end{equation}
Note that, as $\mathcal W(A_0)\subseteq \mathcal W(A)$, and $F_0,Q_0$ are symmetric positive semidefinite, there is a unique symmetric positive semidefinite stabilizing solution $X_0$ of \eqref{eq:riccati_diag}. Moreover, $X_0$ inherits the block diagonal structure of  the coefficients in \eqref{eq:riccati_diag} and can be computed by solving two Riccati equations, one of size $k\times k$ and the other $(n-k)\times (n-k)$. Therefore, showing that the singular values of $\delta X:= X-X_0$ decay rapidly  implies that the offdiagonal blocks $X(k+1:n,1:k)$ and $X(1:k, k+1:n)$ are numerically low-rank.
With this goal in mind, we subtract \eqref{eq:riccati_diag} from \eqref{eq:riccati} and we find that $\delta X$ solves the CARE:
\begin{equation}\label{eq:riccati_correction}
	(A^\top-X_0F)\delta X + \delta X(A-FX_0)-\delta X F \delta X=-\delta Q-\delta A^\top X_0-X_0\delta A +X_0\delta F X_0.
\end{equation}
By applying simple algebraic manipulations to \eqref{eq:riccati_correction} we get
\begin{equation}\label{eq:lyapunov_correction}
	(A_0^\top-X_0F_0)\delta X +\delta X(A-FX)= -\delta Q -\delta A^\top X-X_0\delta A+ X_0\delta FX.
\end{equation}

Equation \eqref{eq:lyapunov_correction} says that $\delta X$ is the solution of a Sylvester equation with a low-rank right-hand side; indeed, even if $\widetilde Q:= -\delta Q -\delta A^\top X-X_0\delta A+ X_0\delta FX$  depends on $\delta X$, we can a priori claim: 
$$
\rank\left(\widetilde Q\right)\leq  4r_a+2r_f+2r_q,
$$
in view of the inequality $\rank(M_1\cdot M_2)\le \min\{\rank(M_1), \rank(M_2)\}$. 
Moreover, the linear coefficients of \eqref{eq:lyapunov_correction} are the closed-loop matrices associated with \eqref{eq:riccati} and \eqref{eq:riccati_diag}, respectively. The latter have  eigenvalues with non positive real parts as $X$, and $X_0$ are stabilizing solutions.  Under the additional assumption that the numerical ranges  of the closed-loop matrices remain in the left part of the complex plane,  we can apply Theorem~\ref{thm:zol} to the CARE setting.
\begin{theorem}\label{thm:sing-decay}
	Let $H\in\mathbb R^{2n\times 2n}$ be of the form
	$$
	H=\begin{bmatrix}
		A&-F\\
		-Q&-A^\top
	\end{bmatrix}
	$$
	where $A\in\mathbb R^{n\times n}$ is  quasiseparable of order $r_a$, $\mathcal W(A)\subset\mathbb C^{-}$, $F,Q\in\mathbb R^{n\times n}$ symmetric positive semidefinite  and quasiseparable of order $r_f$ and $r_q$, respectively, and let $X\in\mathbb R^{n\times n}$ be the symmetric positive semidefinite stabilizing solution of the continuous-time Riccati equation 
	$$
	A^\top X+XA-XFX+Q=0.
	$$
	Moreover, let $H^{(j)},H_{(j)}\in\mathbb R^{2j\times 2j}$ be defined as
	$$
	H^{(j)}=\begin{bmatrix}
		A^{(j)}&-F^{(j)}\\
		-Q^{(j)}&-(A^{(j)})^\top
	\end{bmatrix},\qquad H_{(j)}=\begin{bmatrix}
		A_{(j)}&-F_{(j)}\\
		-Q_{(j)}&-(A_{(j)})^\top
	\end{bmatrix},
	$$
	where $A^{(j)},F^{(j)},Q^{(j)}$  are the principal submatrices of $A,F,Q$ corresponding to the first $j$ indices of rows and columns, and $A_{(j)},F_{(j)},Q_{(j)}$ are those corresponding to the last $n-j$, for $j=1,\dots, n$. Finally, indicate with $X^{(j)}$ and $X_{(j)}$ the symmetric positive semidefinite stabilizing solutions of the Riccati equations associated with $H^{(j)}$ and $H_{(j)}$, respectively. 
	If $E\subset \mathbb C^-$ is such that
	$$
	\mathrm{convex\ hull}\left(\bigcup_{j=1,\dots, n}\mathcal W(A^{(j)} - F^{(j)}X^{(j)}) \quad \cup  \bigcup_{j=1,\dots,n}\mathcal W(A_{(j)} - F_{(j)}X_{(j)})\right)\subset E,
	$$
	then the singular values $\sigma_i(M)$ of any $M\in\mathrm{Off}(X)$ satisfy
	$$
	\frac{\sigma_{ht+1}(M)}{\norm{X}_2}\le (1+\sqrt 2)^2 Z_h(E,-E), \quad h=1,2,\dots, 
	$$
	where $t=4r_a+2r_f+2r_q$.
\end{theorem}
\begin{proof}
	We prove the statement for an offdiagonal block in the strictly lower triangular part, as the matrix $X$ is symmetric. 
	Let $j\in\{1,\dots, n-1\}$, $M=X(j+1:n, 1:j)$, and consider the splittings
	$$
	A= \underbrace{\begin{bmatrix}
			A^{(j)}\\ & A_{(j)}
	\end{bmatrix}}_{A_0}+\delta A, \quad F= \underbrace{\begin{bmatrix}
			F^{(j)}\\ & F_{(j)}
	\end{bmatrix}}_{F_0}+\delta F,\quad Q= \underbrace{\begin{bmatrix}
			Q^{(j)}\\ & Q_{(j)}
	\end{bmatrix}}_{Q_0}+\delta Q.
	$$
	Then the singular values of $M$ are  upper bounded by the ones of the matrix $\delta X$ that solves the Sylvester equation
	$$
	(A_0^\top-X_0F_0)\delta X +\delta X(A-FX)= -\delta Q -\delta A^\top X-X_0\delta A+ X_0\delta FX,$$
	where $X_0$ is the stabilizing solution of \eqref{eq:riccati_diag}.
	Observe that  $\mathcal W(A_0^\top-X_0F_0)=\mathcal W(A_0-F_0X_0)$ and the matrix $A_0-F_0X_0$ is block diagonal with blocks $A^{(j)}-F^{(j)}X^{(j)}$, and $A_{(j)}-F_{(j)}X_{(j)}$, so  $\mathcal W(A_0-F_0X_0)$ is given by the convex hull of the union of the numerical ranges of its diagonal blocks. Applying Theorem~\ref{thm:zol} to the matrix equation defining $\delta X$, yields the claim. 
\end{proof}

The bound in Theorem~\ref{thm:sing-decay} is implicit as it requires to have a control on the numerical range of the closed-loop matrices; see also the discussion in \cite[Section 2.1]{kressner20}.  
On the other hand, we are going to show that explicit bounds are attainable in the special case $F=I$, and that we can at least remove the dependency on the solution $X$ in the case of a generic quasiseparable symmetric positive definite $F$.
\begin{remark}\label{rem:improv1}
We point out that, by directly considering the $2\times 2$ block partitioning of \eqref{eq:care}, we can rewrite the offdiagonal block $X_{21}$ of $X$ as the solution of the Sylvester equations
$$
(A_{22}^\top-X_{22}F_{22})X_{21}+ X_{21}(A_{11}-F_{11}X_{11})=-Q_{21}-A_{12}^\top X_{11}-X_{22}A_{21}+X_{22}F_{21}X_{11}+X_{21}F_{12}X_{21},
$$
whose right-hand side has rank bounded by $2r_a+2r_f+r_q$. This has the potential to improve the bound in Theorem~\ref{thm:sing-decay}, providing a lower value for the parameter $t$. However, pursuing this direction leads to the difficult task of controlling the spectral properties of the matrices $A_{22}^\top-X_{22}F_{22}$ and $A_{11}-F_{11}X_{11}$, that can not be seen as closed-loop matrices.
\end{remark}
\begin{remark}\label{rem:improv2}
Another possibility to reduce the parameter $t$ in Theorem~\ref{thm:sing-decay}, is to consider different block diagonal/low-rank splittings for the symmetric coefficients in \eqref{eq:care}. For instance, if the $(2,1)$ sub-block of $F$ is given by the rank $r_f$ factorization $UV^\top$, then we can write $F=\widetilde{F}_0+\widetilde{\delta F}$ with
$$
\widetilde{F}_0=\begin{bmatrix}
F^{(j)}+VV^\top\\ &F_{(j)}+ UU^\top
\end{bmatrix},\qquad
\widetilde{\delta F}=\begin{bmatrix}
-VV^\top&VU^\top\\
UV^\top& -UU^\top
\end{bmatrix}=\begin{bmatrix}
-V\\ U
\end{bmatrix}\begin{bmatrix}
V\\-U
\end{bmatrix}^\top.
$$
In particular, $\rank(\widetilde{\delta F})\le r_f$ (while $\rank(\delta F)\le 2r_f$), and $\widetilde F_0$ is again block diagonal and symmetric positive definite. Analogous splittings can be applied to $Q$, and to $A$ when it is symmetric negative definite. However, apart from maintaining the definiteness, it is difficult to control other spectral properties of the block diagonal parts obtained with this kind of splitting.  
\end{remark}	
\subsubsection{Case $F=I$}
When $F=I$, the closed-loop matrices $A_0-X_0$ and  $A-X$ are the sum of a symmetric negative semidefinite matrix and another one with numerical range in the open left complex plane.

We begin by discussing  the particular case $A$ symmetric negative definite where   Theorem~\ref{thm:zol} allows us to replace the 
constant $(1+\sqrt 2)^2$ with $1$, in the upper bound of Theorem~\ref{thm:sing-decay}. Additionally, we are able to state a bound that only involves the  knowledge of the spectral interval of $A$ and of the norm of $Q$.

\begin{corollary}\label{cor:sing-decay1}
	Under the same assumptions of Theorem~\ref{thm:sing-decay}, apart from $A\in\mathbb R^{n\times n}$ symmetric negative definite of quasiseparable order $r_A$ with spectrum contained in $[-b,-a]$, and $F=I$,  the singular values of any $M\in\mathrm{Off}(X)$ satisfy
	$$
	\frac{\sigma_{ht+1}(M)}{\norm{X}_2}\le  Z_h(E,-E) ,
	\quad h=1,2,\dots,
	$$
	where $t=4r_a+2r_q$, and $E=\left[\,-\sqrt{b^2 + \norm{Q}_2}\, ,\,  -a \right]$.
\end{corollary}

\begin{proof}
	As previously observed, within this setting all the considered closed-loop matrices are symmetric and negative definite, so it is sufficient to show that their spectra is contained in $E$. The latter is equivalent to show that the Hamiltonian matrices $H^{(j)}$, $H_{(j)}$ have spectra contained in $E\cup -E$, for $j=1,\dots, n$. Let us begin by showing this property for $H$, i.e., when $j=n$. In this case the characteristic polynomial of $H$ is given by
	$$
	\det(\lambda I -H)= \det\left(\begin{bmatrix}
		\lambda I - A&Q\\
		I &\lambda I +A
	\end{bmatrix}\right) = \det(\lambda^2 I -A^2-Q),
	$$
	where the last equality is due to the fact that the lower left block of the $2\times 2$ block matrix commutes with the lower right block \cite{silvester2000determinants}. In particular the eigenvalues of the stabilizing solution corresponds to the negative square root of those of $A^2 + Q$ that, in view of the positive definiteness of $A^2$ and $Q$, belong to the interval $[a^2 , b^2 + \norm{Q}_2]$. For $H^{(j)}$ with $j=1,\dots, n-1$, we have that $\det (\lambda I-H^{(j)})= \det(\lambda I- [A^{(j)}]^2-Q^{(j)})$ and thanks to the interlacing properties of symmetric definite matrices we get that also the eigenvalues of 
	$[A^{(j)}]^2+Q^{(j)}$ are contained in $[a^2 , b^2 + \norm{Q}_2]$. The same argument applies to the Hamiltonian matrices $H_{(j)}$ and this concludes the proof.
\end{proof}
\begin{remark}\label{rem:rational}
 When $A$ is symmetric, $F=I$, and $Q$ is symmetric positive semidefinite, the solution of \eqref{eq:care} admits the following explicit expression:
 \begin{equation}\label{eq:caresol}
 	X = \sqrt{A^2+Q}+A.	
 \end{equation} 
Thus, one can use the best rational approximation error of the square root function to get upper bounds for the offdiagonal singular values of $X$. More precisely\upd{,} if $g(z)\in\mathcal R_{h,h}$ is a rational approximation of the square root, then $\qsrank(g(A^2+Q)+A)\le h(2r_a+r_q)+r_a$, and  
\begin{equation}\label{eq:rational}
\sigma_{h(2r_a+r_q)+r_a+1}(M)\le \min_{g(z)\in\mathcal R_{h,h}}\norm{X-g(A^2+Q)-A}_2\le \min_{g(z)\in\mathcal R_{h,h}}\max_{z\in[a^2, b^2+\norm{Q}_2]}|\sqrt{z}-r(z)|,
\end{equation}
for any $M\in \mathrm{Off}(X)$. In turn, one can \upd{combine} this argument with upper bounds for the rational approximation error, e.g. \cite[Equation (27)]{gawlik}, to get computable quantities in the right-hand side of \eqref{eq:rational}. Such bounds are potentially sharper than the one in Corollary~\ref{cor:sing-decay1}, see also Example~\ref{ex:decay}; however, it is not possible to extend this approach to the case where $A$ is non symmetric.
\end{remark}

Let us generalize Corollary~\ref{cor:sing-decay1} to the case of a possibly non symmetric coefficient $A$ with numerical range in $\mathbb C^{-}$. In that situation, the closed-loop matrices $A^{(j)}-X^{(j)},A_{(j)}-X_{(j)}$  are not ensured to be normal but we can control their numerical ranges with the inclusion $\mathcal W(A^{(j)}-X^{(j)})\subset \mathcal W(A^{(j)})+\mathcal W(-X^{(j)})$. To do that, we still need  an auxiliary result that concerns the norm of the stabilizing solution of a continuous-time Riccati equation with positive definite quadratic coefficient.
\begin{lemma}\label{lem:sol-norm}
	Let $A\in\mathbb R^{n\times n}$ have all eigenvalues with negative real part, $F\in\mathbb R^{n\times n}$ be symmetric positive definite and $Q\in\mathbb R^{n\times n}$ be symmetric positive semidefinite. Then, the stabilizing solution $X$ of $A^\top X+XA-XFX+Q=0$ satisfies:
	\begin{equation}\label{eq:norm-sol}
		\norm{X}_2\leq \frac{\tau +\sqrt{\tau^2+\norm{F}_2\norm{Q}_2}}{\lambda_{\min}(F)},
	\end{equation}
	where $\tau:=\max \{|x|:\ x\in\mathcal W(A) \cap \mathbb R\}$.
\end{lemma}
\begin{proof}
	Since $X$ is real and symmetric positive semidefinite, for any eigenvalue $\lambda\in\mathbb R$ of $X$ we can consider a real eigenvector $v$ of unit norm. Then, we have
	\begin{align*}
		&\quad v^\top(A^\top X+XA-XFX+Q)v=0\\
		\Leftrightarrow&\quad \lambda^2 v^\top Fv-2\lambda v^\top Av - v^\top Qv=0\\
		\Leftrightarrow&\quad \lambda=\frac{v^\top Av\pm \sqrt{(v^\top Av)^2+ (v^\top Fv)(v^\top Qv)}}{v^\top Fv}.
	\end{align*}
	To get \eqref{eq:norm-sol}, it is sufficient to note that $v^\top Av\in\mathcal W(A)\cap \mathbb R$ and derive an upper bound for $|\lambda|$ by means of the inequalities
	$$
	\lambda_{\min}(F)\le v^\top Fv\le \norm{F}_2,\qquad 0\le v^\top Qv\le \norm{Q}_2. 
	$$
\end{proof}
We are now ready to generalize Corollary~\ref{cor:sing-decay1} to the case of a non symmetric quasiseparable coefficient $A$.
\begin{corollary}\label{cor:sing-decay2}
	Under the same assumptions of Theorem~\ref{thm:sing-decay}, if $F=I$, then the singular values of any $M\in\mathrm{Off}(X)$ satisfy
	$$
	\frac{\sigma_{ht+1}(M)}{\norm{X}_2}\le (1+\sqrt{2})^2
	Z_h(E,-E)
	\quad h=1,2,\dots,
	$$
	where $t=4r_a+2r_q$, $E=\mathcal W(A) + \left[-\tau -\sqrt{\tau^2+\norm{Q}_2},\ 0\right]$, and $\tau:= |\min\{\mathcal W(A)\cap \mathbb R\}|$.
\end{corollary}
\begin{proof}
	To get the claim we just need to show that the numerical ranges of the closed-loop matrices $A^{(j)}-X^{(j)}$, and $A_{(j)}-X_{(j)}$ are contained in $E$ for any $j=1,\dots, n$, and apply Theorem~\ref{thm:sing-decay}.
	
	Note that in view of the relations $\mathcal W(A^{(j)})\subseteq\mathcal W(A)$, $\mathcal W(A_{(j)})\subseteq\mathcal W(A)$, $\norm{Q_{(j)}}_2\le \norm{Q}_2$, $\norm{Q_{(j)}}_2\le \norm{Q}_2$, the upper bound in \eqref{eq:norm-sol} applies also to the symmetric positive semidefinite matrices $X^{(j)}$ and $X_{(j)}$, i.e., $\norm{X^{(j)}}_2, \norm{X_{(j)}}_2\leq \tau +\sqrt{\tau^2+\norm{Q}_2}$. 
	This yields
	\begin{align*}
	\mathcal W(A^{(j)}-X^{(j)})&\subseteq \mathcal W(A^{(j)}) + \mathcal W(-X^{(j)})\subseteq \mathcal W(A) + \left[-\tau -\sqrt{\tau^2+\norm{Q}_2},\ 0\right]= E,\\
		\mathcal W(A_{(j)}-X_{(j)})&\subseteq \mathcal W(A_{(j)}) + \mathcal W(-X_{(j)})\subseteq \mathcal W(A) + \left[-\tau -\sqrt{\tau^2+\norm{Q}_2},\ 0\right]= E.
	\end{align*}
\end{proof}
\begin{remark}
	In the case $A$ symmetric negative definite, we can compare the bounds coming from  Corollary~\ref{cor:sing-decay1} and Corollary~\ref{cor:sing-decay2}. We see that  Corollary~\ref{cor:sing-decay2} is  looser than Corollary~\ref{cor:sing-decay1} as the set where the rational approximation problem is evaluated is  $[-2b-\sqrt{b^2+\norm{Q}_2}, -a]$ that is larger than  $[-\sqrt{b^2+\norm{Q}_2}, -a]$. 
\end{remark}
\begin{example}\label{ex:decay}
Let us run a couple of numerical tests to validate our theoretical bounds for the singular values of the offdiagonal blocks of the solution of a CARE. We consider two equations of the form \eqref{eq:care} with $F=I$, $Q$ equals to the tridiagonal symmetric positive definite 
$$
Q= \frac{T\cdot T^\top}{\norm{T\cdot T^\top}_2},\qquad T=\begin{bmatrix}
	1\\
	1&1\\
	&\ddots&\ddots\\
	&&1&1	
\end{bmatrix},
$$ and $A$ diagonal. Note that, under these assumptions, we can set $t=2$ in the inequalities given in Corollary~\ref{cor:sing-decay1} and Corollary~\ref{cor:sing-decay2}. In the first test, we take $A$ with diagonal entries logarithmically spaced in $[-1, -0.001]$; in the second test, the diagonal entries of $A$ are equally spaced on the complex circle of radius $1$ and center $-1.1$. For both experiments we consider matrices of size $n=500$, and we compute the solution of \eqref{eq:care} by means of the \texttt{icare} Matlab function; in the second test, the matrix $A^\top$ in \eqref{eq:care} is replaced by $A^H$ in view of the complex entries of $A$. Then, we calculate the \emph{offdiagonal singular values of $X$}, defined as
$$
\sigma_\ell^{\mathsf{off}}(X):=\max_{j=1,\dots,n-1}\sigma_\ell(X(j+1:n, 1:j)),
$$
by explicitly computing the SVDs of all the subdiagonal submatrices of $X$. Concerning the upper bounds, we see that in the first test we can set $E=[-\sqrt{2}, -0.001]$, while in the second experiment we consider $E=\{|z-c|\le r\}\supseteq \{|z-1.1|\le 1\}+[-2.1-\sqrt{2.1^2+1}, 0]$, with $c=\frac{4.3+\sqrt{2.1^2+1}}2$, and $r= \frac{4.1+\sqrt{2.1^2+1}}{2}$.
To estimate the Zolotarev numbers $Z_h(E,-E)$ we employ the upper bounds in \cite[Theorem 2.1]{beckermann17} when $A$ has real entries, and \cite[Theorem 1]{townsend18} in the complex circle case. Finally, in the real case, we numerically compute the upper bound in Remark~\ref{rem:rational}, by means of the \texttt{minimax} function of the \texttt{chebfun} toolbox~\cite{chebfun}, for estimating the best rational approximation error of the square root function on the interval $[10^{-6}, 2]$ .

The offdiagonal singular values, and the corresponding upper bounds are reported in Figure~\ref{fig:sing-decay}. We observe that our theoretical results manage to describe the superlinear decay of the offdiagonal
singular values, and the bound in Remark~\ref{rem:rational} is quite tight; however, there is a significant gap between the  decay rates predicted by Corollary~\ref{cor:sing-decay1}, Corollary~\ref{cor:sing-decay2}, and 
the actual behaviour of the sequence $\{\sigma_\ell^{\mathsf{off}}(X)\}_\ell$.	
\end{example}

	\begin{figure}
	\centering
	\begin{tikzpicture}
		\begin{semilogyaxis}[legend pos = north east, ymax=1e3, xlabel = $\ell$, ylabel=$\sigma_\ell^{\mathsf{off}}$, legend style={					
				font=\scriptsize, at={(1,1)}}, width=.48\textwidth]
			\addplot[mark=x, blue] table[x index = 0, y index = 1] {singdecay1.dat};
			\addplot[mark=o, red] table[x index = 0, y index = 2] {singdecay1.dat};
			\addplot[mark=square, cyan] table[x index = 0, y index = 1] {singdecay11.dat};
			\legend{$\sigma_\ell^{\mathsf{off}}(X)$, Bound from Corollary~\ref{cor:sing-decay1}, Bound from Remark~\ref{rem:rational}}
		\end{semilogyaxis}
	\end{tikzpicture}
	\begin{tikzpicture}
	\begin{semilogyaxis}[legend pos = north east, ymax=1e3, , xlabel = $\ell$, ylabel=$\sigma_\ell^{\mathsf{off}}$, legend style={					
			font=\scriptsize, at={(1,1)}}, width=.48\textwidth]
		\addplot[mark=x, blue] table[x index = 0, y index = 1] {singdecay2.dat};
		\addplot[mark=o, red] table[x index = 0, y index = 2] {singdecay2.dat};
		\legend{$\sigma_\ell^{\mathsf{off}}(X)$, Bound from Corollary~\ref{cor:sing-decay2}}
	\end{semilogyaxis}
\end{tikzpicture}
	\caption{Offdiagonal singular values in the solution $X\in\mathbb R^{500\times 500}$ of \eqref{eq:care}, with $A$ diagonal, $F=I$, and $Q$ tridiagonal symmetric positive definite. On the left, the matrix $A$ has diagonal entries logarithmically spaced in $[-1, -0.001]$; on the right, the diagonal entries of $A$ are equally spaced on the circle of radius $1$ and center $-1.1$.  }\label{fig:sing-decay}
\end{figure}
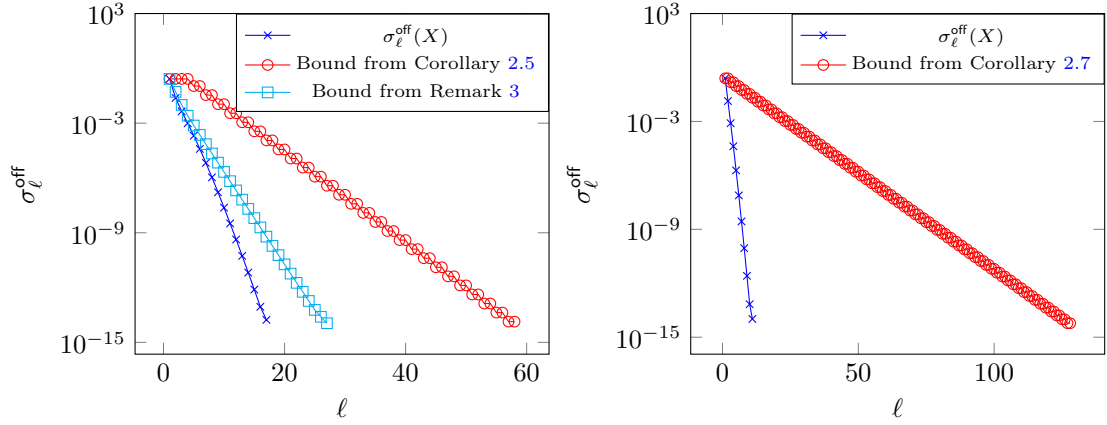
\subsubsection{Case $F\neq I$: bounds based on the numerical ranges of the closed-loop matrices} \label{sec:2.2.2}
Let us address the more general scenario where we consider invertible quadratic coefficients that differ from the identity matrix. A natural idea is to apply some algebraic manipulations to $\eqref{eq:care}$ in order to obtain an equivalent CARE where the quadratic coefficient is the identity matrix and then apply Corollary \ref{cor:sing-decay2}. Pursuing this direction leads to the following result. 
\begin{corollary}\label{cor:sing-decay3}
	Under the same assumptions of Theorem~\ref{thm:sing-decay}, if $F$ is symmetric positive definite, $LL^\top=F$ denotes its Cholesky factorization, and $\mathcal W(L^{-1}AL)\subset \mathbb C^-$,  then the singular values of  any $M\in\mathrm{Off}(X)$   satisfy
	$$
	\frac{\sigma_{ht+r_f+1}(M)}{\norm{X}_2}\le \kappa(F)(1+\sqrt{2})^2 Z_h(E,-E),
	\quad h=1,2,\dots,
	$$
	where $$t=4r_a+6r_f+2r_q,$$
	 $$
	 E=\mathcal W(L^{-1}AL) + \left[-\tau -\sqrt{\tau^2+\norm{F}_2\norm{Q}_2},\ 0\right],
	 $$
	   $$\tau:= |\min\{\mathcal W(L^{-1}AL)\cap \mathbb R\}|.$$   
\end{corollary}
\begin{proof}
	Multiplying \eqref{eq:care} from the left by $L^\top$ and from the right by $L$ we get
	\begin{equation}\label{eq:riccati-cholesky}
		\widetilde A^\top Y + Y\widetilde A- Y^2+ \widetilde Q=0,
	\end{equation}
	with $\widetilde A:=L^{-1}AL, \widetilde Q =L\upd{^\top}QL$, and $Y:=L^\top XL$. Note that $L$ is lower triangular and quasiseparable of order $r_f$  so that also $L^{-1}$ is lower triangular and quasiseparable with the same order. In view of the \upd{subadditivity} properties of the quasiseparable rank, we have that $\qsrank(\widetilde A)\le r_a+r_f$, and  $\qsrank(\widetilde Q)\le r_f+r_q$.
	By applying Corollary~\ref{cor:sing-decay2} to \eqref{eq:riccati-cholesky} we get 
	\begin{equation}\label{eq:Ybound}
		\frac{\sigma_{ht+1}(M_Y)}{\norm{Y}_2}\le (1+\sqrt{2})^2 Z_h(E,-E),
	\end{equation}
	where $M_Y$ indicates 
	any element of $\mathrm{Off}(Y)$. Observe that if we look at the $2\times 2$ block partitioning
	$$
	Y=\begin{bmatrix}
		Y_{11} & Y_{12}\\
		Y_{21}&Y_{22}
	\end{bmatrix}, \qquad X=\begin{bmatrix}
		X_{11} & X_{12}\\
		X_{21}&X_{22}
	\end{bmatrix}\qquad L^{-1}=\begin{bmatrix}
		\widetilde L_{11} & 0\\
		\widetilde L_{21}&\widetilde L_{22}
	\end{bmatrix}
	$$
	where the diagonal blocks are square, and all the matrices are partitioned in the same way, then we have that the offdiagonal blocks of $X=L^{-\upd{\top}}YL^{-1}$ are given by 
	$$
	X_{21}= \widetilde L_{22}^\top Y_{21}\widetilde L_{11} +\widetilde L_{22}^\top Y_{22}\widetilde L_{21}, \qquad  X_{12}= \widetilde L_{11}^\top Y_{12}\widetilde L_{22} +\widetilde L_{21}^\top Y_{22}\widetilde L_{22}.
	$$
	Let us show that the singular values of $X_{21}$ satisfy the claim. Since the rank of $\widetilde L_{22}^\top Y_{22}\widetilde L_{21}$ is at most $r_f$ and the singular values of $Y_{21}$ satisfy \eqref{eq:Ybound}, then we have
	\begin{align*}
		\sigma_{ht+r_f+1}(X_{21})\le \sigma_{hk+1}(\widetilde L_{22}^\top Y_{21}\widetilde L_{11})&\le \norm{F^{-1}}_2 \sigma_{ht+1}(Y_{21})\\ &\le \norm{F^{-1}}_2\norm{Y}_2 (1+\sqrt{2})^2 Z_h(E,-E)
		\\ & \le \norm{F^{-1}}_2 \norm{F}_2\norm{X}_2 (1+\sqrt{2})^2
		Z_h(E,-E),
	\end{align*}
	where we have used the relations: $\norm{\widetilde L_{jj}}_2\le \norm{L^{-1}}_2$, $\norm{L^{-1}}_2\norm{L^{-\upd{\top}}}_2= \norm{F^{-1}}_2$, and $\norm{L}_2\norm{L^\top}_2= \norm{F}_2$.
	As the matrix $X$ is symmetric, there is no need to repeat the argument for $X_{12}$,  and this concludes the proof. 
\end{proof}
\begin{example}\label{ex:decay2}
Corollary~\ref{cor:sing-decay3} suggests that a high condition number of $F$ may destroy the decaying behaviour of the offdiagonal singular values of the solution of \eqref{eq:care}. In this example we show that, actually, $\kappa(F)$ has a much more limited impact on the quantities $\sigma_\ell^{\mathsf{off}}(X)$. We consider two parameter dependent CAREs where we vary the condition number of the coefficient $F$.  More precisely, we take $F$ diagonal with logarithmically spaced diagonal entries in the interval $\left[\frac{1}{\sqrt{\kappa(F)}}, \sqrt{\kappa(F)}\right]$, and values $\kappa(F)\in\{1,10^2,10^4, 10^6,10^8\}$. 
In the two tests, the matrix $A$ is chosen as:
\begin{itemize}
	\item $A=WDW\upd{^\top}$ where $D$ is diagonal with logarithmically spaced entries in $[-1,-0.001]$,
	\item $A= W-1.1 \cdot I$,
\end{itemize} 
where $W$ is a random unitary upper Hessenberg matrix (so $W$ is $1$-quasiseparable as in a unitary matrix the rank of the lower offdiagonal blocks is equal to the corresponding block above). In particular, $W$ is obtained as the $Q$ factor of the QR factorization of a random upper Hessenberg matrix (in Matlab: \texttt{[Wj, \textasciitilde ] = qr(hess(randn(n)))}) and has eigenvalues on the complex unit circle.
For both tests, the $Q$ coefficient is taken as in Example~\ref{ex:decay}.

The relative offdiagonal singular values, i.e., $\sigma_\ell^{\mathsf{off}}(X)/\sigma_1^{\mathsf{off}}(X)$, for the various choices of $\kappa(F)$ are shown in Figure~\ref{fig:sing-decay2}; the left part of the figure refers to the case where $A$ has real spectrum, and the right part to the one where $A$ has eigenvalues on the circle with radius $1$ centered at $-1.1$. From both plots we see that the numerical quasiseparable structure is present in the solution, even for large values of $\kappa(F)$. Counter-intuitively, in the case where $A$ has a real spectrum, we see a slight speed-up in the decay of $\sigma_\ell^{\mathsf{off}}(X)$, as the condition number overcome the value $10^2$. These outcomes indicate that the bound in Corollary~\ref{cor:sing-decay3} is too pessimistic when $\kappa(F)$ is large, at least in typical scenarios. 
\end{example}
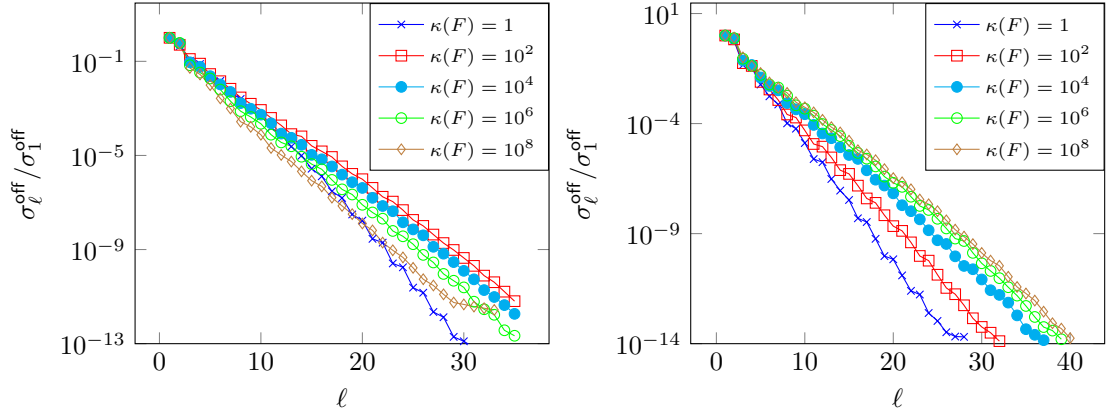
\begin{figure}
	\centering
		\begin{tikzpicture}
		\begin{semilogyaxis}[legend pos = north east, ymax=30, ymin=1e-13 , xlabel = $\ell$, ylabel=$\sigma_\ell^{\mathsf{off}}/\sigma_1^{\mathsf{off}}$, legend style={					
				font=\scriptsize, at={(1,1)}}, width=.48\textwidth]
			\addplot[mark=x, blue] table[x index = 0, y index = 1] {singdecayF.dat};
			\addplot[mark=square, red] table[x index = 0, y index = 2] {singdecayF.dat};
				\addplot[mark=*, cyan] table[x index = 0, y index = 3] {singdecayF.dat};
					\addplot[mark=o, green] table[x index = 0, y index = 4] {singdecayF.dat};
						\addplot[mark=diamond, brown] table[x index = 0, y index = 5] {singdecayF.dat};
			\legend{
			$\kappa(F)=1$\phantom{$0^1$}, $\kappa(F)=10^2$, $\kappa(F)=10^4$, $\kappa(F)=10^6$, $\kappa(F)=10^8$}
		\end{semilogyaxis}
	\end{tikzpicture}~\begin{tikzpicture}
	\begin{semilogyaxis}[legend pos = north east, ymax=30, ymin=1e-14, xlabel = $\ell$, ylabel=$\sigma_\ell^{\mathsf{off}}/\sigma_1^{\mathsf{off}}$, legend style={					
			font=\scriptsize, at={(1,1)}}, width=.48\textwidth]
		\addplot[mark=x, blue] table[x index = 0, y index = 1] {singdecayF2.dat};
		\addplot[mark=square, red] table[x index = 0, y index = 2] {singdecayF2.dat};
		\addplot[mark=*, cyan] table[x index = 0, y index = 3] {singdecayF2.dat};
		\addplot[mark=o, green] table[x index = 0, y index = 4] {singdecayF2.dat};
		\addplot[mark=diamond, brown] table[x index = 0, y index = 5] {singdecayF2.dat};
		\legend{
			$\kappa(F)=1$\phantom{$0^1$}, $\kappa(F)=10^2$, $\kappa(F)=10^4$, $\kappa(F)=10^6$, $\kappa(F)=10^8$}
	\end{semilogyaxis}
\end{tikzpicture}
	\caption{Relative offdiagonal singular values of the solution $X\in\mathbb R^{500\times 500}$ of \eqref{eq:care}, with $A$ diagonal, $F$ diagonal with increasing condition number, and $Q$ tridiagonal symmetric positive definite. On the left, the matrix $A$ has eigenvalues logarithmically spaced in $[-1, -0.001]$; on the right, the eigenvalues of $A$ are on the circle of radius $1$ and center $-1.1$.}\label{fig:sing-decay2}
\end{figure}
\subsubsection{Case $F\neq I$: bounds based on the spectra of the closed-loop matrices}
Theorem 2.1 in \cite{beckermann17} allows us to slightly modify
the bound in Theorem~\ref{thm:sing-decay}  in case additional information on the closed-loop matrices $C^{(j)}:=A^{(j)}-F^{(j)}X^{(j)}$, $C_{(j)}:=A_{(j)}-F_{(j)}X_{(j)}$ is available. For instance, if $C^{(j)},C_{(j)}$ are diagonalizable for every $j$ we can replace the numerical range with the spectra of the closed-loop matrices, by paying the price of having the constant 
$$
\mathfrak K:=\kappa_{\mathrm{eig}}(A-FX)\cdot \max\left\{\max_{j=1,\dots, n}\kappa_{\mathrm{eig}}(C^{(j)})\, \, ,\, \max_{j=1,\dots, n}\kappa_{\mathrm{eig}}(C_{(j)})\right\}. 
$$
in place of $(1+\sqrt 2)^2$. In the following result, we enclose the spectrum of all the closed-loop matrices into a specific subset of $\mathbb C^-$. Additionally, when the antisymmetric part of $L^{\top}AL^{-\top}$ has a small magnitude, we can give a lower bound of the distance of this set from the imaginary axis.
\begin{corollary}\label{cor:sing-decay4}
	 Under the same assumptions of Theorem~\ref{thm:sing-decay}, if $F,Q$ are symmetric positive definite, and the closed-loop matrices $C^{(j)},C_{(j)}$ are diagonalizable for $j=1,\dots,n$,
	  then there \upd{exists} $\tau>0$ such that the singular values of any $M\in\mathrm{Off}(X)$ satisfy
	$$
	\frac{\sigma_{ht+1}(M)}{\norm{X}_2}\le \mathfrak K\cdot
	Z_h(E,-E),
	\quad h=1,2,\dots 
	$$
	where $t=4r_a+2r_f+2r_q$, and $E=\left\{z\in\mathbb C:\ \Re(z)\le -\tau,\  r_\ell\le |z|\le r_u \right\}$
	with 
	\begin{align*}
		r_\ell&:= \sqrt{\frac{\lambda_{\min}(Q)}{3\norm{F^{-1}}_2}},&\quad
		r_u:= \norm{A}_2\sqrt{\kappa(F)}+\sqrt{2\norm{A}_2^2\kappa(F)+\norm{Q}_2\norm{F}_2}.
	\end{align*}
Additionally, if the following inequality holds
\begin{equation}\label{eq:antisym}
\frac{\norm{L^{\top}AL^{-\top} - L^{-1}A^\top L}_2^2}4<\frac{\lambda_{\min}(Q)}{\norm{F^{-1}}_2}, 
\end{equation}
then  $\tau\ge\sqrt{\frac{\lambda_{\min}(Q)}{\norm{F^{-1}}_2}-\frac{\norm{L^{\top}AL^{-\top} - L^{-1}A^\top L}_2^2}4}>0$.
\end{corollary}
\begin{proof}
	To prove the claim it is sufficient to show that the spectra of the closed-loop matrices $C^{(j)},C_{(j)}$ are all contained in $E$. Since the latter have all eigenvalues in $\mathbb C^-$, we can take as $\tau$ the maximum real part over the finite set given by the union of the eigenvalues of $C^{(j)},C_{(j)}$, for $j=1,\dots,n$. 
	To show the inclusion in the circular annulus, we prove that the associated Hamiltonian matrices have eigenvalues in $\{z\in\mathbb C: r_\ell\le|z|\le r_u\}$. We start with the case $j=n$, and observe that
	\begin{align*}
		&\begin{bmatrix}
			L^\top\\ &L^{-1}
		\end{bmatrix}
		\begin{bmatrix}
			A &-Q\\
			-F&-A
		\end{bmatrix}
		\begin{bmatrix}
			L^{-\upd{\top}}\\ &L
		\end{bmatrix}=\begin{bmatrix}
			\widetilde A &-\widetilde Q\\
			-I&-\widetilde A^\top
		\end{bmatrix}\\
		\Rightarrow& \det\left(\begin{bmatrix}
			\lambda I-A &Q\\
			F&\lambda I+A
		\end{bmatrix}\right)=\det(\lambda^2I-(\widetilde A-\widetilde A^\top)\lambda -\widetilde A\widetilde A^\top-\widetilde Q),
	\end{align*}
	with $\widetilde A=L^\top AL^{-\upd{\top}}$, $\widetilde Q=L^\top QL$, and $L$ \upd{being the} Cholesky factor of $F$. In view of the generalized Cauchy theorem for matrix polynomials \cite[Theorem 3.4]{melman2013generalization}, we have that the absolute values of the eigenvalues of the Hamiltonian are all larger than $d$, where $d$ is defined as
	\begin{align*}
		d=\frac{-\norm{\widetilde A-\widetilde A^\top}_2+\sqrt{\norm{\widetilde A-\widetilde A^\top}_2^2+4\norm{(\widetilde A\widetilde A^\top+\widetilde Q)^{-1}}_2^{-1}}}2.
	\end{align*}
	We remark that for any $\theta >0$
	$$
	\min_{x\in [0,\infty)}\sqrt{x^2+\theta}-x=\sqrt{\frac{\theta}{3}},
	$$
	so that
	$$
	d\ge \sqrt{\frac{\norm{(\widetilde A\widetilde A^\top+\widetilde Q)^{-1}}_2^{-1}}{3}}\ge\sqrt{\frac{\lambda_{\min}(\widetilde A\widetilde A^\top)+\lambda_{\min}(\widetilde Q)}{3}}\ge \sqrt{\frac{\lambda_{\min}(Q)}{3\norm{F^{-1}}_2}},
	$$
	and this lower bound applies to all the closed-loop matrices thanks to the interlacing properties of the eigenvalues of symmetric definite matrices.
	
	Concerning the upper bound, we can again apply \cite[Theorem 3.4]{melman2013generalization} that ensures that the absolute values of all the eigenvalues are less than $D$ satisfying
	\begin{align*}
		D&=\frac{\norm{\widetilde A-\widetilde A^\top}_2+\sqrt{\norm{\widetilde A-\widetilde A^\top}_2^2+4\norm{\widetilde A\widetilde A^\top+\widetilde Q}_2}}2\\
		&\le \frac{\norm{\widetilde A-\widetilde A^\top}_2+\sqrt{\norm{\widetilde A-\widetilde A^\top}_2^2+4(\norm{A}_2^2\kappa(F)+\norm{Q}_2\norm{F}_2})}2\\
		&\le \frac{2\norm{\widetilde A}_2+\sqrt{4\norm{\widetilde A}_2^2+4(\norm{A}_2^2\kappa(F)+\norm{Q}_2\norm{F}_2})}2\\
		&\le \norm{A}_2\sqrt{\kappa(F)}+\sqrt{2\norm{A}_2^2\kappa(F)+\norm{Q}_2\norm{F}_2}).
	\end{align*}
	This inequality also applies to any closed-loop matrix associated with principal submatrices of $A,F,$ and $Q$.
	
	To conclude we need to prove the lower bound for the absolute value of the real part of an eigenvalue when \eqref{eq:antisym} is satisfied. Observe that if $(\lambda,v)$ is an eigenpair for the quadratic matrix polynomial $\lambda^2I-\lambda(\widetilde A-\widetilde A^\top)-\widetilde A\widetilde A^\top-\widetilde Q$ then also $(\overline{\lambda},\overline{v})$. If $v$ is of unit norm then we have that
	\begin{align*}
	0&=\frac{1}{2}\left( \lambda^2 -\lambda v^H(\widetilde A-\widetilde A^\top)v-v^H(A\widetilde A^\top+\widetilde Q)v+\overline{\lambda}^2 -\overline{\lambda}\overline{v}^H(\widetilde A-\widetilde A^\top)\overline{v}-\overline{v}^H(A\widetilde A^\top+\widetilde Q)\overline{v}\right)\\	
	&=\Re(\lambda)^2-\Im(\lambda)^2-\Im( v^H(\widetilde A-\widetilde A^\top)v)\cdot  \Im(\lambda) - v^H(A\widetilde A^\top+\widetilde Q)v\\
	&\Rightarrow\ |\Re(\lambda)| = \sqrt{\Im(\lambda)^2+ \Im( v^H(\widetilde A-\widetilde A^\top)v)\cdot  \Im(\lambda) + v^H(A\widetilde A^\top+\widetilde Q)v}\\
		&\Rightarrow\ |\Re(\lambda)| \ge \sqrt{\Im(\lambda)^2-\norm{\tilde A-\tilde{A}^\top}_2\Im(\lambda) +\lambda_{\min}(Q)\norm{F^{-1}}_2^{-1}}
	.
	\end{align*}
By minimizing over $\Im(\lambda)$ \upd{in} the right-hand side  in the above inequality, we get the claim. 
\end{proof}
\begin{remark}
The quantity $\lambda_{\min}(\widetilde A\widetilde A^\top)$ in the proof of Corollary~\ref{cor:sing-decay4} can be lower bounded by $\sigma_n(A)^2\kappa(F)^{-1}$ in place of $0$. However, it appears difficult to control this quantity for all the principal submatrices of $A$, by only assuming  $\mathcal W(A)\subset \mathbb C^-$. Under the additional assumption $A$ symmetric negative definite, there is no need to assume $Q$ positive definite as we  get the tighter lower bound $\sqrt{\frac{\lambda_{\min}(A)^2\kappa(F)^{-1}+\lambda_{\min}(Q)\kappa(F)^{-1}}{3}}$ for the magnitude of the eigenvalues. An analogous improvement is obtained for the lower bound of the absolute value of the real part.
\end{remark}

\subsection{Tighter bounds for the TT ranks of the value function}\label{sec:tt}
When the stabilizing solution $X$ of \eqref{eq:care} has offdiagonal blocks of numerical low-rank, the value function associated to the corresponding linear-quadratic regulator is well approximated in the TT format.
More precisely, let the value function be the multivariate bilinear map
$$
V(\mathbf y)=V(y_1, \dots, y_n) = \sum_{i,j=1}^nX_{ij} \cdot y_i\cdot y_j = \mathbf y^\top X\mathbf y,
$$
\upd{then} we look at compressed representations of $V(\mathbf y)$ in the \emph{functional Tensor Train format}~\cite{oseledets13}, i.e., approximations of the form:
\begin{equation}\label{eq:tt-value}
	V(\mathbf y)\approx \sum_{\alpha_1,\dots, \alpha_{n-1}=1}^{r_1,\dots,r_{d-1}} U_1(y_1,\alpha_1)\cdot U_2(\alpha_1, y_2,\alpha_2)\cdot \ldots\cdot U_{d-1}(\alpha_{d-2}, y_{d-1},\alpha_{d-1})\cdot U_d(\alpha_{d-1}, y_d),
\end{equation}
where the matrix-valued functions $U_j:\mathbb R\rightarrow \mathbb R^{r_{j-1}\times r_j}$ (with $r_0=r_n=1$) are called \emph{TT cores}, and the integer quantities $r_1,\dots, r_{d-1}$, are the so-called \emph{TT ranks}. Demonstrating that $V(\mathbf y)$ admits an accurate approximation of the form in \eqref{eq:tt-value} ensures that most of its discretized counterparts admit storage efficient representations in the tensor train format~\cite{oseledets11}. For instance, \eqref{eq:tt-value} implies that the $m^n$ tensor containing the evaluation of $V(\mathbf y)$ on a tensorized grid with $m$ points along each coordinate direction, is representable with a storage cost of $O(m\cdot n\cdot \max_{j=1,\dots, n-1}\{r_j^2\})$. Therefore, it is of interest to prove sharp bounds for the TT ranks of $V(\mathbf y)$. 

In \cite[Theorem 4.2]{dolgov13} and \cite[Lemma 3.5]{kazeev13}, it has been shown a link between the rank of the offdiagonal blocks of $X$ and the TT ranks of $V(\mathbf y)$; in the case of a symmetric matrix $X$ we have that $V(\mathbf y)$ admits an exact functional tensor train representation such that
\begin{equation*}\label{eq:ttrank}
r_j\leq \rank\left( X(1:j,\ j+1:n)\right) +2,\qquad j=1,\dots,n-1.	
\end{equation*}


In particular, all the TT ranks are bounded by $\qsrank(X)+2$. The compression of $V(\mathbf y)$ in the functional TT format can be extended to the case where $X$ has only numerically low-rank offdiagonal blocks. Indeed, if $\widetilde X$ is a symmetric  quasiseparable matrix of order $r$, and $\norm{X-\widetilde X}_2\le \epsilon$, then the approximate value function $\widetilde V(\mathbf y)=\mathbf y^\top\widetilde X\mathbf y$ admits an exact TT representation with all TT ranks bounded by $r+2$, and it is such that $\max_{\norm{\mathbf y}_2=M}|V(\mathbf y)-\widetilde V(\mathbf y))|\le \epsilon\cdot  M^2$. This argument\footnote{In \cite{dolgov21} the bound is stated for state vectors $\mathbf y$ of bounded infinity norm, i.e., $\max_{\mathbf y\in[-M,M]^n}|V(\mathbf y)-\widetilde V(\mathbf y))|\le \epsilon\cdot  M^2$; however, to get a correct bound  the right-hand side has to be multiplied by an additional factor $n$.} is at the basis of \cite[Theorem 3.1]{dolgov21} where, under some conditions that ensure the decay of the offdiagonal singular values of $X$, it is shown that for any $\epsilon>0$ it exists an approximate value function $\widetilde V(\mathbf y)$ whose TT ranks are bounded by $\mathcal O((\log(\epsilon^{-1}))^{\frac 72})$. We highlight that this result is valid only under the assumption that the quadratic coefficient $F$ is low-rank.

Following similar steps, we  state theoretical bounds on the TT ranks that apply to both the case $F$ low-rank and $F$ of low quasiseparable rank; these bounds incorporate the novel estimates for the offdiagonal singular values of $X$, that we have developed in the previous section. Afterwards we will comment on the advantages of our result with respect to \cite[Theorem 3.1]{dolgov21}. 
\begin{theorem}\label{thm:ttranks}
Let $X\in\mathbb R^{n\times n}$ be  the symmetric negative solution of the continuous-time Riccati equation 
$
A^\top X+XA-XFX+Q=0,
$	where $A\in\mathbb R^{n\times n}$ is  quasiseparable of order $r_a$ with eigenvalues of negative real parts, $F,Q\in\mathbb R^{n\times n}$ are symmetric positive semidefinite,  and $Q$ is quasiseparable of order $r_q$. For any $\epsilon>0$, and $M>0$, if one of the following two conditions is satisfied
\begin{itemize}
	\item[$(i)$] $\rank(F) = r_f$, $\mathcal W(A)\subseteq E\subset \mathbb C^{-}$,	and $h\in\mathbb N$ is the minimum integer such that
	$$
	Z_h(E,-E)\le \frac{\epsilon}{2(1+\sqrt 2)^2M^2\norm{X}_2\sqrt{n}},
	$$
	\item[$(ii)$] $\rank(F)=n$,  $\qsrank(F)=r_f$,  $\mathcal W(L^{-1}AL)\subset \mathbb C^{-}$, $$W(L^{-1}AL)+ \left[-\tau -\sqrt{\tau^2+\norm{F}_2\norm{Q}_2},\ 0\right]\subseteq E$$ with $\tau:= |\min\{\mathcal W(L^{-1}AL)\cap \mathbb R\}|$,	and $h\in\mathbb N$ is the minimum integer such that
	$$
	Z_h(E,-E)\le \frac{\epsilon}{2(1+\sqrt 2)^2M^2\kappa(F)\norm{X}_2\sqrt{n}}\upd{.}
	$$
\end{itemize}
 \upd{Then} there exists an approximate value function $\widetilde V(\mathbf y)$ such that  $\max_{\norm{\mathbf y}_2\le M}|V(\mathbf y)-\widetilde V(\mathbf y))|\le \epsilon$, and $\widetilde V(\mathbf y)$ admits an exact functional TT representation with TT ranks $r_j$ bounded by
\begin{equation*}\label{eq:ttrank-bound}
r_j\le \min\left(c(h),\ \min(j,\ n-j)\right) +2,\qquad j=1,\dots, n-1,	
 	\end{equation*}
where $$
c(h)=
\begin{cases}
2h\cdot  r_a+r_f+r_q	&\text{case $(i)$}\\
h(4r_a+6r_f +2r_q) +r_f	&\text{case $(ii)$}
\end{cases}.
$$
\end{theorem}
\begin{proof}
First, note that $\min(j,n-j)$ is a bound for the rank of the offdiagonal block $X(j+1:n, 1:n)$, so that $\min(j,n-j)+2$ is a bound for the $j$th TT rank of the exact representation of $V(\mathbf y)$, in view of \cite[Theorem 4.2]{dolgov13}. 

Further, in case $(ii)$, Corollary~\eqref{cor:sing-decay3} ensures that
$$
\sigma_{h(4r_a+6r_f+2r_q)+r_f+1}(M)\le \frac{\epsilon}{2M^2\sqrt n}
$$
for every $M\in\mathrm{Off}(X)$. 

In case $(i)$, we show the analogous property for $\sigma_{2hr_a+r_f+r_q+1}(M)$; this follows from  rewriting $X$ as the solution of the linear matrix equation
\begin{equation*}\label{eq:proof}
A^\top X+XA= XFX-Q,
\end{equation*}
where the right-hand side has quasiseparable rank bounded by $r_f+r_q$, and applying Theorem~\ref{thm:quasi-zol}. 

For both cases, by means of \cite[Theorem 2.16]{massei18} we get that there exists a symmetric matrix $\widetilde X$ of quasiseparable order $c(h)$ such that
$$
\norm{X-\widetilde X}_2\le 2\sqrt{n} \max_{M\in\mathrm{Off}(X)}\sigma_{c(h)+1}(M)\le\frac{\epsilon}{M^2}.
$$
To get the claim, we consider the inequality $|V(\mathbf y)-\widetilde V(\mathbf y)|=|\mathbf y^\top(X-\widetilde X)\mathbf y|\le \norm{X-\widetilde X}_2\norm{\mathbf y}_{2}^2$.
\end{proof}
Since the condition $E\subset \mathbb C^{-}$ ensures the exponential decay of the Zolotarev numbers $Z_h(E,-E)$ with respect to the index $h$, Theorem~\ref{thm:ttranks} bounds the growth of the TT ranks, as the approximation error $\epsilon$ of the value function decreases, with an $\mathcal O(\log(\epsilon^{-1}))$ term. This is asymptotically much tighter than the bound
\begin{equation*}\label{eq:sergey-bound}
r_j\le (r_a + 2r_f+r_q)\cdot \log(\epsilon^{-1} + D)^{\frac 72}+2, \qquad \text{with $D$ independent on $\epsilon$},
\end{equation*}
stated in  \cite[Theorem 3.1]{dolgov21} (that applies only to the case $(i)$). This is due to the fact that the latter is obtained by means of an exponential sum approximation of the inverse Lyapunov operator, that has only a square root exponential convergence rate. 

  Another advantage of the bounds in Theorem~\ref{thm:ttranks} is that the only quantity that remains implicit is $\norm{X}_2$; even if estimating $\norm{X}_2$  might be difficult in general,  when $F$ is full rank (case $(ii)$) the upper bound  given in Lemma~\ref{lem:sol-norm} can be employed to get an explicit a priori estimate of the TT ranks, once that the quantities $\epsilon$ and $M$ have been selected.  On the opposite,  the constant $D$ in \eqref{eq:sergey-bound} depends on several implicit quantities, some of them  related to the closed-loop matrix.
\subsection{Efficient solution of the Riccati equation}\label{sec:quasi-care-solve}
The quasiseparable structure in the coefficients and in the stabilizing solution enables us to design a divide-and-conquer algorithm for the solution of \eqref{eq:care}, that is applicable to the large-scale scenario. The latter procedure is based on the efficient representation of quasiseparable matrices by means of hierarchical matrix formats, and the use of a low-rank solver for the CAREs with low-rank right-hand sides. The next sections are dedicated to explaining these ingredients and the ultimate algorithm.

\subsubsection{Representing quasiseparable matrices in the HSS format}
Matrices having a low quasiseparable order are representable with a linear or almost linear amount of parameters ~\cite{vandebril08,eidelman14}, and to efficiently store and operate with such matrices it is often a good idea to rely on hierarchical matrix formats. In this work, we rely on a specific subclass of the set of hierarchical representations 
called hierarchically semiseparable matrices (HSS)~\cite{xia10}. As discussing hierarchical representations is below the scope of this work, we will keep the amount of details to a minimum here, and refer the interested reader to \cite{hackbusch15,xia10,xia2012complexity}.

To represent an $n\times n$ matrix $M$ of quasiseparable order $r$ in the HSS format, we consider the $2\times 2$ block partitioning
\begin{equation}\label{eq:splitting}
M=\begin{bmatrix}
	M(I_1^{(1)}, I_1^{(1)}) & M(I_1^{(1)}, I_2^{(1)})\\
	M(I_2^{(1)}, I_1^{(1)}) & M(I_2^{(1)}, I_2^{(1)})	
\end{bmatrix},
\end{equation}
 where the diagonal blocks are taken of the closest sizes, say $\lceil\frac n2\rceil\times \lceil\frac n2\rceil$ and $\lfloor\frac n2\rfloor\times \lfloor\frac n2\rfloor$, respectively. The offdiagonal blocks have rank at most $r$ and are compressed in a low-rank matrix format; the diagonal blocks are square and can 
 be recursively partitioned in the same way. The recursion is stopped when the blocks 
 at the lowest level are smaller than 
 a prescribed minimal size $n_{\min} \times n_{\min}$; the latter blocks are then stored as dense matrices. 
Thus, for a generic offdiagonal block at the level $\ell$ of the recursion we have a factorization
\[
M(I^{(\ell)}_i,I^{(\ell)}_j) = U_i^{(\ell)} S_{i,j}^{(\ell)} (V_j^{(\ell)})^\top, \quad S_{i,j}^{(\ell)} \in\mathbb C^{r\times r}, \quad U_i^{(\ell)}\in \mathbb C^{n_i^{(\ell)} \times r},\quad V_j^{(\ell)} \in \mathbb C^{n_j^{(\ell)} \times r}.
\]
To achieve a linear complexity storage, in the HSS format the bases matrices of these low-rank representations are nested across the different levels. In general, the bases matrices of an HSS representation might have a non uniform number of columns; we call \emph{HSS rank} the maximum number of columns of a basis matrix across all the level of the HSS representation. In the case of a quasiseparable matrix of order $r$, the HSS rank of its HSS representation is at most $2r$. 

Using this structured representation, one only needs $O(nr)$ memory to represent $M$. We remark that retrieving the HSS representation of a matrix in an efficient way is a non-trivial task \emph{per se}, and there are recent works that describe how to do that in the case where $M$ is accessible through its action on vectors~\cite{halikias2024structured,levitt2024linear}. In this work we always assume that the HSS representation of the matrices of interest is already available. 

Matrix-vector multiplications and solution 
of linear systems with a matrix in the HSS format can be performed with $\mathcal O(nr)$ and 
$\mathcal O(nr^2)$ flops, respectively. Our numerical results leverage the implementation of this format and the related matrix operations 
available in \texttt{hm-toolbox} \cite{massei20}. 
\subsubsection{Low-rank approximation of $\delta X$}
The analysis carried out in section~\ref{sec:decay-care} suggests the use of a solver that computes a low-rank approximation of the matrix $\delta X$ that solves equation~\eqref{eq:riccati_correction}. For large-scale CAREs with low-rank right-hand sides, as \eqref{eq:riccati_correction},
various numerical methods have been proposed~\cite{benner20,simoncini14,simoncini16b,radi18}; importantly, these methods can also leverage the HSS representation of the coefficients to lower their computational costs. 

Similarly to \cite[Section 3.2]{kressner20}, we employ a strategy based on projecting the matrix equation onto a tensorized Krylov subspace. To be more specific, let $U\in\mathbb R^{n\times t}$ contain an orthonormal basis of $t$-dimensional subspace $\mathcal U\subset\mathbb R^n$. Then, we consider approximate solutions of the form $\widetilde {\delta X}:=UYU^\top$, where $Y\in\mathbb R^{t\times t}$ is obtained
from solving a compressed matrix equation. The latter corresponds to imposing a Galerkin condition on the residual with respect to the tensorized space  $\mathcal U\otimes\mathcal U$. 

The choice of the projection subspace $\mathcal U$ is crucial for the accuracy of the approximation.  Here we consider the so-called \emph{extended Krylov subspace}~\cite{simoncini07} that is a popular and effective projection subspace, in the context of matrix equations. 
\begin{definition} Let $A\in\mathbb R^{n\times n}$, $B \in\mathbb R^{n\times k}$, $k<n$, and $s\in\mathbb N$. The vector space 
	\begin{align*}
		\mathcal {EK}_\upd{s}(A,B):= \mathrm{range}\Big\{\Big[ B,A^{-1}B,AB, A^{-2}B,\ldots,A^{s-1}B, A^{-s}B\Big] \Big\}
	\end{align*}
	is called \emph{extended Krylov subspace} of order $s$ with respect to $(A,B)$.
\end{definition}

To give a concise description of what the extended Krylov subspace method (EKSM)~\cite{simoncini07} does for approximating  the solution of~\eqref{eq:riccati_correction}, let us introduce 
$A_{\mathsf{cl}}:= A- FX_0$. Then,  suppose that the offdiagonal parts of the CARE's coefficients are given in factorized form:
\[
\delta A=U_AV_A^*, \quad \delta
Q=U_QD_QU_Q^*, \quad \delta F=U_FD_FU_F^*,
\]
with $U_A,V_A\in\R^{n\times \rank{(\delta A)}}$, $U_Q\in\R^{n\times \rank{(\delta Q)}}$, $U_F\in\R^{n\times \rank{(\delta F)}}$, and symmetric matrices $D_Q \in\R^{\rank{(\delta Q)}\times \rank{(\delta Q)}}$, $D_F \in\R^{\rank{(\delta F)}\times \rank{(\delta F)}}$. Thus, the right-hand side $\widehat Q:= -\delta Q-\delta A^\top X_0-X_0\delta A +X_0\delta F X_0$ of~\eqref{eq:riccati_correction} can be written as $\widehat Q= UDU^\top$, with 
\begin{equation*}\label{eq:rhs-fact}
	U:=[U_Q,V_A,X_0U_A,X_0U_F], \qquad D=\begin{bmatrix} -D_Q\\ &
	\begin{bmatrix} 0&-I_{\rank{(\delta A)}}\\ -I_{\rank{(\delta A)}}&0\end{bmatrix}\\ &&D_F\end{bmatrix}.
\end{equation*}
The correction equation~\eqref{eq:riccati_correction} now reads as
$$
	A_{\mathsf{cl}}^\top\delta X +\delta XA_{\mathsf{cl}}- \delta XF\delta X = \widehat Q.
$$
The EKSM constructs, incrementally with respect to the parameter $s$, approximate solutions of the form $\delta X_s=U_sY_sU_s^\top$, where 
$U_s$ contains an orthonormal basis of  $\mathcal {EK}_s(A^\top_{\mathsf{cl}},U)$. The small matrix 
$Y_s$ is determined as the solution of the compressed CARE
\begin{align*}
	\tilde A^\top_{\mathsf{cl}}Y_s+Y_s\tilde A_{\mathsf{cl}}-Y_s\tilde FY_s +  \tilde UD\tilde U^\top= 0,\quad \tilde A_{\mathsf{cl}}:=U_s^\top A_{\mathsf{cl}}U_s,~\tilde F:=U_s^\top FU_s,~\tilde U:=U_s^\top U.
\end{align*}
Note that, the condition $\mathcal W(A_{\mathsf{cl}})\subset \mathbb C^-$ implies the existence and uniqueness of the symmetric stabilizing solution of the compressed CARE; although the latter solution might be indefinite, this poses no major issue. The computation of $Y_s$ is addressed by algorithms for small, dense CAREs \cite{benner2015numerical}. 

The pseudocode of EKSM is reported in Algorithm~\ref{alg:EKSM}. We remark that the quasiseparable matrices $A_{\mathsf{cl}},F\in\mathbb R^{n\times n}$ are expected to be represented within the HSS format introduced in the previous section. The HSS arithmetic is  exploited in  the execution of the extended block Arnoldi algorithm at line~\ref{step:arnoldi}, that requires to compute matrix vector multiplications and solve linear systems with the matrix $A_{\mathsf{cl}}$; these operations have a linear complexity with respect to $n$.

We refer to the literature~\cite{simoncini14,simoncini16b} and to \cite[Section 3.2]{kressner20} for further  implementation details about this well established method.

\begin{algorithm}
	\caption{Extended Krylov subspace method for solving $A_{\mathsf{cl}}^\top\delta X +\delta XA_{\mathsf{cl}}- \delta XF\delta X = UDU^\top$}\label{alg:EKSM}
	\begin{algorithmic}[1]
		\Procedure{low\_rank\_care}{$A_{\mathsf{cl}},F, U,D$}
		\For{$s=1,2,\dots$}
		\State Compute (incrementally) an orthonormal basis $U_s$ for $\mathcal{EK}_s(A_{\mathsf{cl}},U)$ via the extended block Arnoldi algorithm \label{step:arnoldi}
		\State $\widetilde A_{\mathsf{cl}}\gets U_s^\top A_{\mathsf{cl}} U_s$, \quad $\widetilde F \gets U_s^\top FU_s$,\quad $\widetilde Q\gets U_s^\top UDU^\top U_s$
		\State $Y_s\gets$  \Call{dense\_care}{$\widetilde A_{\mathsf{cl}},\widetilde F,\widetilde Q$}
		\If{Converged}
		\State \Return $\widetilde {\delta X} =U_sY_sV_s^\top$
		\EndIf
		\EndFor 
		\EndProcedure
	\end{algorithmic}
\end{algorithm}

\begin{remark} 
We stress that, for solving \eqref{eq:riccati_correction}, we might replace EKSM with other low-rank solvers such as the rational Krylov subspace method~\cite{simoncini14,simoncini16b} or ADI-type methods~\cite{radi18} that offer analogous computational benefits. Such methods require to select some shift parameters that are based on an estimate of the spectral properties of the coefficients. This is crucial to guarantee a fast convergence, that is potentially superior to the one of EKSM. On the other hand, EKSM has been observed to have a fast convergence on the numerical tests considered in this work, so we decided to keep this method that \upd{does} not need additional parameters. 
\end{remark}

\subsubsection{A divide-and-conquer scheme}
We now derive a divide-and-conquer method for  equations of the form \eqref{eq:care} with HSS coefficients $A,F,Q$; the paradigm is  similar to the divide-and-conquer methods discussed in \cite{kressner20}, and \cite{kressner19}.
 
As already seen in Section~\ref{sec:decay-care}, the solution of~\eqref{eq:care} can be decomposed as $X=X_0+\delta X$ where $X_0$ is the stabilizing solution of \eqref{eq:riccati_diag}, that has block diagonal coefficients, while $\delta X$ is the stabilizing solution of \eqref{eq:riccati_correction}, that has a low-rank right-hand side. In particular, if the diagonal blocks of the coefficients in \eqref{eq:riccati_diag} corresponds to those identified at the first level of the HSS representations of $A,F$, and $Q$, then this equation can be split into two CAREs (with halved sizes) that have again HSS coefficients. Thus, a natural idea is to  recursively solve the CAREs associated with the diagonal
blocks, form the block diagonal solution $X_0$ of \eqref{eq:riccati_diag}, approximate the solution $\delta X$ of \eqref{eq:riccati_correction} via Algorithm~\ref{alg:EKSM}, and finally return $X_0+\delta X$. The resulting procedure, that we denote by \textsc{d\&c\_care}, is summarized in Algorithm~\ref{alg:dac_care}. 

We remark that the low-rank factorization computed at line~\ref{step:lr-rhs} might be redundant and we recommend to apply a recompression procedure, such as \cite[Algorithm 2.17]{hackbusch15}, that aims at reducing its rank. Similarly, the operations at line~\ref{step:sum1} and at line~\ref{step:sum2} may provide HSS representations of the output with higher HSS ranks than necessary; also here, we recommend to apply a recompression procedure tailored to the HSS format, e.g., \cite[Algorithm 3]{xia2012complexity}. 

The computational complexity of Algorithm~\ref{alg:dac_care} depends on the convergence of \textsc{eksm}, and to the cost of the method used for solving the small dense equations associated with the diagonal blocks on the lowest
level of the recursion. To provide some insights on the typical cost,
we make the following idealistic assumptions:
\begin{itemize}
	\item $n=2^pn_{\min}$ and the partitioning \eqref{eq:splitting} always generate equally sized
	diagonal blocks;
	\item \textsc{eksm} converges in a constant number of iterations;
	\item solving the dense unstructured CAREs at the base of the recursion costs $\mathcal O(n_{\min}^3)$;
	\item the matrices $A,Q,F$ have HSS rank bounded by $r$;
	\item the recompression procedures at lines~\ref{step:lr-rhs}, \ref{step:sum1}, and \ref{step:sum2}, always provide outcome with $\mathcal O(r)$ rank or $\mathcal O(r)$ HSS rank. 
\end{itemize}
Under the above assumptions, we have that the cost of each call of \textsc{eksm} is $\mathcal O(nr^2)$, and is dominated by the cost of computing the ULV factorization of $A_{\mathsf{cl}}$, and solving $\mathcal O(r)$ linear systems with triangular HSS matrices. Thanks to the other assumptions, the HSS rank of $X_0$ is $\mathcal O(r)$ and the cost of all the lines in Algorithm~\ref{alg:dac_care}, apart from the recursive calls, is dominated by $\mathcal O(nr^2)$. Thus, the cost of solving an equation of size $n$ is given by $\mathcal O(nr^2)$ plus twice the cost of solving an equation of size $\frac n2$. By applying the master theorem, we get an overall complexity of $\mathcal O(n\log(n)r^2)$.  
\begin{algorithm}
	\caption{Divide-and-conquer method for CAREs with HSS coefficients}\label{alg:dac_care}
	\begin{algorithmic}[1]
		\Procedure{d\&c\_care}{$A,F,Q$}\Comment{Solve $A^\top X+XA-XFX+Q=0$}
		\If{$A$ is a dense matrix}
		\State \Return \Call{dense\_care}{$A,F,Q$}
		\Else
		\State Decompose \[A =\begin{bmatrix}
			A_{11}&0 \\ 0&A_{22}
		\end{bmatrix}+U_AV_A^\top,\ F=\begin{bmatrix}
		F_{11}&0 \\ 0&F_{22}
	\end{bmatrix} + U_FD_FU_F^\top,\, Q=\begin{bmatrix}
			Q_{11}& 0\\ 0&Q_{22}
		\end{bmatrix}+U_QD_QU_Q^*\]
		\State $X_{11} \gets $ \Call{d\&c\_care}{$A_{11},B_{1},Q_{11}$}
		\State $X_{22} \gets $ \Call{d\&c\_care}{$A_{22},B_{2},Q_{22}$}
		\State \label{line:X0} Set $X_0\gets\begin{bmatrix}
			X_{11} & 0 \\ 0&X_{22}
		\end{bmatrix}$
		\State \label{line:largerlowrank} Set $U=[U_Q, V_A, X_0U_A,X_0U_F]$ and 
		$D$ as in~\eqref{eq:rhs-fact}. \Comment{Recompression is recommended}\label{step:lr-rhs}
		\State $A_{\mathsf{cl}}\gets A-FX_0$\label{step:sum1} \Comment{HSS recompression is recommended}
		\State \label{line:lrupdate} $\delta X\gets $ \Call{low\_rank\_care}{$A_{\mathsf{cl}},F,U,D$}
		\State \label{step:hodlrcompress} \Return $X_0+\delta X$ \Comment{HSS recompression is recommended} \label{step:sum2} 
		\EndIf
		\EndProcedure
	\end{algorithmic}
\end{algorithm}

\subsubsection{Testing the complexity of Algorithm~\ref{alg:dac_care}}\label{sec:caresolve}
Let us showcase the scaling property of Algorithm~\ref{alg:dac_care} on CAREs with randomly generated quasiseparable coefficients of increasing sizes or increasing quasiseparable orders. We consider five equations of the form \eqref{eq:care}, each one identified by a coefficient $A_j$, for $j=1,\dots,5$, but sharing the same  quadratic coefficient $F$ and right-hand side $Q$. In the first $4$ tests we select parameter dependent matrices of size $n=500\cdot 2^h$, with $h=0,\dots, 4$; in the last test we fix $n=2000$ and we take $\qsrank(A_5)=2^h$, with $h=1,\dots, 5$.	The major difference between the first four case studies is the numerical range of the coefficient $A_j$. In test $1$ and $2$ the distance between the numerical range of $A_j$ and the imaginary axis remain constant as $n$ increases; instead, in test $3$ and $4$, $\mathcal W(A_j)$ gets closer to $\mathbf i\cdot \mathbb R$, as $n\to\infty$. In test $5$ we take a matrix with the same spectral properties of $A_1$ and we only vary its quasiseparable rank. 

To generate quasiseparable matrices of low order with a prescribed numerical range, we leverage random unitary upper Hessenberg matrices generated as in Example~\ref{ex:decay2}. More precisely, in test $1$ and $3$, we set $A_j= W_jD_jW_j^H$ where $W_j$ is obtained as the $Q$ factor of the QR factorization of a random upper Hessenberg matrix, and $D_j$ is a diagonal matrix with prescribed spectrum. Test $5$ is generated analogously apart from the matrix $W_5$ that is obtained starting from a random upper Hessenberg matrix with $2^h$ subdiagonals. In test $2$ and $4$, we set $A_j=W_j- \theta_j I_n$, with $\theta_j\in\mathbb R$ and $W_j$ unitary upper Hessenberg matrix generated as before. Note that, in the first $4$ tests the matrix  $A_j$ is normal, and of quasiseparable order at most $2$. $\mathcal W(A_1),\mathcal W(A_3)$ are given by the convex hull of the diagonal entries of $D_1$, and $D_3$, respectively. $\mathcal W(A_2),\mathcal W(A_4)$ are enclosed in the disc of radius $1$ and center $-\theta_j$, $j=2,4$. We recap the five choices of $A_j$ below:
\begin{itemize}
	\item $A_1=W_1 D_1 W_1^H$, where $D_1$ has logarithmically spaced diagonal entries in $[-1, -0.001]$ (in Matlab: \texttt{D1=-diag(logspace(-3, 0, n))}), 
	\item $A_2=W_2 - 2 \cdot I_n$, 
	\item $A_3=W_3 D_3 W_3^H$, where $D_3$ has logarithmically spaced diagonal entries in $[-1, -\frac{1}{n^2}]$ (in Matlab: \texttt{D3=-diag(logspace(-2*log10(n), 0, n))}),
	\item $A_4=W_4 - (1 + \frac{1}{\log{n}}) \cdot I_n$.
	\item $A_5=W_5 D_1 W_5^H$, where $n=2000$, and \texttt{[$W_5$,\textasciitilde]=qr(triu(randn(n), -r))}, $r\in \{2,4,8,16,32\}$. 
\end{itemize}
For all test cases, the quadratic coefficient and the right-hand side are set as
\begin{itemize}
	\item $F=W_F D_F W_F^H$, where $D_F$ has logarithmically spaced diagonal entries in $[0.01, 100]$ (in Matlab: \texttt{DF = diag(logspace(-2, 2, n))}),
	\item $Q=W_Q D_Q W_Q^H$, where $D_Q$ has equispaced diagonal entries in $[0, 1]$ (in Matlab: \texttt{DQ = diag(linspace(0, 1, n))}),
\end{itemize}
and $W_F,W_Q$ are random unitary upper Hessenberg matrices, generated as $W_j$, $j=1,\dots,4$. 

The five numerical tests described above, and all the other experiments in this paper,  have been run on a server with two Intel(R) Xeon(R) E5-2650v4 CPU with 12 cores and 24 threads each, running at 2.20 GHz, using
MATLAB R2021a with the Intel(R) Math Kernel Library Version 11.3.1. Algorithm~\ref{alg:EKSM} is stopped when the approximate solution attains a relative residual norm less than $10^{-8}$. Concerning the HSS representation: the minimum
block size parameter $n_{\min}$ is set to $250$ and the tolerance for the compression of the low-rank blocks is set to $10^{-10}$.   

The performances of Algorithm~\ref{alg:dac_care} and of the Matlab function \texttt{icare} are reported in Table~\ref{tab:hss-complexity1} where the label `Time' denotes the CPU time of the execution of the algorithm and  `Res' indicates the relative residual 
$$\mathrm{Res}:=\norm{A^\top X-XA^\top-XFX+Q}_F/\norm{Q}_F,$$
associated with the corresponding approximate solution.
For Algorithm~\ref{alg:dac_care}, we also report the HSS rank of the latter approximate solution. As the numerical tests involved randomly generated matrices, the quantities reported in Table~\ref{tab:hss-complexity1} have been averaged over $20$ runs.

We see that, in test 1,2, and 3, the HSS rank of the solution is almost independent on $n$, and the timings of our Algorithm~\ref{alg:dac_care} scales linearly with $n$; in Test $4$ the HSS rank of the solution mildly grows with the size of the problem, and this makes the CPU time of Algorithm~\ref{alg:dac_care} to grow slightly more than linearly. The reference dense method, \texttt{icare}, has a higher time consumption in all tests and its dependency on the parameter $n$ is roughly cubic; on the largest instances of the tests, Algorithm~\ref{alg:dac_care} achieves  speed-up factors from $65\times$ to $700\times$. The residual norms associated with \texttt{icare} are 3 to 4 order of magnitude more accurate than those of Algorithm~\ref{alg:dac_care}; this is due to the tolerance $10^{-8}$, used to stop \textsc{eksm} when solving the CAREs with a low-rank right-hand side encountered in the divide-and-conquer scheme.  

The timings in test $5$ indicates an almost quadratic scaling of the computational costs with respect to the quasiseparable order of the coefficients, when the latter is sufficiently high, i.e., $r\ge 8$. Comments similar to the other tests, apply to the residuals norms measured in this test.

	\begin{table}[h!]
		\centering
	\caption{Comparison of the performances of  Algorithm~\ref{alg:dac_care} and \texttt{icare}, when solving the five CAREs considered in Section~\ref{sec:caresolve}; the results are averaged over $20$ runs.}
	\vspace{.3cm}
	
		\begin{filecontents*}{testcomplexity1.dat}
%% LaTeX2e file `testcomplexity1.dat'
%% generated by the `filecontents' environment
%% from source `paper' on 2026/01/29.
%%
500  1.6934 0.12646  4.4103e-10 7.8545e-11 20.8 0.41039 2.1312 1.7337e-14 1.392 0.02948  1.431e-10 1.8348e-11 29  0  2.113 1.2935e-14
1000 3.2888 0.080416 5.2424e-10 1.2556e-10 30.9 0.64072 15.002 1.6939e-14 3.0711 0.040577 2.4602e-10 3.5942e-11 52.25 0.44426 11.062 1.7543e-14
2000 6.222 0.15788  4.7548e-10 7.5889e-11 30.3 0.47016 108.23 2.0654e-14 6.6171 0.056717 3.3279e-10 1.0589e-10 53.2 0.41039 75.088 2.2824e-14
4000 11.955 0.11937  4.981e-10 2.594e-11 30.35 0.67082 1976.5 2.5784e-13 14.297 0.089958 2.2948e-10 2.0048e-11 53.25 0.44426 544.46 3.2472e-14
8000 23.507 0.13967  4.4214e-10 2.4869e-11 30.25 0.55012 15705 6.575e-13 30.27 0.10438  3.811e-10 2.9076e-11 53.95 0.22361 4073 4.5818e-14
\end{filecontents*}
\pgfplotstabletypeset[
		every head row/.style={
			before row={
				\toprule
				\multicolumn{6}{c}{\textbf{Test 1}}\\
				\hline			
				\multicolumn{1}{c|}{}&
				\multicolumn{3}{c|}{\textsc{d\&c\_care}}&
				\multicolumn{2}{c}{\texttt{icare}}\\
			},
			after row = \midrule,
		},
		columns = {0,1,3,5,7,8}, 
		columns/0/.style = {column name = $n$, column type=c|},
		columns/1/.style = {column name = Time,precision=1,zerofill, fixed},
		columns/3/.style = {column name = Res,precision=1,zerofill},
		columns/5/.style = {column name = HSS rank,, column type=c|,precision=1,zerofill, fixed,precision=1,zerofill, fixed},
		columns/7/.style = {column name = Time ,precision=1,zerofill, fixed},
		columns/8/.style = {column name = Res,precision=1,zerofill},
		]{testcomplexity1.dat}\vspace{.3cm}
		
				\pgfplotstabletypeset[
		every head row/.style={
			before row={
				\toprule
\multicolumn{6}{c}{\textbf{Test 2}}\\
\hline			
\multicolumn{1}{c|}{}&
\multicolumn{3}{c|}{\textsc{d\&c\_care}}&
\multicolumn{2}{c}{\texttt{icare}}\\
			},
			after row = \midrule,
		},
		columns = {0,9,11,13,15,16}, 
		columns/0/.style = {column name = $n$, column type=c|},
		columns/9/.style = {column name = Time,precision=1,zerofill, fixed},
		columns/11/.style = {column name = Res,precision=1,zerofill},
		columns/13/.style = {column name = HSS rank, column type=c|, precision=1,zerofill, fixed},
		columns/15/.style = {column name = Time,precision=1,zerofill, fixed},
		columns/16/.style = {column name = Res,precision=1,zerofill, column type=c},
		]{testcomplexity1.dat}
		\vspace{.3cm}
		
				\begin{filecontents*}{testcomplexity2.dat}
%% LaTeX2e file `testcomplexity2.dat'
%% generated by the `filecontents' environment
%% from source `paper' on 2026/01/29.
%%
500  1.7671 0.045528 4.6415e-10 8.1436e-11 20.8 0.41039 2.719 2.4886e-14 2.0687 0.12895  3.8289e-10 1.9648e-10 39  0.32444 2.2222 1.8193e-14
1000 3.6108 0.076041 4.0433e-10 1.0222e-10 30.5 0.51299 18.126 2.185e-14 5.778 0.16255  4.8735e-10 1.5751e-10 84.5 0.60698 12.339 2.3248e-14
2000 6.9264 0.047786 4.1256e-10 5.5137e-11 30.6 0.50262 152.14 2.5136e-14 13.895 0.085288 4.6304e-10 4.544e-11 98.95 0.39403 88.492 2.9476e-14
4000 12.702 0.0954  4.7248e-10 3.2253e-11 30.65 0.48936 2225.6 2.634e-13 33.461 0.53246  5.4431e-10 4.8754e-11 108.6 0.59824 653.44 3.8532e-14
8000 24.783 0.26042  4.0529e-10 3.103e-11 29.7 0.80131 18591 9.5624e-13 76.22 0.74432  5.8021e-10 2.3932e-11 114.75 0.71635 5030.8 5.5196e-14
\end{filecontents*}
\pgfplotstabletypeset[
		every head row/.style={
			before row={
				\toprule
\multicolumn{6}{c}{\textbf{Test 3}}\\
\hline			
\multicolumn{1}{c|}{}&
\multicolumn{3}{c|}{\textsc{d\&c\_care}}&
\multicolumn{2}{c}{\texttt{icare}}\\
			},
			after row = \midrule,
		},
		columns = {0,1,3,5,7,8}, 
		columns/0/.style = {column name = $n$, column type=c|},
		columns/1/.style = {column name = Time,precision=1,zerofill, fixed},
		columns/3/.style = {column name = Res,precision=1,zerofill},
		columns/5/.style = {column name = HSS rank, column type=c|,precision=1,zerofill, fixed},
		columns/7/.style = {column name = Time,precision=1,zerofill, fixed},
		columns/8/.style = {column name = Res,precision=1,zerofill},
		]{testcomplexity2.dat}
		\vspace{.3cm}
		
		\pgfplotstabletypeset[
		every head row/.style={
			before row={
				\toprule
\multicolumn{6}{c}{\textbf{Test 4}}\\
\hline			
\multicolumn{1}{c|}{}&
\multicolumn{3}{c|}{\textsc{d\&c\_care}}&
\multicolumn{2}{c}{\texttt{icare}}\\
			},
			after row = \midrule,
		},
		columns = {0,9, 11,13,15,16}, 
		columns/0/.style = {column name = $n$, column type=c|},
		columns/9/.style = {column name = Time,precision=1,zerofill, fixed},
		columns/11/.style = {column name = Res,precision=1,zerofill},
		columns/13/.style = {column name = HSS rank, column type=c|,precision=1,zerofill, fixed},
		columns/15/.style = {column name = Time,precision=1,zerofill, fixed},
		columns/16/.style = {column name = Res,precision=1,zerofill, column type=c},
		]{testcomplexity2.dat}
		\vspace{.3cm}
		
				\begin{filecontents*}{testcomplexity_rank.dat}
%% LaTeX2e file `testcomplexity_rank.dat'
%% generated by the `filecontents' environment
%% from source `paper' on 2026/01/29.
%%
2       5.4406  0.23084 3.3063e-10      7.1287e-11      25.7    0.71414 117.56  7.3642  2.4321e-14      8.5428e-16
4       6.3991  0.18632 4.0635e-10      8.1322e-11      33.65   0.57228 114.98  7.2838  2.4218e-14      7.8665e-16
8       11.524  0.71605 3.9635e-10      7.0632e-11      50.25   0.82916 118.18  7.0677  2.4093e-14      1.0501e-15
16      37.249  2.577   4.4246e-10      4.7602e-11      95.65   1.3143  119.38  7.652   2.3468e-14      1.0729e-15
32      114.87  1.7133  9.488e-10       8.329e-11       218.9   1.044   117.46  7.4653  2.3055e-14      1.4317e-15
\end{filecontents*}
\pgfplotstabletypeset[
		every head row/.style={
			before row={
				\toprule
				\multicolumn{6}{c}{\textbf{Test 5}}\\
				\hline			
				\multicolumn{1}{c|}{}&
				\multicolumn{3}{c|}{\textsc{d\&c\_care}}&
				\multicolumn{2}{c}{\texttt{icare}}\\
			},
			after row = \midrule,
		},
		columns = {0,1,3,5,7,9}, 
		columns/0/.style = {column name = $\qsrank(A_5)$, column type=c|},
		columns/1/.style = {column name = Time,precision=1,zerofill, fixed},
		columns/3/.style = {column name = Res,precision=1,zerofill},
		columns/5/.style = {column name = HSS rank, column type=c|,precision=1,zerofill, fixed},
		columns/7/.style = {column name = Time,precision=1,zerofill, fixed},
		columns/9/.style = {column name = Res,precision=1,zerofill, column type=c},
		]{testcomplexity_rank.dat}
	\label{tab:hss-complexity1}
\end{table}

\section{The banded case}\label{sec:banded}
Specific instances of the problem considered so far are CAREs with banded coefficients, i.e., the offdiagonal blocks of the matrices are both low-rank and sparse. Throughout this section we denote by $\beta_a$, $\beta_f$, and $\beta_q \in \mathbb{N}$ the bandwidths of  $A$, $F$ and $Q$, respectively, and we assume that these parameters are significantly smaller than the size of the matrices. In the latter scenario, a natural question is whether the solution $X$ is also approximately banded. We see that an offdiagonal decay in the magnitude of the entries can be formally proved in the case $A$ symmetric positive definite, and $F=I$. In view of formula \eqref{eq:caresol}, and in the spirit of results ensuring decay properties for functions of sparse matrices~\cite{benzi-boito}, we consider approximations of the form $X\approx p(A^2+Q)+A$, where $p(z)$ is a polynomial approximation of $\sqrt{z}$ on an interval containing the spectrum of $A^2+Q$. If $h$ is the degree of $p(z)$, then $p(A^2+Q)+A$ has bandwidth at most $h\cdot \max\{2\beta_a,\beta_q\}$; this implies the inequality
\begin{equation}\label{eq:decay-band}
|X_{ij}|\le \min_{\mathrm{deg}(p)\le h}\norm{X-p(A^2+Q)-A}_2\le \min_{\mathrm{deg}(p)\le h}\ \max_{z\in[a^2, b^2+\norm{Q}_2]}|\sqrt{z}-p(z)|,
\end{equation}
where $h=\left\lceil\frac{|i-j|}{\max\{2\beta_a,\beta_q\}} \right\rceil-1$, and $i\neq j$. 

Inequality \eqref{eq:decay-band} establishes a direct link between the sparsity of $X$ and the polynomial approximation property of the square root function over $[a^2, b^2+\norm{Q}_2]$. In particular, we expect a weaker decay in the offdiagonal entries of $X$, when $a^2$ gets closer to $0$, see also Figure~\ref{fig:decay-band}.

When $F\neq I$, and/or $A$ is not symmetric we can not rely on formula \eqref{eq:caresol} to rigorously analyse the sparsity of $X$; on the other hand we expect the offdiagonal decay property in more general settings so we aim at providing a procedure that exploits the approximate banded structure if present. With the latter goal in mind, we will briefly review inexact Newton-Kleinman schemes for \eqref{eq:care}, and then\upd{,} we will analyse their behaviour when incorporating a thresholding mechanism along their iterations.
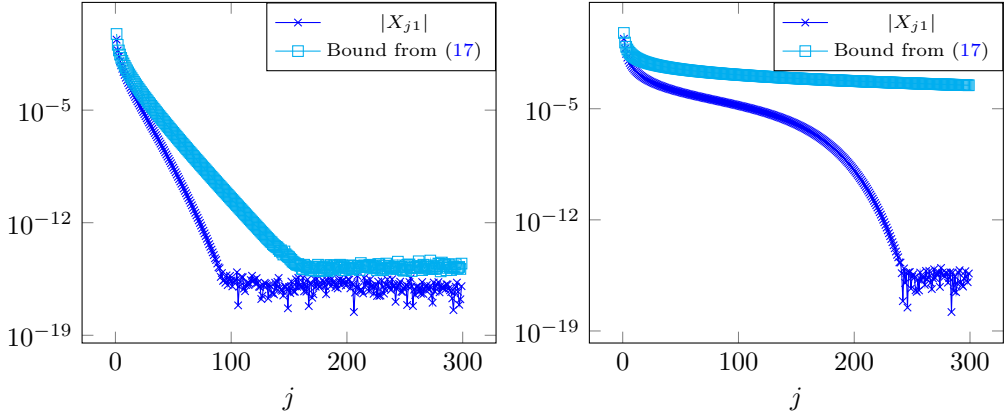
\begin{figure}
\centering
\begin{tikzpicture}
	\begin{semilogyaxis}
		[legend pos = north east, ymax=5e1, xlabel=$j$, legend style={font=\scriptsize, at={(1,1)}}, width=.48\textwidth]
		\addplot[mark=x, blue] table[x index = 0, y index = 1]{banddecay.dat};
		\addplot[mark=square, cyan] table[x index = 0, y index = 2]{banddecay.dat};
		\legend{$|X_{j1}|$, Bound from \eqref{eq:decay-band}};
	\end{semilogyaxis}
\end{tikzpicture}
\begin{tikzpicture}
\begin{semilogyaxis}[legend pos = north east, ymax=5e1, xlabel=$j$, legend style={font=\scriptsize, at={(1,1)}}, width=.48\textwidth]
\addplot[mark=x, blue] table[x index = 0, y index = 3] {banddecay.dat};
\addplot[mark=square, cyan] table[x index = 0, y index = 4] {banddecay.dat};
\legend{$|X_{j1}|$, Bound from \eqref{eq:decay-band}};
\end{semilogyaxis}
	\end{tikzpicture}
\caption{Magnitude of the entries in the first column of the solution $X\in\mathbb R^{300\times 300}$ of a CARE with $A$ diagonal with positive entries, $F=I$, and $Q$ symmetric positive definite and tridiagonal. On the left, it is set $A=$ \texttt{diag(logspace(-1, 0, 300))}, on the right $A=$ \texttt{diag(logspace(-3, 0, 300))}. For both cases, the matrix $Q$ is chosen as in Example~\ref{ex:decay}. The upper bound \eqref{eq:decay-band}, denoted by cyan squares, has been computed via the \texttt{minimax} function of the \texttt{chebfun} toolbox.}\label{fig:decay-band}
\end{figure}

\subsection{Inexact Newton-Kleinman methods}\label{sec:tink}
A well established iterative scheme for solving  \eqref{eq:care} consists in applying the Newton's method to the quadratic equation $\mathcal R(X)=0$. This approach, known in the literature as Newton-Kleinman \upd{(NK)} iteration~\cite{kleinman1968iterative}, yields a sequence denoted as $\{X_k\}_k$, where the $(k+1)$-\upd{st} term is defined as the solution of the Lyapunov equation:

\begin{equation}
A_{cl,k}^\top X_{k+1} + X_{k+1} A_{cl,k}= -X_k F X_k-Q,
\label{newton}
\end{equation}
where the matrix  $A_{cl,k}:= A-F X_k$  is called the \emph{closed-loop matrix at the $k$-th step}.
If the Newton-Kleinman iteration is initialized with a positive semidefinite  matrix $X_0$, such that $(A-F X_0)$ possesses eigenvalues solely within the open left half-plane, then all the  matrices $A_{cl,k}$ are stable, $X_k$ converges quadratically to a solution of \eqref{eq:care}, and the convergence is monotonically nonincreasing (for $k\ge 1$) with respect to the the L\"owner order \cite[Theorem 3.1]{bini2011numerical}.  

In the large-scale setting, the use of iterative solvers for the Lyapunov equations \eqref{newton} has been proposed to enhance the computational efficiency of the Newton-Kleinman iteration; these variants are known as \emph{inexact Newton-Kleinman methods} \cite{feitzinger2009inexact,benner2016inexact}. More formally, these methods replace $\{X_k\}_k$ with a sequence of approximations $\{\widetilde X_k\}_k$ satisfying
\begin{equation}
	\tilde{A}_{cl,k}^\top  \tilde{X}_{k+1} + \tilde{X}_{k+1} \tilde{A}_{cl,k} +   \tilde{X}_{k} F \tilde{X}_{k}+Q = R_{k+1},
	\label{inexactNK}
\end{equation}
where $\tilde{A}_{cl,k}:=A-F\tilde X_k$, and the residual matrices $R_k$ have small norms.
This introduces some inexactness in the Newton's scheme, and makes non trivial to guarantee that $\{\tilde{X}_k\}_k$ retains the convergence properties of $\{X_k\}_k$. For instance, in \cite{benner2016inexact} it has been proposed the use of a line search procedure to ensure a monotone decrease of the norm of the Riccati residuals, and attain a global convergence.

\subsubsection{A novel algorithm tailored to the banded case}
\label{subsec:novel_alg}
In this section we propose and analyse an inexact Newton-Kleinman iteration that takes advantage of the banded structure of the coefficients of the CARE \eqref{eq:care}. The idea is to select a banded starting matrix $\tilde X_0$ and to employ an iterative solver for the Lyapunov equations \eqref{newton} that maintains a sufficient level of sparsity when generating the inexact sequence $\{\widetilde X_k\}_k$.  More precisely, we propose the use of the Generalized Minimal Residual (GMRES) method to approximate the solution of the $(n^2\times n^2)$-linear system
\begin{equation}
	\tilde{\mathcal{A}}_{k} x = b_{k},
	\label{linear_system_lyap}
\end{equation}
where $\tilde{\mathcal{A}}_{k}=\left(I \otimes \tilde{A}_{cl,k}^\top + \tilde{A}_{cl,k} \otimes I \right)$ and $ b_{k} = -\text{vec}(\tilde{X}_{k} F \tilde{X}_{k}+Q)$,
that is equivalent to the Lyapunov equation \eqref{inexactNK}, without the residual term. Solving this system inexactly, and reshaping the approximate solution $\widehat x_{k+1}$, yields  $\widehat  {X}_{k+1} = \text{vec}^{-1}(x_{k+1})\in\mathbb R^{n\times n}$. The key property of this iterative scheme is that if $\widetilde X_k$ is symmetric with bandwidth $s_k$,  and $\overline{it}_k$ iterations of GMRES have been executed starting from a zero initial guess, then $\widehat X_{k+1}$ is again symmetric and its bandwidth is at most $(\overline{it}_k-1)\beta_a+\beta_f+\beta_q+2s_k$.
In addition, any line search strategy that linearly combines $\tilde X_k$ and $\widehat X_{k+1}$, does not increase the bandwidth of the iterate and preserves the symmetry.

Thus, we follow the approach in \cite{benner2016inexact}\footnote{In \cite{benner2016inexact} an extra condition is required on the residual of the Lyapunov equation \eqref{newton}, $i.e.$ $\norm{R_k }_2 \le \eta_k \norm{ \mathcal{R}(\tilde{X}_k) }_2$, with $\eta_k \in (0,1)$. Since this condition is not needed for the convergence of the algorithm, it has been avoided in the current work.} and select  $\lambda_k \in (0,1)$ such that
\begin{equation}
	\norm{ \mathcal{R}((1-\lambda_k) \tilde{X}_k + \lambda_k \widehat {X}_{k+1}) }_2 \le (1-\lambda_k \alpha) \norm{ \mathcal{R}(\tilde{X}_k) }_2,
	\label{cond_linesearch}
\end{equation}
for a certain $\alpha>0$;
we remark that the parameter $\lambda_k$ is obtained by computing one of the roots of a quartic scalar polynomial \cite[equation 3.11]{benner2016inexact}. Then, we set
$\tilde X_{k+1}:=(1-\lambda_k) \widetilde{X}_k + \lambda_k \widehat {X}_{k+1}.
$

We remark that, although the bandwidth grows at most linearly with respect to $\overline{it}_k$ along the outer iterations of Algorithm~\ref{alg_NK}, sparsity might be lost too quickly.  For this reason, we also consider a thresholding mechanism that keeps under control the bandwidth of the sequence $\{\tilde{X}_k\}_k$.  Formally, we introduce the truncation operator $\mathcal{T}_s: \mathbb{R}^{n\times n} \rightarrow \mathbb{R}^{n\times n}$, defined as
\begin{equation*}
	(\mathcal{T}_s(M))_{i,j} = \begin{cases}
		M_{i,j}, & |i-j| \le s,\\
		0, & |i-j| > s,
	\end{cases}
	\label{banded}
\end{equation*}
and we overwrite $\widetilde X_{k+1}$ with  $\mathcal T_s(\widetilde X_{k+1})$, to eventually reduce its bandwidth.
As discussed in the next section, the parameter $s$ used in this thresholding operation will be chosen in an adaptive way along the iterations, to ensure that the convergence properties of the sequence are preserved. Intuitively, when $\widetilde X_k$ has an offdiagonal decay applying $\mathcal{T}_s$ provides a significant reduction of storage and computational costs without dramatically increasing the number of iterations of the Newton scheme, needed to attain a sufficient accuracy. 

We refer to this procedure with the acronym \textsc{tink} that stands for \emph{Truncated Inexact Newton-Kleinman iteration}; the pseudocode of \textsc{tink} is reported in Algorithm~\ref{alg_NK}.

\begin{algorithm}[h]
	\caption{Truncated inexact Newton-Kleinman iteration for Riccati equations}
	\begin{algorithmic}[1]
		\Procedure{tink}{$A$, $F$, $Q$, $\widetilde X_0$, $\mathsf{tol}$, $k_{\max}$, $\mathsf{linesearch}$, $\mathsf{truncation}$}
		\For{$k=0,1,\dots, k_{max}-1$}
		\If{$\norm{ \mathcal{R}(\widetilde X_{k}) }_2 < \mathsf{tol}$}
		\State \textbf{break}
		\EndIf
		\State{Run $\overline{it}_k$ iteration steps of GMRES on  $\tilde{\mathcal A}_k x =-\text{vec}\left( \tilde{X}_{k} F \tilde{X}_{k}+Q\right) $ and obtain $\widehat{X}_{k+1}$}
		\If{$\mathsf{linesearch}$}
		\State{Find $\lambda_k$ such that \eqref{cond_linesearch} holds true}
		\Else
		\State{$\lambda_k \gets 1$}
		\EndIf
		\State{$\tilde X_{k+1} = (1-\lambda_k) \tilde X_k + \lambda_k \widehat{X}_{k+1}$}
		\If{$\mathsf{truncation}$}
		\State Determine $s_k$ such that \eqref{eq:res_care_trunc} and \eqref{eq:res_lyap_trunc} hold true
		\Else
		\State $s_k\gets n$
		\EndIf
		\State{$\tilde X_{k+1} \gets \mathcal{T}_{s_k}(\tilde X_{k+1})$}
		\EndFor
		\State \Return{$\widetilde X_k$}
		\EndProcedure
	\end{algorithmic}
	\label{alg_NK}
\end{algorithm}

\subsubsection{Convergence analysis}

Here we show that for certain choices of the number of inner iterations $\overline{it}_k$ and of the truncation parameters $s_k$, we can ensure  the convergence of the sequence generated by Algorithm~\ref{alg_NK}. We start with the case in which there is no truncation, i.e., $s_k$ is chosen equal to $n$ for all $k\ge 0$. At the end we will clarify under which conditions the results extend to the case with truncation.  

The convergence analysis is based on two principal ingredients: the monotone decrease of the Riccati residuals and the stabilization of the closed-loop matrices in each iteration.

By looking at condition \eqref{cond_linesearch}, we see that the line search step guarantees the decrease of the Riccati residuals.
In absence of a line search procedure, it is possible to ensure just a local convergence of the method, see \cite{feitzinger2009inexact} for further details.

Concerning the stabilization of the closed-loop matrices, we recall Theorem 4.3 in \cite{feitzinger2009inexact} that, under the conditions $\tilde{X}_0$ symmetric and positive semidefinite, $\tilde{A}_{cl,0}$ stable, and
\begin{equation}
	R_{k} \le  Q\qquad \text{for all $k\ge 0$,}
	\label{condition_Rk}
\end{equation}
ensures that for all $k\ge 0$:
\begin{enumerate}
	\item[(i)] the iterate $\widetilde X_{k}$  is well defined, symmetric and positive semidefinite,
	\item[(ii)]  the matrix $\widetilde {A}_{cl,k}= A-F\widetilde X_{k}$ is stable.
\end{enumerate}

Note that, since Algorithm~\ref{alg_NK} preserves the symmetry of the starting point $\tilde X_0$ throughout its iterations, the associated residual matrices $R_k$ are symmetric, and condition~\eqref{condition_Rk} is satisfied whenever $\norm{R_k}_2\le \lambda_{\min}(Q)$. Combining this observation with a convergence bound for the GMRES method, yields the following result.

\begin{theorem}\label{thm:GMRESstable}
Let $\mathcal{W}(L^{-1}A L) \subset \mathbb{C}^{-}$, $\tilde{X}_0$ be symmetric and positive semidefinite such that $\tilde{A}_{cl,0}$ is stable,  $Q$ be symmetric positive definite, and introduce the matrices 

$$\tilde{\mathcal L}_k:=(L\otimes L^\top)\cdot \tilde{\mathcal A}_k\cdot(L^{-1}\otimes L^{-\top})=I \otimes L^\top\tilde{A}_{cl,k}^\top L^{-\top} + L \tilde{A}_{cl,k} L^{-1} \otimes I .$$
 If in each outer iteration of Algorithm~\ref{alg_NK}, without truncation, the number of GMRES  iterations satisfies
\begin{equation}
	\overline{it}_k \ge 2 \log\left((1+\sqrt{2})\frac{\, \kappa(F)\norm{ \tilde{X}_kF\tilde{X}_k+Q }_F }{\lambda_{min}(Q)}\right)\bigg/ \log\left(1 - \frac{\lambda^2_{min}((\tilde{\mathcal{L}}_{k}+\tilde{\mathcal{L}}_{k}^\top)/2)}{\lambda_{max}(\tilde{\mathcal{L}}_{k}\tilde{\mathcal{L}}_{k}^\top)} \right),
	\label{n_iter}
\end{equation}	
 then\upd{,} for all $k\ge 0$ the iterate $\widetilde X_{k}$  is well defined, symmetric and positive semidefinite,
and  the matrix $\widetilde {A}_{cl,k}$ is stable.
	
\end{theorem}

\begin{proof}
%
Let us denote by  $\widehat R_{k+1}:= \tilde{A}_{cl,k}^\top  \widehat{X}_{k+1} + \widehat{X}_{k+1} \tilde{A}_{cl,k} +   \tilde{X}_{k} F \tilde{X}_{k}+Q$ the  residual matrix in step $k$ of Algorithm~\ref{alg_NK}, before applying the line search step. If
 $\overline{it}_k$ iterations of GMRES have been applied on the linear system $\tilde{\mathcal A}_k x=- \mathrm{vec}(\tilde{X}_kF\tilde{X}_k+Q)$ then

 \begin{align}\label{eq:gmres}
\frac{\norm{\widehat R_{k+1}}_F}{\norm{\tilde{X}_kF\tilde{X}_k+Q}_F}&\le \min_{\begin{smallmatrix}\mathrm{deg}(p)=\overline{it}_k\\ p(0)=1\end{smallmatrix}}\norm{p(\tilde{\mathcal A}_k)}_2 = \min_{\begin{smallmatrix}\mathrm{deg}(p)=\overline{it}_k\\ p(0)=1\end{smallmatrix}}\norm{ (L^{-1}\otimes L^{-\top})p(\tilde{\mathcal L}_k) (L \otimes L^{\top})}_2 \nonumber\\
&\le \kappa(F) \min_{\begin{smallmatrix}\mathrm{deg}(p)=\overline{it}_k\\ p(0)=1\end{smallmatrix}}\norm{p(\tilde{\mathcal L}_k)}_2\le (1+\sqrt{2})\kappa(F) \min_{\begin{smallmatrix}\mathrm{deg}(p)=\overline{it}_k\\ p(0)=1\end{smallmatrix}}\max_{z\in\mathcal W(\tilde{\mathcal L}_k)}|p(z)|,
\end{align}
where in the last inequality we have used the Crouzeix-Palencia bound~\cite{crouzeix}. Moreover, if $\tilde{X}_k$ is positive semidefinite (that we can assume by induction), then  
\begin{align*}
\mathcal W\left(L^\top\tilde{A}_{cl,k}^\top L^{-\top}\right)&\subseteq \mathcal W\left(L^\top A^\top L^{-\top}\right)+\mathcal W\left(-L^\top XL\right),\\ \mathcal W\left(L\tilde{A}_{cl,k} L^{-1}\right)&\subseteq \mathcal W\left(L^{-1} AL\right)+\mathcal W\left(-L^\top XL\right),
\end{align*}
implying $\mathcal W(\tilde{L}_k)\subset\mathbb C^{-}$, i.e., the Hermitian part of $\tilde{L}_k$ is symmetric negative definite. Actually, we can be more specific and write
$$\mathcal W(\tilde{\mathcal L}_k)\subset\left\{\Re(z)\le \lambda_{\min}\left((\tilde{\mathcal L}_k+\tilde{\mathcal L}_k^\top)/2\right)\right\}\cap\left\{|z|\le \norm{\tilde{\mathcal L}_k}_2\right\}\subset\mathbb C^{-},$$ see \cite{beckerman06, embree2023}. 
 Bounding the polynomial approximation problem in \eqref{eq:gmres} with Elman's estimate \cite[Theorem 5.4]{elman1982iterative},~\cite[page 776]{beckerman06}, and combining this with \eqref{n_iter}, we get
$$
\norm{\widehat R_{k+1}}_2\le \norm{\widehat R_{k+1}}_F\le \lambda_{\min}(Q).
$$
Note that, applying the line search step does not increase the Lyapunov residual norm above $\lambda_{\min}(Q)$, in view of the inequality
$$
\norm{R_{k+1}}_2\le (1-\lambda_k)\norm{R_k}_2+\lambda_k\norm{\widehat R_{k+1}}_2\le \lambda_{\min}(Q). 
$$  Moreover, the closed-loop matrix at the $(k+1)$\upd{-st} step is given by
$$
\tilde{A}_{cl,k+1}=(1-\lambda_k)(A-F\widetilde X_{k}) +\lambda_k(A-F\widehat X_{k+1}),
$$
that is similar to $(1-\lambda_k)(L^{-1}AL-L^\top\widetilde X_{k}L) +\lambda_k(L^{-1}AL-L^\top\widehat X_{k+1}L) $. The latter is a convex combination of matrices with numerical range in $\mathbb C^{-}$, so $\tilde{A}_{cl,k+1}$ is stable.
Finally, to get the claim it is sufficient to apply Theorem 4.3 in \cite{feitzinger2009inexact}.
\end{proof}

\begin{remark}\label{rem:cg}
Sharper estimates for the quantities $\overline{it}_k$ can be obtained by replacing Elman's bound with tighter upper bounds for the polynomial approximation problem in \eqref{eq:gmres}, such as \cite[Theorem 2.1]{beckerman06}. 
	Further, when $\tilde{\mathcal{A}}_k$ is symmetric negative definite it is convenient to replace GMRES with the Conjugate Gradient  method (applied to $-\tilde{\mathcal A}_k x=-b_k$) as the latter requires only 
$$
	\overline{it}_k = \left\lceil\log\left(\frac{2 \, \kappa(\tilde{\mathcal{A}}_{k}) \norm{R_k }_{\tilde{\mathcal{A}}_{k}}}{\lambda_{min}(Q)} \right)\bigg/\log\left( \frac{\sqrt{\kappa(\tilde{\mathcal{A}}_{k})}-1}{\sqrt{\kappa(\tilde{\mathcal{A}}_{k})}+1} \right)\right\rceil,
$$
inner iterations to ensure the claim of Theorem~\ref{thm:GMRESstable}. However, we remark that $\mathcal A_k$ is symmetric negative definite for any $k\ge 0$ only in very peculiar cases. For instance, this happens when $A$ is symmetric negative definite and $F$ commutes with $\tilde{X}_{k}$ for any $k\ge 0$.

\end{remark}

Theorem~\ref{thm:GMRESstable} ensures that, given a  stabilizing positive semidefinite iterate $\tilde X_k$ and a reasonable number of GMRES iterations, the matrix $\tilde X_{k+1}$ is again stabilizing and positive semidefinite. Following a standard line search argument, as done in \cite[Theorem 10]{benner2016inexact},  it is possible to prove that, when the step length $\lambda_k$ is bounded away from zero, the sequence $\{\tilde X_k\}$ converges to the unique stabilizing solution of the CARE \eqref{eq:care}. 

\begin{corollary}[Application of Theorem 10 in \cite{benner2016inexact} to the current setting]
	Under the same assumptions of Theorem~\ref{thm:GMRESstable}, if $\exists \bar\lambda >0$ such that $\lambda_k  \ge \bar\lambda$ for all $k$, then $\norm{ \mathcal{R}(\tilde X_k) }_F \rightarrow 0$, and
	$\tilde X_k \rightarrow X^*$, where $X^* >0$ is the unique stabilizing solution of the CARE \eqref{eq:care}. 
\label{thm:conv_INK}
\end{corollary}

\begin{proof}
	The proof of this result involves verifying that Theorem~\ref{thm:GMRESstable} ensures the satisfaction of the conditions specified in \cite[Theorem 10]{benner2016inexact}, which, in turn, confirms the claim.
\end{proof}

Let us now consider the effect of the truncation operator $\mathcal{T}_{s_k}(\cdot)$ on the convergence of Algorithm~\ref{alg_NK}. We remark that, the convergence results presented in Corollary~\ref{thm:conv_INK}  hold also after truncation if we maintain the monotone decrease of the Riccati residuals and the stabilizing property of the iterates. The first condition is preserved by requiring that there exists $\zeta\in[0,1)$ such that, for all $k\ge 0$, the parameter $s_k$ is chosen so that the following inequality is satisfied
\begin{equation}
\norm{ \mathcal{R}(\mathcal T_{s_k}((1-\lambda_k)\tilde X_k+\lambda_k\widehat X_k)) }_2 \le (1-\zeta) \norm{ \mathcal{R}(\tilde X_k) }_2.
\label{eq:res_care_trunc}
\end{equation}
The second condition is preserved every time the Lyapunov residual associated with the truncated iterate remains below $\lambda_{\min}(Q)$, i.e., 
\begin{equation}
\norm{\tilde{A}_{cl,k}^\top  \mathcal T_{s_k}((1-\lambda_k)\tilde X_k+\lambda_k\widehat X_k) + \mathcal T_{s_k}((1-\lambda_k)\tilde X_k+\lambda_k\widehat X_k) \tilde{A}_{cl,k} +   \tilde{X}_{k} F \tilde{X}_{k}+Q}_2 \le \lambda_{\min}(Q).
\label{eq:res_lyap_trunc}
\end{equation}
Whenever we select the parameters $s_k$ so that \eqref{eq:res_care_trunc}, and \eqref{eq:res_lyap_trunc} hold, Algorithm~\ref{alg_NK} is ensured to converge. Note that, the choice $s_k=n$ for all $k$ always matches \eqref{eq:res_care_trunc} (for $\zeta=0$), and \eqref{eq:res_lyap_trunc}, and corresponds to the no truncation case.

\begin{corollary}
Under the same assumptions of Corollary \ref{thm:conv_INK}, apart from executing Algorithm~\ref{alg_NK} with truncation,  if \eqref{eq:res_care_trunc} and \eqref{eq:res_lyap_trunc} hold for every $k\ge 0$, then the  sequence $\{\tilde X_k\}_k$ converges to the unique stabilizing positive definite solution of the CARE \eqref{eq:care}.
\end{corollary}

\subsubsection{Details on the implementation of \protect{\textsc{tink}}}
In this section we comment on some aspects of our implementation of \textsc{tink}. First, the algorithm is based on the choice of an initial guess $\tilde{X}_0$ symmetric and positive definite such that $\tilde{A}_{cl,0}$ is stable. In our tests either the matrix $A$ is inherently stable,  or the initial point $\tilde{X}_0 = c I$ results in a stable closed loop matrix $\tilde{A}_{cl,0}$. Here, $c\ge 0$  is selected to minimize the Riccati residual while still ensuring the stabilization of the closed-loop matrix. Another minimization is performed in the line search procedure, where we minimize a quartic polynomial \cite[equation 3.11]{benner2016inexact}. Both tasks are performed using the Matlab function \texttt{fminbnd} within the interval $(0,1)$.

The pseudocode in Algorithm~\ref{alg_NK} \upd{has} nice convergence properties when the assumptions of Theorem~\ref{thm:conv_INK} are satisfied; however, to make the algorithm competitive with state-of-the-art methods we need to combine it with some approximations and greedy strategies. We point out the differences between Algorithm~\ref{alg_NK} and our actual implementation in the following list:

\begin{itemize}
	\item[$(i)$] We do not prescribe a priori the number of GMRES iterations according to \eqref{n_iter} but we run it for at least $5$ iterations, and then  we stop when the Lyapunov residual satisfies
	$\norm{R_k }_2 \le \lambda_{\upd{\min}}(Q).
	$ 
	\item[$(ii)$] The computation of the residual norms $\norm{ \mathcal{R}(X_k) }_2$ and $\norm{R_k }_2$ is replaced with the evaluation of probabilistic upper bounds. More explicitly, we generate $10$ independent standard Gaussian vectors 
	$\{\omega^{i}\}_{i=1}^{10}$, and we compute the right-hand sides of the inequalities:
	\begin{equation*}
		\begin{split}
		\norm{ \mathcal{R}(X_k) }_2 &\le 2 \sqrt{\frac{2}{\pi}} \max_{i=1,\ldots,10} \norm{ (A^\top X_k + X_k A - X_kF X_k + Q) \omega^{i} }_2,\\		
		\norm{R_k }_2 &\le 2 \sqrt{\frac{2}{\pi}} \max_{i=1,\ldots,10} \norm{ (\tilde{A}_{cl,k-1}^\top  \tilde{X}_{k} + \tilde{X}_{k} \tilde{A}_{cl,k-1} +   \tilde{X}_{k-1} F \tilde{X}_{k-1}+Q) \omega^{i} }_2.
		\end{split}
	\end{equation*}
	The latter hold true with probability at least $1-2^{-10}$, see \cite{halko2011finding}.
	\item[$(iii)$] In each Newton iteration the truncation parameter $s_k$ is chosen with a greedy strategy. The process begins with the small initial value $s_k=8$. Then, we verify whether conditions \eqref{eq:res_care_trunc} and \eqref{eq:res_lyap_trunc} are satisfied, by means of the random estimators at point $(ii)$. In the affirmative case, the new iterate $\tilde{X}_{k+1}$ is selected as the truncated one. Otherwise, we set $s_k\gets s_k+5$ and we repeat the procedure until conditions \eqref{eq:res_care_trunc} and \eqref{eq:res_lyap_trunc} are satisfied. 
	\item[$(iv)$] We empirically found that it is convenient to perform line search only at the first iteration, see also the numerical results in Section~\ref{sec:linesearch}. If not stated otherwise, we implicitly assume that, in \textsc{tink}, line search is not executed apart from the first iteration.
\end{itemize} 
Finally, to estimate the typical cost of \textsc{tink} we suppose that:
\begin{itemize}
	\item finding/estimating $\lambda_{\min}(Q)$ costs at most $\mathcal O(n)$,
	\item the number of NK iterations is $k_{\max}$,
	\item the bandwidth of the matrices generated during the execution never exceeds $b\in\mathbb N$.
\end{itemize}
Since each iteration of GMRES increases the bandwidth of the iterate, under the above assumptions we have that each call of GMRES runs for at most $b$ iterations and, in particular, its cost is at most $\mathcal O(n b^2)$. The latter is determined by the $\mathcal O(b)$ matrix-matrix multiplications and the orthogonalization phase of the Arnoldi process. The dominant cost in evaluating the residual norms and applying the truncation operator as described at $(ii)$, and $(iii)$ is given, in the worst case scenario, by evaluating $10$ matrix-vector products with a banded matrix of bandwidth $b$ for $\mathcal O(b)$ times; this also costs $\mathcal O(nb^2)$. Therefore, the overall cost of \textsc{tink} is $\mathcal O(nb^2k_{\max})$.



\subsection{Numerical tests}\label{sec:tink-tests}
Let us showcase the features of \textsc{tink} by means of various numerical tests with different settings.
Throughout this section we denote by $A_0$ the $n\times n$ tridiagonal matrix with $-2$ on the main diagonal and $1$ on its sub and superdiagonal.


\subsubsection{The role of the line search method and of the truncation strategy}\label{sec:linesearch}
In this experiment, we aim to examine the impact of incorporating the truncation strategy and the Line Search algorithm into the proposed method. We fix $n = 2000$, and set $A = A_0$,  $Q = \upd{\mathrm{tridiag}}(0.48\cdot e, e, 0.48\cdot e)$, and $F = L L^\top$ with $L= \upd{\mathrm{tridiag}}(0 \cdot e, e, 0.1\cdot e)$, where $e \in \mathbb{R}^n$ is the vector of all ones. Unless stated otherwise truncation is performed in every NK iteration and we select the stopping tolerance $\mathsf{tol} = 10^{-12}$.  Then, we consider four variants of \textsc{tink}: 
\begin{itemize}
\item the line search algorithm is not employed (No-LS),
\item the line search algorithm is applied only at the first iteration (1-LS),
\item the line search algorithm is applied at all the iterations (LS),
\item the line search algorithm is applied only at the first iteration and the truncation is not employed. (1-LS no trunc).
\end{itemize}

\begin{figure}[h!]	
\centering
 \includegraphics[width=0.45\textwidth]{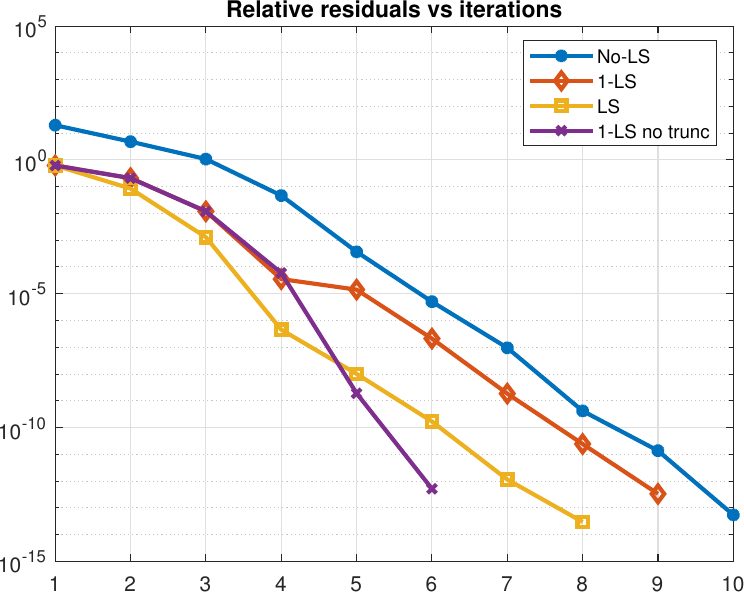}
	\includegraphics[width=0.45\textwidth]{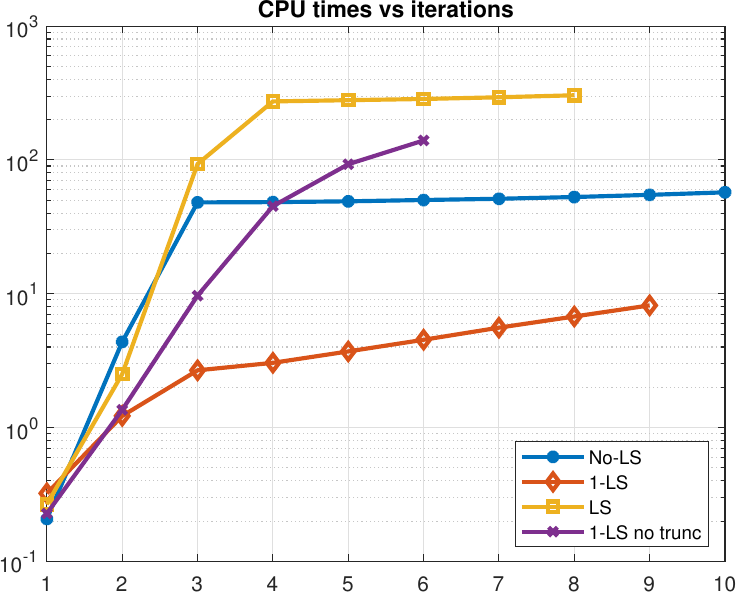}
	\includegraphics[width=0.45\textwidth]{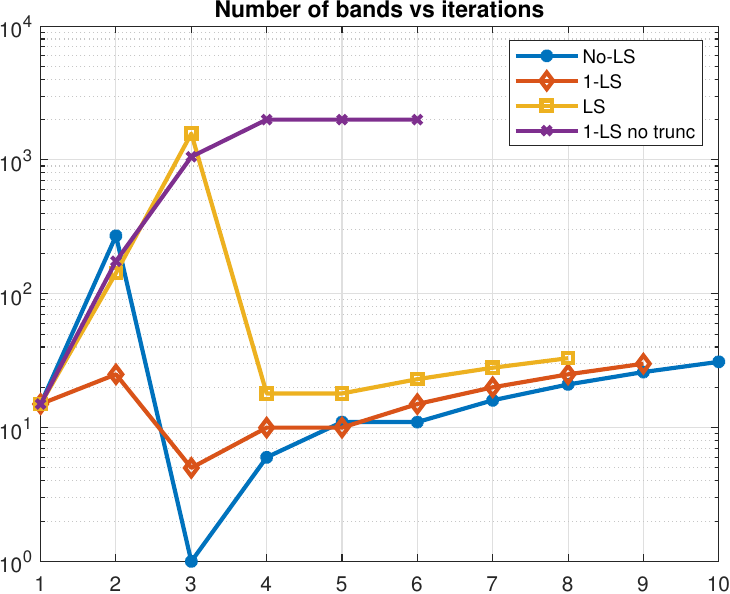}
	\caption{Comparison of the methods in terms of the influence of the line search method and truncation strategy, for $n=2000$.} 
 \label{fig:linesearch}
\end{figure}
The performances of the four methods are displayed in Figure~\ref{fig:linesearch}. The top-left panel shows the beneficial effect of the line search at the first iteration, as it reduces the first residual by an order of magnitude. Among the algorithms that make use of truncation, the LS method exhibits a faster decrease of the residual norm, converging in 8 iterations compared to the 9 and 10 iterations required by 1-LS and No-LS, respectively. However, the LS method necessitates about $400$ seconds to converge, making it slower than the No-LS and 1-LS methods. The fastest method is the 1-LS, which benefits from the initial residual reduction  and the low bandwidth of its iterates. The bottom panel of Figure~\ref{fig:linesearch} indicates that the high computational cost of LS and No-LS is due to a significant increase of the number of bands during the first iterations, whereas the 1-LS always maintains a bandwidth lower than $40$.

Finally, we compare the performances of the truncated and non-truncated versions of 1-LS. During the initial iterations, the residuals are similar, but ultimately, the 1-LS no trunc achieves a faster reduction and converges in 6 iterations. However, this comes at the price of considering much more bands and consuming much more CPU \upd{time}, with the latter being an order of magnitude higher than that of the 1-LS case.

In view of the experimental results of this section, for subsequent numerical tests, we will consistently employ the 1-LS version of the \textsc{tink} algorithm.

\subsubsection{Comparing \textsc{tink} and \textsc{d\&c\_care}}
\label{sec:numer_compar}
We conclude this section by comparing the performances of \textsc{tink} with the divide-and-conquer approach in Algorithm~\ref{alg:dac_care}. We consider a case study where we vary two parameters: the size $n$ of the matrices and the condition number of the quadratic coefficient $F$. More specifically, we set $A=A_0$, $F$ diagonal with entries logarithmically spaced in $\left[\frac{1}{\sqrt{\kappa(F)}}, \sqrt{\kappa(F)}\right]$, and $Q=\mathrm{tridiag}(0.1\cdot e, e, 0.1\cdot e)$. Then, we solve the CARE with the two algorithms for $n= 500\cdot 2^h$, $h=0,\dots, 4$, and $\kappa(F)\in\{1, 10, 100, 1000\}$. In the case $\kappa(F)=1$, i.e. $F=I$, the \textsc{tink} algorithm uses the conjugate gradient method in place of GMRES.

In Table~\ref{tab:nk},  we report the timings and either the bandwidth or the HSS rank of the computed approximate solution. For each choice of $n$, and $\kappa(F)$, we mark in bold the best CPU time among the two algorithms.  In all cases, and for both algorithms, we get relative residuals of the order of $10^{-10}$, so we did not report this data. Note that, the bandwidth of the computed solution grows as $\kappa(F)$ increases; although this growth appears not dramatic on the returned solution, we highlight that the intermediate quantities generated within the Newton-Kleinman iterations have significantly larger bandwidths. This causes the anomalous increase of the CPU time of \textsc{tink} in the case $n=8000$, $\kappa(F)=1000$. Empirically, we have observed that forcing the number of bandwidths to not exceed a prescribed amount (e.g., $150$) in all the intermediate quantities ruins the convergence rate of the Newton's iteration, resulting in even higher computational times. On the opposite, the HSS rank of the approximate solution and the convergence of EKSM employed in Algorithm~\ref{alg:dac_care} are not significantly affected by $\kappa(F)$. As a result, the timings of  Algorithm~\ref{alg:dac_care} are very similar when fixing $n$. By looking at the numerical results in Table~\ref{tab:nk}, we conclude that \textsc{tink} can compete with \textsc{d\&c\_care} for small values of $\kappa(F)$, say $\kappa(F)\le 100$. For poorly conditioned quadratic coefficients, the structure appears to be present in the solution but is less evident, or even absent, in the intermediate terms produced by the Newton-Kleinman scheme. In the latter case one might look for alternative sparsity preserving strategies to limit the bandwidth of intermediate quantities while keeping under control the number of iterations of the method. 

\upd{For completeness, we also assessed the performance of the classical Newton–Kleinman method implemented with the standard MATLAB routine \texttt{lyap} for solving the Lyapunov equations at each iteration. The computational cost of this approach grows cubically with the problem dimension, as expected for dense Lyapunov solvers. In particular,  for large-scale instances the method becomes significantly more expensive than both \textsc{tink} and \textsc{d\&c\_care}. For example, to obtain comparable values of relative residual, when $n=8000$ and $\kappa(F)=1000$, the Newton–Kleinman method based on \texttt{lyap} requires more than $10^3$ seconds, whereas the divide-and-conquer strategy completes in approximately $22$ seconds, resulting in a speed-up of of nearly a factor of $50$. The break-even point with our proposed methods is attained around $n=1000$. 
}

\begin{table}[h!]
	\centering
	\caption{Comparison of the performances of  Algorithm~\ref{alg:dac_care}, and Algorithm~\ref{alg_NK}, when solving the  CAREs considered in Section~\ref{sec:numer_compar}}\label{tab:nk}
	\vspace{.3cm}
	
	\begin{filecontents*}{test_tink_vs_dac.dat}
%% LaTeX2e file `test_tink_vs_dac.dat'
%% generated by the `filecontents' environment
%% from source `paper' on 2026/01/29.
%%
500     0.37712 0.62268 1.1467  3.4096  1.4859e-10      1.6987e-10      1.4249e-10      1.8841e-10      25      30      40      50      1.5671  1.1313  1.1531  1.1745  2.1244e-10      2.1618e-10    2.2736e-10      2.4179e-10      7       7       7       7
1000    0.74097 1.2216  2.4289  3.5534  1.5057e-10      1.7797e-10      1.8708e-10      3.0893e-10      25      30      40      50      2.521   2.3279  2.4211  2.3818  2.6014e-10      1.6139e-10   1.6104e-10      1.7095e-10      14      14      14      14
2000    1.6261  2.414   4.6371  8.0969  1.5155e-10      1.8129e-10      2.0159e-10      5.0223e-10      25      30      40      50      4.8554  4.9127  4.9394  4.9736  2.8096e-10      2.189e-10    1.2989e-10      1.5629e-10      14      14      14      14
4000    2.7471  5.6218  10.77   22.129  1.5204e-10      1.8293e-10      2.0638e-10      1.9228e-10      25      30      40      55      9.7042  10.018  10.574  10.343  2.9081e-10      1.9743e-10  1.6907e-10      1.712e-10       14      14      14      14
8000    5.6205  13.58   25.558  157.79  1.5228e-10      1.8375e-10      2.0827e-10      1.9834e-10      25      30      40      55      20.752  21.582  22.02   22.173  2.9561e-10      1.8787e-10  1.5654e-10      1.7486e-10      14      14      14      14
\end{filecontents*}
\pgfplotstabletypeset[
	every head row/.style={
		before row={
			\toprule
			\multicolumn{9}{c}{\textbf{Performances of \textsc{tink} (Algorithm~\ref{alg_NK})}}\\
			\hline			
			\multicolumn{1}{c|}{}&
			\multicolumn{2}{c|}{$\kappa(F)=1$}&
			\multicolumn{2}{c|}{ $\kappa(F)=10$}&
			\multicolumn{2}{c|}{$\kappa(F)=100$}&
			\multicolumn{2}{c}{$\kappa(F)=1000$}\\
		},
		after row = \midrule,
	},
	columns = {0,1,9,2,10,3,11,4,12}, 
	columns/0/.style = {column name = $n$, column type=c|},
	columns/1/.style = {column name = Time,precision=1,zerofill, fixed, postproc cell content/.style={
					@cell content/.add={$\bf}{$}
				}},
	columns/9/.style = {column name = bandwidth, column type=c|},
	columns/2/.style = {column name = Time,precision=1,zerofill, fixed, postproc cell content/.style={
			@cell content/.add={$\bf}{$}
	}},
columns/10/.style = {column name = bandwidth, column type=c|},
	columns/3/.style = {column name = Time,precision=1,zerofill, fixed},
columns/11/.style = {column name = bandwidth, column type=c|},
	columns/4/.style = {column name = Time,precision=1,zerofill, fixed},
columns/12/.style = {column name = bandwidth, column type=c},
every row 0 column 5/.style={
	postproc cell content/.style={
		@cell content/.add={$\bf}{$}
	}
},
every row 1 column 5/.style={
	postproc cell content/.style={
		@cell content/.add={$\bf}{$}
	}
},
every row 2 column 5/.style={
	postproc cell content/.style={
		@cell content/.add={$\bf}{$}
	}
}
	]{test_tink_vs_dac.dat}\vspace{.3cm}
	
		\pgfplotstabletypeset[
	every head row/.style={
		before row={
			\toprule
			\multicolumn{9}{c}{\textbf{Performances of \textsc{d\&c\_care} (Algorithm~\ref{alg:dac_care})}}\\
			\hline			
			\multicolumn{1}{c|}{}&
			\multicolumn{2}{c|}{$\kappa(F)=1$}&
			\multicolumn{2}{c|}{ $\kappa(F)=10$}&
			\multicolumn{2}{c|}{$\kappa(F)=100$}&
			\multicolumn{2}{c}{$\kappa(F)=1000$}\\
		},
		after row = \midrule,
	},
	columns = {0,13,21,14,22,15,23,16,24}, 
	columns/0/.style = {column name = $n$, column type=c|},
	columns/13/.style = {column name = Time,precision=1,zerofill, fixed},
	columns/21/.style = {column name = HSS rank, column type=c|},
	columns/14/.style = {column name = Time,precision=1,zerofill, fixed},
	columns/22/.style = {column name = HSS rank, column type=c|},
	columns/15/.style = {column name = Time,precision=1,zerofill, fixed},
	columns/23/.style = {column name = HSS rank, column type=c|},
	columns/16/.style = {column name = Time,precision=1,zerofill, fixed, postproc cell content/.style={
			@cell content/.add={$\bf}{$}
	}},
	columns/24/.style = {column name = HSS rank, column type=c},
	every row 1 column 5/.style={
		postproc cell content/.style={
			@cell content/.add={$\bf}{$}
		}
	},
	every row 3 column 5/.style={
		postproc cell content/.style={
			@cell content/.add={$\bf}{$}
		}
	},
	every row 4 column 5/.style={
		postproc cell content/.style={
			@cell content/.add={$\bf}{$}
		}}
	]{test_tink_vs_dac.dat}	
\end{table}

  \section{Numerical tests on feedback control applications}
\label{sec:feedback_appl}
This section illustrates the application of the proposed algorithms to the optimal control problems introduced in Section~\ref{sec:intro-appl}. In the first numerical experiment, we address the optimal control of a nonlinear PDE, specifically the Allen-Cahn equation. In the second experiment, we focus on the control of a system of interacting particles governed by the Cucker-Smale model.

\subsection{Allen-Cahn equation}\label{sec:allehn-cahn-1d}
Let us consider the following nonlinear Allen-Cahn PDE with homogeneous Neumann boundary conditions:
\begin{equation}
\left\{ \begin{array}{l}
\partial_t y(t,x) = \sigma \partial_{xx} y(t,x) +  y(t,x) (1-y(t,x)^2) + u(t,x),  \\
 y(0,x)=y_0(x),
\end{array} \right.
\label{AC}
\end{equation}
with $x \in [-L,L]$ and $t \in (0,+\infty)$ and the following cost functional
$$
\tilde{J}_{y_0}(u) = \int_0^{\infty} dt  \int_{-L}^L dx \;(|y(t,x)|^2 +  \tilde{\gamma} |u(t,x)|^2).
$$
Approximating the PDE by finite difference schemes with stepsize $\Delta x$, we obtain the ODEs system in form 
$$
\dot{y}(s)=A(y) y(s) + Bu(s),
$$
and the discretized cost functional reads
$$
J_{y_0}(u) = \int_0^{\infty}  \Delta x \, y(t)^\top y(t) +  \gamma  \, u(t)^\top u(t) \, dt \,,
$$

with
$$
A(y) = \sigma A^{\Delta x}_0 + I_n - diag(y \odot y),\quad y \in \mathbb{R}^n, \quad B = I_n ,
$$
where $\odot$ refers to the Hadamard product, $I_n \in \mathbb{R}^{n \times n}$ is the identity matrix, $\gamma = \tilde{\gamma} \Delta x$ and $A^{\Delta x}_0$ is the tridiagonal matrix arising from the discretization of the Laplacian with Neumann boundary conditions with stepsize $\Delta x$.
We fix $\tilde{\gamma} = 0.1$, $\sigma = 10^{-3}$ and $y_0 = [\sin(\pi x_i)]_{i=1}^n$, with $x_i =-L+ (i-1) \Delta x$. 

The feedback control \eqref{control_sdre_feed} is derived through solving the SDREs~\eqref{sdre}, as described in Section \ref{sec:intro-appl}, during the time integration of the dynamical system. Specifically, the system is integrated using the Matlab function \texttt{ode15s}, which is designed for stiff differential equations.

For this numerical test, we employ the \textsc{tink} algorithm to carry on the SDRE approach.
Since the closed loop matrix $A(x)-\gamma^{-1}X(x)$ is symmetric for all $x\in \mathbb{R}^n$, the linear system \eqref{linear_system_lyap} arising from the the Newton-Kleinman iterations is solved with the Conjugate Gradient method.

Note that, the exact solution of the SDRE \eqref{sdre} for a given $x \in \mathbb{R}^d$ is given by formula \eqref{eq:caresol} with a rescaling term: 
\begin{equation}
	X(x)= \gamma \left(\sqrt{A(x)^2+\gamma^{-1} Q}+A(x)\right).
	\label{sdre_exact}
\end{equation}
So, as competitor, we consider the evaluation of formula \eqref{sdre_exact} in dense arithmetic, where the matrix square root is computed using the Matlab function \texttt{sqrtm}.

In Table \ref{table_AC}, we compare the performances of the \textsc{tink} algorithm with the use of the exact formula \eqref{sdre_exact} labeled with "\texttt{sqrtm}". The term "CPU per control" denotes the average time required to compute a feedback control as described by \eqref{control_sdre_feed}, specifically for the resolution of an SDRE. The term "Average bandwidth" refers to the average number of diagonals in the approximate solution of the SDREs.

The \textsc{tink} algorithm exhibits an increasing gain with respect to \texttt{sqrtm} as $n$ rises, achieving a speed-up factor of nearly 16 when $n=2000$; this is due to the exceptionally low average bandwidth. Concurrently,  we get the same total costs up to the reported digits. 

Finally, in Figure \ref{fig:AC}, we present the time evolution of the uncontrolled dynamics (left panel) and the controlled dynamics using the SDRE solved with \textsc{tink} (right panel) for $n=500$. Observe that the uncontrolled solution converges to a stable equilibrium, with values of 1
 when the initial condition is positive and -1 otherwise. Conversely, the SDRE control successfully drives the solution to the unstable equilibrium $\overline{y}=0$ within the initial time instances. The solution controlled via the exact formula closely resembles the one depicted in the left panel of Figure \ref{fig:AC} and is therefore not shown.

\begin{table}[hbht]
\centering
\begin{tabular}{c|ccc|cc}    

        & \multicolumn{3}{c|}{\textsc{tink}}   & \multicolumn{2}{c}{\texttt{sqrtm}} \\
$n$   & CPU per control & Total cost & Average bandwidth  & CPU per control & Total cost
\\\hline
500    & 0.0108 &   10.386514 & 3  &  0.0261  &   10.386514\\
1000 &  0.0234 &  10.386514 & 3 &  0.1855 & 10.386514\\
2000 &   0.0807 &   10.386514 & 8 &   1.3373 &  10.386514\\
 \end{tabular}
  \caption{Comparison of the performances in the computation of the optimal control via \textsc{tink} algorithm and via the exact formula \eqref{sdre_exact} for different dimensions $n$.}
 \label{table_AC}
\end{table}

\begin{figure}[htbp]	
\centering
 \includegraphics[width=0.45\textwidth]{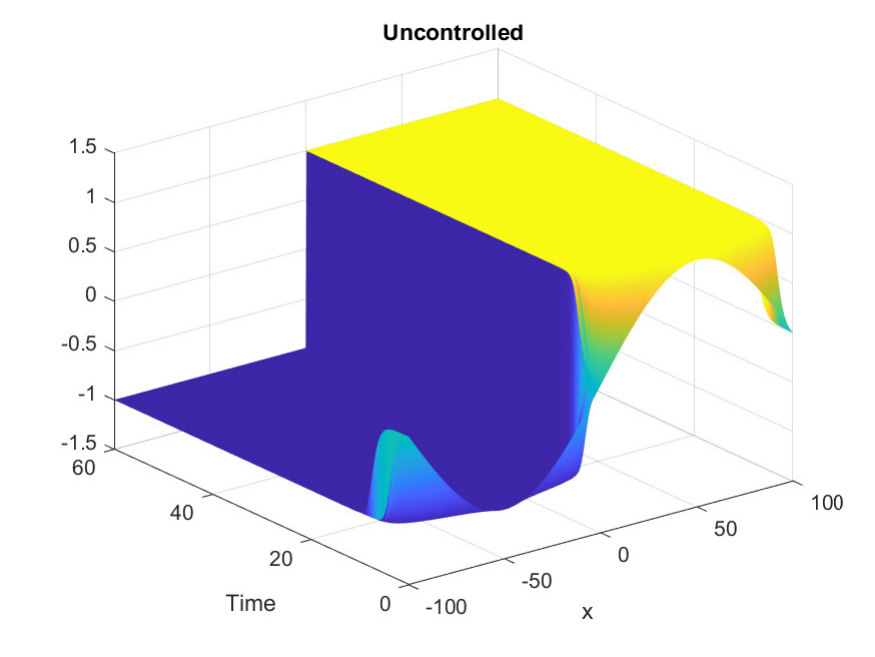}
	\includegraphics[width=0.45\textwidth]{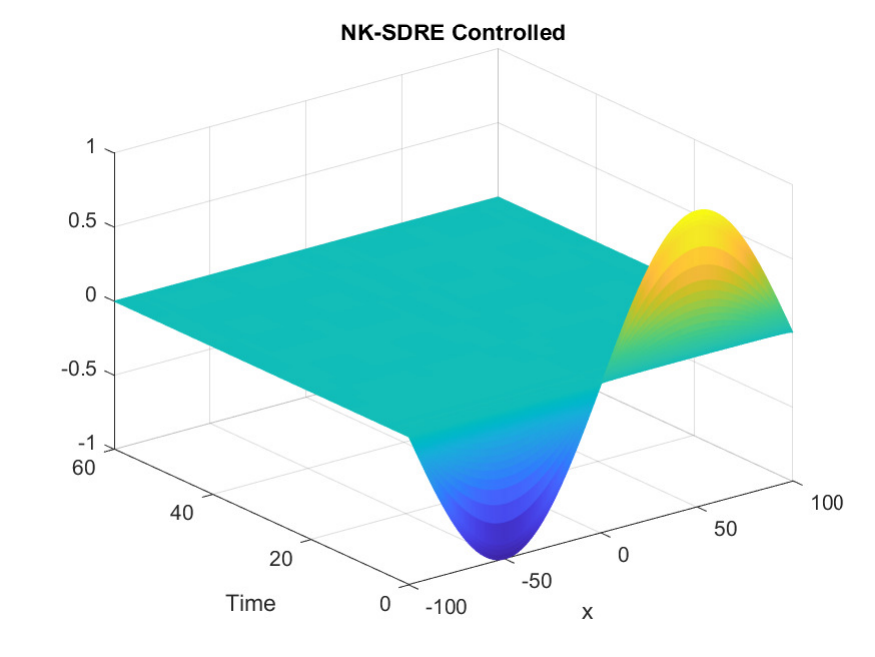}
	\caption{Uncontrolled solution (left) and controlled solution via SDRE solved with \textsc{tink} (right) with $n=500$.} 
 \label{fig:AC}
\end{figure}

\subsection{Cucker-Smale model}
Given $N_a$ interacting agents,
we consider the dynamical system governed by the Cucker-Smale model as follows:
\begin{equation*}
\begin{bmatrix} \dot{y} \\ \dot{v}  \end{bmatrix} = A(y)\begin{bmatrix} y \\ v  \end{bmatrix}+Bu= \begin{bmatrix} 0_{N_a} & I_{N_a} \\ 0_{N_a} & \mathcal{A}_{N_a}(y)  \end{bmatrix} \begin{bmatrix} y \\ v  \end{bmatrix} + \begin{bmatrix} O_{N_a} \\ I_{N_a}  \end{bmatrix} u,
\end{equation*}
with the interaction matrix defined by
$$
\left[ \mathcal{A}_{N_a}(y) \right]_{i,j}= \begin{cases} -\frac{1}{N_a} \sum_{k \neq i} K(y_i,y_k) & if \;i=j, \\
\frac{1}{N_a} K(y_i,y_j) & otherwise,
						 				 \end{cases}
$$
where
$$
K(y_i,y_j)=\frac{1}{1+| y_i-y_j |^2} \,.
$$
Let $y_v = [y; v] \in \mathbb{R}^n$, with $n=2 N_a$.
Our objective is to drive all positions and velocities to zero, while minimizing the following cost functional:
$$
J(y(\cdot),v(\cdot),u(\cdot))=  \int_0^\infty  y(s)^\top Q_{N_a} y(s) + v(s)^\top  Q_{N_a} v(s)  +  u(s)^\top R u(s)\, ds,
$$
where $Q_{N_a}=R=\frac{1}{N_a} I_{N_a} \in \mathbb{R}^{N_a \times N_a}$.
Due to the block structure of the problem, we consider the solution of the SDRE \eqref{sdre} in block form, $i.e.$,
$$
X(y_v) = \begin{bmatrix} X_{1,1}(y_v) & X_{1,2}(y_v) \\ X_{2,1}(y_v) & X_{2,2}(y_v) \end{bmatrix}.
$$
The SDRE is then equivalent to solving the following system of matrix equations:
$$
\begin{cases}
-X_{1,2}(y_v) R^{-1} X_{2,1}(y_v) + Q_{N_a}=0_{N_a}, \\
X_{1,1}(y_v)+\mathcal{A}_{N_a}(y) X_{1,2}(y_v) - X_{1,2}(y_v) R^{-1} X_{2,2}(y_v)=0_{N_a},\\
X_{1,2}(y_v)+X_{2,1}(y_v)+\mathcal{A}_{N_a}(y) X_{2,2}(y_v) + X_{1,2}(y_v) \mathcal{A}_{N_a}(y_v) - X_{2,2}(y_v) R^{-1} X_{2,2}(y_v)+ Q_{N_a}=0_{N_a},\\
X_{1,2}(y_v) = X_{2,1}(y_v).
\end{cases}
$$
We immediately observe that since $R$ is a constant identity matrix, the solution to the first equation is
\begin{equation*}
X_{1,2}(y_v) = X_{2,1}(y_v) = \sqrt{R Q_{N_a}} = \frac{1}{N_a} I_{N_a}.
\end{equation*}
Substituting this solution into the third equation yields the following SDRE for the block $X_{2,2}(y_v)$:
\begin{equation}
\mathcal{A}_{N_a}(y) X_{2,2}(y_v) + X_{2,2}(y_v) \mathcal{A}_{N_a}(y) - X_{2,2}(y_v) R^{-1} X_{2,2}(y_v)+ \frac{3}{N_a} I_{N_a}= 0_{N_a}.
\label{sdre_na}
\end{equation}
For the purpose of implementing the feedback control law \eqref{control_sdre_feed}, we note that the computation of the block $X_{1,1}(y_v)$ is unnecessary, since the control is given by
$$
u(y_v) = -R^{-1} B^{\top} X(y_v)y_v=-y-R^{-1}X_{2,2}(y_v)v.
$$
The optimal control problem is ultimately reduced to approximating the SDRE \eqref{sdre_na}, where the matrix $\mathcal{A}_{N_a}(y)$ exhibits a quasiseparable structure, and both the constant term and the quadratic coefficient are given by constant identity matrices.
To ensure an off-diagonal decay of the matrix $\mathcal{A}_{N_a}$, we first sort the position variables. Subsequently, we solve the corresponding SDRE \eqref{sdre_na} and compute the associated feedback law. The control vector is then reordered to match the original configuration of the positions.

For each test, we consider a random initial condition $y_v(0)\in \mathbb{R}^n$, where each entry lies within the interval $[0,1]$. In Table \ref{table_cucker}, we compare the performance of the \textsc{d\&c\_care} algorithm both with and without the inclusion of the aforementioned sorting procedure. In line with the previous numerical test, the term "CPU per control" refers to the average time required to compute a feedback control, specifically for solving the SDRE. Meanwhile, "Average HSS rank" denotes the mean HSS rank of the SDRE solutions throughout the integration process. We observe that incorporating the sorting procedure results in a reduction of the average HSS rank, which consequently decreases CPU time. This effect is illustrated more clearly in Figure \ref{fig:CS_hssrank}, where the HSS ranks of the solution to the SDRE at different time instances are depicted. It is noteworthy that the sorted technique begins with a rank of 3, which decreases to 2 after time 5. In contrast, the absence of the initial sorting leads to a rank of 6, which diminishes over the course of time integration. Finally, we emphasize that the CPU time per control scales linearly, making this approach both efficient and scalable for solving high-dimensional problems.

With $n=500$ fixed, Figure \ref{fig:CS1} illustrates the evolution of the position variables (left panel) and velocity variables (right panel) for the uncontrolled dynamics. It is observed that the velocities reach consensus, converging to a value close to the average of the initial velocities (approximately 0.5325). Due to the presence of positive velocities, the positions increase over time. Conversely, Figure  \ref{fig:CS2} depicts the behaviour of the position and velocity variables under the controlled dynamics. Here, both variables converge to zero, providing a close approximation of the optimal control solution.

\begin{table}[hbht]
\centering
\begin{tabular}{c|cc|cc}    

        & \multicolumn{2}{c|}{\textsc{d\&c\_care} with sorting} & \multicolumn{2}{c}{\textsc{d\&c\_care} without sorting} \\
$n$   & CPU per control &  Average HSS rank  & CPU per control & Average HSS rank
\\\hline
500    & 0.1678 &     2.54  & 0.1704  & 3.94   \\
1000 & 0.4172    & 2.55  &   0.4587 &  3.96 \\
2000 & 0.7440    & 2.53  &  0.9812  & 3.95 \\ 
 \end{tabular}
  \caption{Performances in the computation of the optimal control via \textsc{d\&c\_care} algorithm with and without ordering the position variables for different dimensions $n$.}
 \label{table_cucker}
\end{table}

\begin{figure}[htbp]	
\centering
 \includegraphics[width=0.45\textwidth]{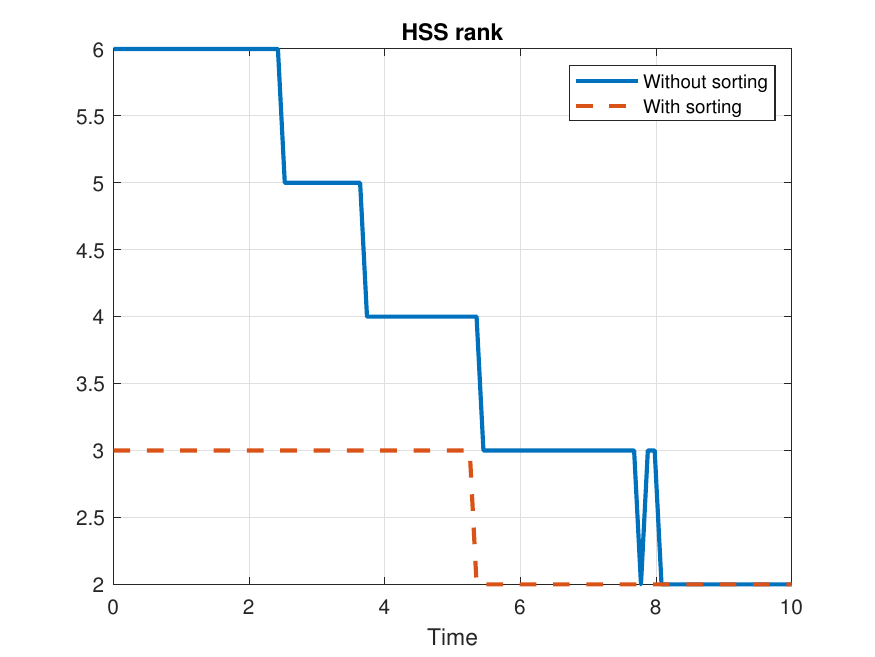}
	\caption{The HSS ranks of the solution to the SDRE \eqref{sdre_na} at various time instances for both the sorted and unsorted algorithms with $n=2000$.} 
 \label{fig:CS_hssrank}
\end{figure}

\begin{figure}[h]	
\centering
 \includegraphics[width=0.45\textwidth]{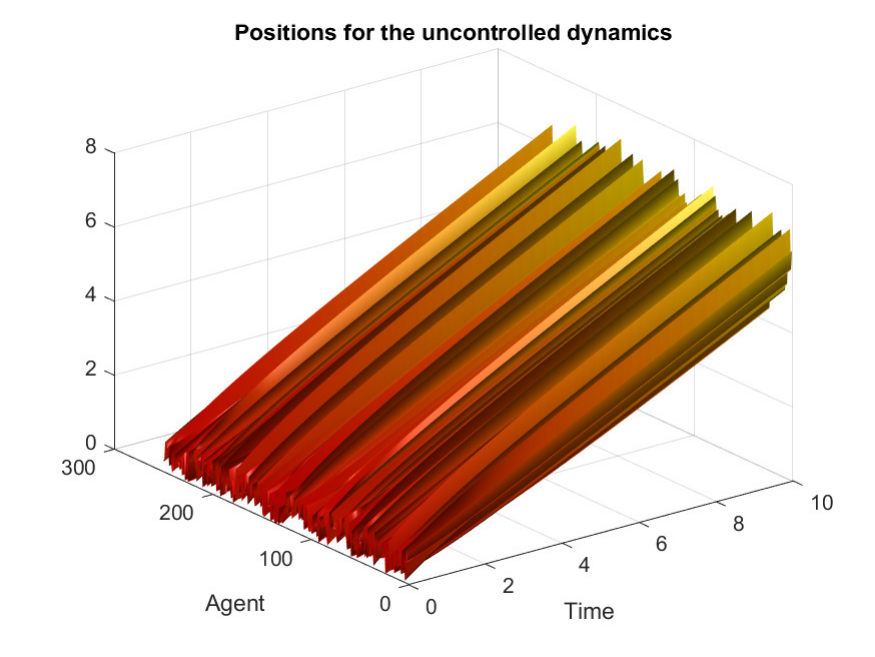}
	\includegraphics[width=0.45\textwidth]{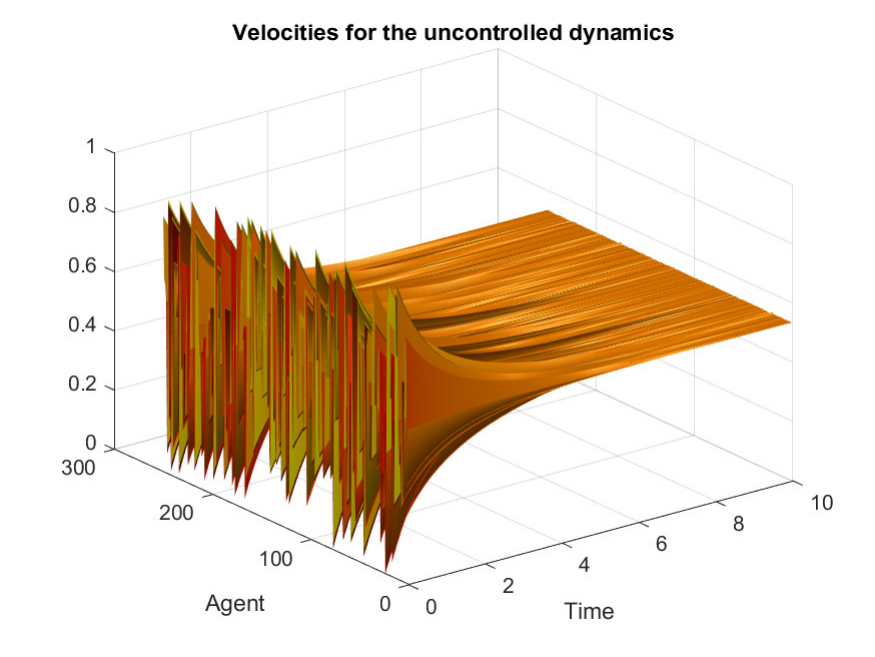}
	\caption{Uncontrolled solution for the positions (left) and for the velocity (right) with $n=500$.} 
 \label{fig:CS1}
\end{figure}

\begin{figure}[h]	
\centering
 \includegraphics[width=0.45\textwidth]{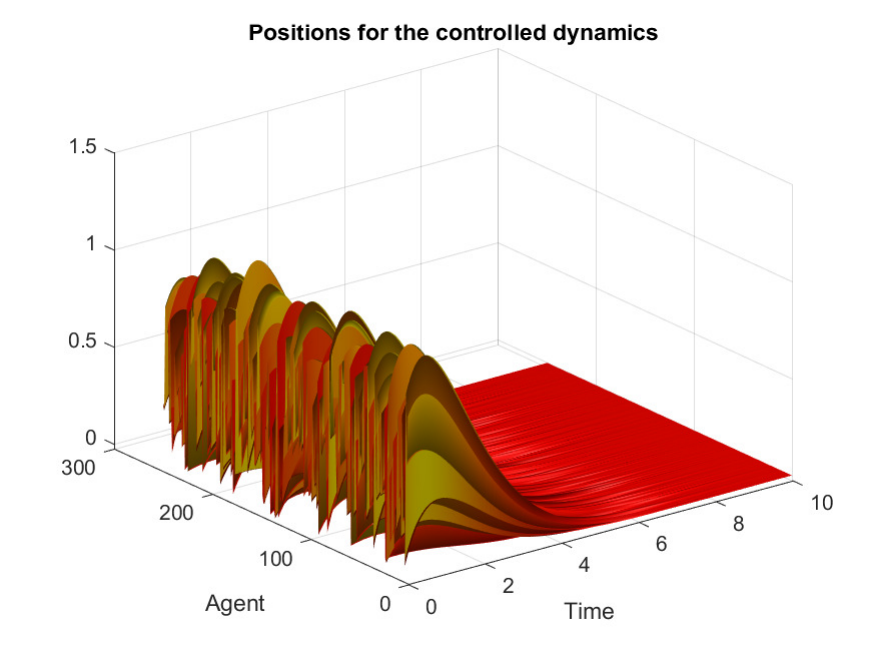}
	\includegraphics[width=0.45\textwidth]{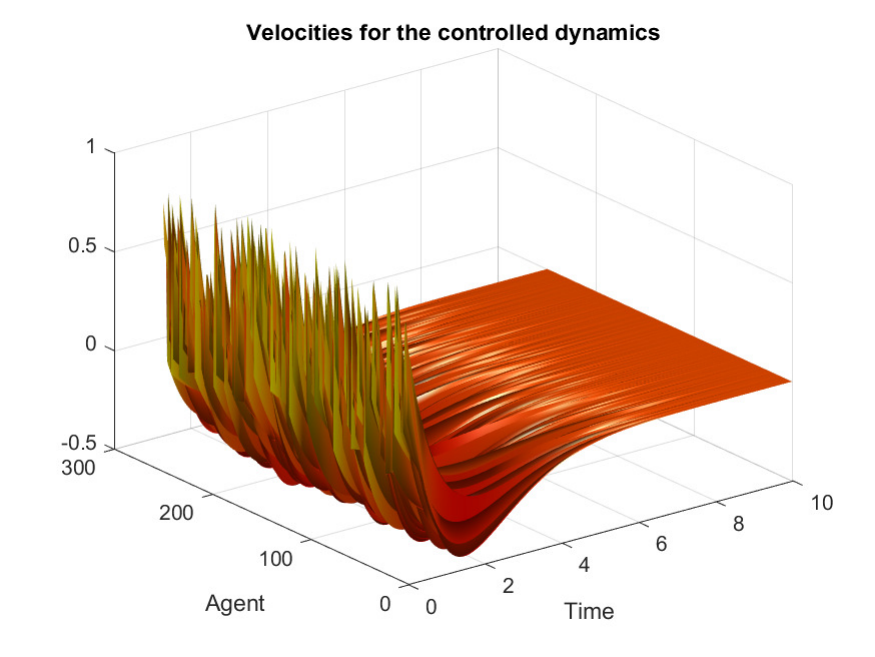}
	\caption{Controlled solution for the positions (left) and for the velocity (right) via SDRE solved with \textsc{d\&c\_care} with $n=500$.} 
 \label{fig:CS2}
\end{figure}
\upd{
\subsection{Some numerical tests in 2D}
The techniques developed in this manuscript can be exploited to mitigate the curse of dimensionality in certain 2D control problems. We illustrate this for operators with a Laplace-like sparsity pattern. In this setting, neither the matrix coefficients nor the solution of \eqref{eq:care} are (numerically) quasiseparable; however, it is still possible to provide a structured representation of the solution and to design numerical procedures that are faster than dense solvers. The key idea is to use well-established graph-based reorderings to reveal subproblems with either low-rank or quasiseparable solutions.

As a working example, consider the case where 
\begin{equation}\label{eq:2d-laplace}
	A=(n+1)^2\left(I_n\otimes \mathrm{tridiag}(1,-2,1) + \mathrm{tridiag}(1,-2,1)\otimes I_n\right) 
\end{equation}
is the discretization of the 2D Laplace operator with Dirichlet boundary conditions, $F$ and $Q$ are diagonal, and the parameter $n$ is such that operating with dense $n\times n$ matrices is affordable, while forming dense $n^2\times n^2$ matrices must be avoided. By applying a nested dissection ordering with a separator set of cardinality $n$, we obtain an equivalent CARE whose coefficient matrix $A$ has the block structure
$$
A=\begin{bmatrix}
	A_{11}&0&A_{13}\\
	0&A_{22}&A_{23}\\
	A_{31}&A_{32}&A_{33}
\end{bmatrix}\in\mathbb R^{n^2\times n^2},
$$
where $[A_{31}\ A_{32} ] = [A_{13}^\top\ A_{23}^\top ]\in\mathbb R^{n\times(n^2-n)}$; additionally, $Q$ and $F$ remain diagonal. The relatively small size of the separator implies that the reordered matrix is within a moderate ``rank distance'' from a block diagonal form. In the spirit of our divide-and-conquer algorithm, we write $X=X_0+\delta X$, where $X_0$ is the solution of the Riccati equation associated with the block diagonal part of $A$, and $\delta X$ satisfies
$$
(A-X_0F)\delta X+\delta X(A-FX_0)-\delta XF\delta X+\delta AX_0+X_0\delta A=0,
$$
with
$$
\delta A = \begin{bmatrix}
	0&0&A_{13}\\
	0&0&A_{23}\\
	A_{31}&A_{32}&0
\end{bmatrix}.
$$
The matrix $X_0$ is $3\times 3$ block diagonal, and each diagonal block, denoted by $X_{ii}^{(0)}$, is the solution of the Riccati equation $A_{ii}X_{ii}^{(0)}+ X_{ii}^{(0)}A_{ii}-X_{ii}^{(0)}F_{ii}X_{ii}^{(0)}+Q_{ii}=0$. Since the size of the third diagonal block is small, $X_{33}^{(0)}\in\mathbb R^{n\times n}$ is computed via a dense solver. We then apply the reverse Cuthill-McKee reordering to both $A_{11}$ and $A_{22}$ (and their associated CAREs) to make these blocks as banded as possible, before compressing them into the HSS format. Empirically, we observe that this yields off-diagonal blocks of significantly lower rank compared to compressing $A$ directly. 

The HSS representations of the (permuted) $X_{11}^{(0)}$ and $X_{22}^{(0)}$ are then retrieved using Algorithm~\ref{alg:dac_care}. Finally, we compute a low-rank approximation of $\delta X$ via Algorithm~\ref{alg:EKSM}; this requires efficient matrix-vector multiplications and linear system solves with the matrix $A-X_0F$. Matrix-vector multiplications are handled via the relation $(A-X_0F)v=Av - X_0(F v)$, leveraging the sparsity of $A$ and the HSS structure of $X_0$ and $F$. Regarding the linear system, we consider the factorization:
$$
A-X_0F=\begin{bmatrix}
	(A_{11}-X_{11}^{(0)}F_{11})&0&0\\
	0&(A_{22}-X_{22}^{(0)}F_{22})&0\\
	0&0&I_n	
\end{bmatrix}\begin{bmatrix}
	I&0&(A_{11}-X_{11}^{(0)}F_{11})^{-1}A_{13}\\
	0&I&(A_{22}-X_{22}^{(0)}F_{22})^{-1}A_{23}\\
	A_{31}&A_{32}&A_{33}-X_{33}^{(0)}F_{33}
\end{bmatrix}.
$$    
This implies that evaluating $(A-X_0F)^{-1}v$ is equivalent to solving a block diagonal linear system with (permuted) HSS diagonal blocks and a sparse linear system. Moreover, this procedure can be applied recursively to compute $X_{11}^{(0)}$ and $X_{22}^{(0)}$ instead of calling Algorithm~\ref{alg:dac_care} immediately. This reduces both the size and the HSS rank of the CAREs solved by the divide-and-conquer algorithm, which in some cases results in CPU time savings. However, increasing the recursion depth causes the procedure to spend more time on HSS compression of the matrices $A_{ii}-X_{ii}^{(0)}F_{ii}$ and on calls to Algorithm~\ref{alg:EKSM}; thus, the optimal recursive depth $\ell$ is problem-dependent. The complete procedure is detailed in Algorithm~\ref{alg:care_2d}. In practice, the solution is returned in a structured format that keeps both the HSS representation of the first term and the low-rank factorization of the second.
\begin{algorithm}
	\caption{CARE solver in the case where $A$ is Laplace-like sparse}\label{alg:care_2d}
	\begin{algorithmic}[1]
		\Procedure{d\&c\_care\_2D}{$A,F,Q, \ell$}
		\State $A\gets \Pi_{nd}A\Pi_{nd}^\top$, $F\gets\Pi_{nd}F\Pi_{nd}^\top$, $Q\gets\Pi_{nd}Q\Pi_{nd}^\top$\Comment{nested dissection}
		\State $X_3^{(0)}\gets\textsc{dense\_care}(A_{33}, F_{33}, Q_{33})$
		\If{$\ell =0$}
		\State
		$A\gets \Pi_{cm}A\Pi_{cm}^\top$, $F\gets\Pi_{cm}F\Pi_{cm}^\top$, $Q\gets\Pi_{cm}Q\Pi_{cm}^\top$\Comment{reverse Cuthill-McKee}
		 \State$X_{11}^{(0)}\gets\textsc{d\&c\_care}(A_{11},F_{11},Q_{11})$,\quad $X_{22}^{(0)}\gets\textsc{d\&c\_care}(A_{11},F_{22},Q_{22})$  
		\Else
		\State $X_{11}^{(0)}\gets\textsc{d\&c\_care\_2D}(A_{11},F_{11},Q_{11},\ell-1)$
		\State $X_{22}^{(0)}\gets\textsc{d\&c\_care\_2D}(A_{11},F_{22},Q_{22},\ell-1)$ 
		\EndIf
		\State Compress $A_{ii}-X_{ii}^{(0)}F_{ii}$ in the HSS format for $i=1,2$
		\State Find a low-rank factorization $\delta AX_0+X_0\delta A= UDU^\top$
		\State $\delta X\gets\textsc{low\_rank\_care}(A-X_0F,F, U, D)$
		\If{$\ell =0$}
		\State \Return $\Pi_{cm}^\top\Pi_{nd}^\top\left( X_0+ \delta X\right)\Pi_{nd}\Pi_{cm}$
		\Else
		\State \Return $\Pi_{nd}^\top\left( X_0+ \delta X\right)\Pi_{nd}$
		\EndIf
		\EndProcedure
	\end{algorithmic}
\end{algorithm}

\subsubsection{Testing the Riccati solver}\label{sec:2d-laplace}
Let us consider solving the CARE with $A$ given as in \eqref{eq:2d-laplace}, and $F=Q=I_{n^2}$. We test the performance of Algorithm~\ref{alg:care_2d} for $n\in\{100,150,200\}$, and recursion depth $\ell\in\{0,1,2\}$. In this test, we set $n_{\min}=500$ and the tolerance for the compression of the low-rank blocks to $10^{-10}$; moreover, Algorithm~\ref{alg:EKSM} is stopped when the relative residual of the approximant is below $10^{-5}$. For $n\in \{100,150\}$ the residual of the computed solution is of the order of $10^{-6}$, while for $n=200$ it is of the order of $10^{-4}$; there are no significant changes when $\ell$ varies. In Table~\ref{tab:2D-laplace} we report the total CPU time, and the percentage of the CPU time spent on the three main tasks: the calls to Algorithm~\ref{alg:EKSM} ($T_{lr}^{\%}$), the calls to Algorithm~\ref{alg:dac_care} ($T_{d\&c}^{\%}$), and the factorizations and HSS compression of the matrices $A_{ii}-X_{ii}^{(0)}F_{ii}$ ($T_{hss}^{\%}$). The results demonstrate a significant advantage of our approach with respect to a dense solver, which, for coefficients of size larger than $10^4$, requires at least thousands of seconds, e.g. see the numerical results in Table~\ref{tab:hss-complexity1}. The computational cost scales less than quadratically with respect to the matrix size. As the number of recursion levels increases, the cost of the calls to Algorithm~\ref{alg:dac_care} decreases while the majority of the CPU time is spent on solving equations with low-rank right-hand side, and on compressing intermediate results in the HSS format. The latter task is performed via \cite[Algorithm 1]{xia10}; a further speed-up could be obtained by means of compression techniques based on fast matrix-vector multiplications, like the algorithm in \cite{levitt2024linear}. 
}

\begin{table}[h!]
	\centering
	\caption{Timings of Algorithm~\ref{alg:care_2d} when solving the  CAREs involving the 2D Laplace operator considered in Section~\ref{sec:2d-laplace}, for various matrix sizes and recursion depths. The table reports the overall cpu time and the percentage spent on the low-rank solver ($T^{\%}_{lr}$), the divide-and-conquer solver ($T^{\%}_{d\&c}$), and the hss compressions ($T^{\%}_{hss}$).}\label{tab:2D-laplace}
	\vspace{.3cm}
	
	\begin{filecontents*}{test_riccati_2D.dat}
%% LaTeX2e file `test_riccati_2D.dat'
%% generated by the `filecontents' environment
%% from source `paper' on 2026/01/29.
%%
0       38.046  72.531  13.336  6.2925  5.2817e-06      150.67  77.906  11.043  5.7889  6.5295e-06 492.98  80.258  8.4756  5.5968  0.00014535
1       35.682  39.894  31.399  22.101  6.4235e-06      135.19  37.086  26.203  31.192  8.5225e-06  447.67  41.828  19.414  34.196  0.00014543
2       32.438  20.126  40.334  32.331  6.453e-06       129.25  19.731  30.787  43.86   8.5048e-06  414.14  19.679  24.405  51.156  0.00014543
\end{filecontents*}
\pgfplotstabletypeset[
	every head row/.style={
		before row={
			\toprule
			\multicolumn{13}{c}{\textbf{Performances of \textsc{d\&c\_care\_2D} (Algorithm~\ref{alg:care_2d})}}\\
			\hline			
			\multicolumn{1}{c|}{}&
			\multicolumn{4}{c|}{ $n=100$}&
			\multicolumn{4}{c|}{$n=150$}&
			\multicolumn{4}{c}{$n=200$}\\
		},
		after row = \midrule,
	},
	columns = {0,1,3,2,4,6,8,7,9,11,13,12,14}, 
	columns/0/.style = {column name = $\ell$, column type=c|},	
	columns/1/.style = {column name = Time,precision=1,zerofill, fixed},	
	columns/2/.style = {column name = $T_{d\&c}^{\%}$, precision=0, fixed, column type=c},
	columns/3/.style = {column name = $T_{lr}^{\%}$, precision=0, fixed, column type=c},
	columns/4/.style = {column name = $T_{hss}^{\%}$, precision=0, fixed, column type=c|},	
	columns/6/.style = {column name = Time,precision=1,zerofill, fixed},	
	columns/7/.style = {column name = $T_{d\&c}^{\%}$, precision=0, fixed, column type=c},
	columns/8/.style = {column name = $T_{lr}^{\%}$, precision=0, fixed, column type=c},
	columns/9/.style = {column name = $T_{hss}^{\%}$, precision=0, fixed, column type=c|},	
	columns/11/.style = {column name = Time,precision=1,zerofill, fixed},	
	columns/12/.style = {column name = $T_{d\&c}^{\%}$, precision=0, fixed, column type=c},
	columns/13/.style = {column name = $T_{lr}^{\%}$, precision=0, fixed, column type=c},
	columns/14/.style = {column name = $T_{hss}^{\%}$, precision=0, fixed, column type=c},	
	]{test_riccati_2D.dat}
\end{table}

\upd{
\subsubsection{Allen--Cahn equation in two space dimensions}
Let us repeat the numerical test of Section~\ref{sec:allehn-cahn-1d} in a two-dimensional setting, providing an application of the structured Riccati solvers developed in the previous section.
More precisely, we consider the 2D analogue of equation~\eqref{AC} with $x\in \Omega = [-1,1]^2$, and initial condition chosen as $
y_0(x_1,x_2) = \sin(\pi x_1)\sin(\pi x_2)
$. The discretization of the Laplace operator is performed using a standard five-point finite difference stencil on a uniform Cartesian grid with $n$ equidistant points in both spatial directions. The other parameters are chosen as in Section~\ref{sec:allehn-cahn-1d}. As in the one-dimensional case, the uncontrolled dynamics admits two stable equilibria at $\pm 1$ and an unstable equilibrium at $\overline y = 0$. The objective of the control is therefore to stabilize the unstable equilibrium through a suitable feedback law. The SDRE solutions are approximated via Algorithm~\ref{alg:care_2d}, with $\ell=2$.

The quantitative performance of the proposed approach is reported in Table~\ref{table_AC2}, where we provide the average CPU time required to compute a single feedback control, together with the total cost attained by the controlled and uncontrolled dynamics, for increasing problem sizes $N=n^2$.  
As expected, the computational cost per control increases with the spatial resolution, reflecting the quadratic growth of the state dimension. Nonetheless, the total cost associated with the controlled dynamics remains essentially invariant with respect to $N$, indicating that the SDRE-based feedback law yields a consistent approximation of the underlying infinite-dimensional optimal control problem. In contrast, the uncontrolled dynamics results in significantly larger costs, thereby confirming the effectiveness of the proposed control strategy in stabilizing the unstable equilibrium.

Figure~\ref{fig:AC_2d} illustrates the qualitative behavior of the system for $n=100$.  
The top panels show the initial condition and the uncontrolled evolution, which converges toward one of the stable equilibria $\pm 1$ depending on the sign of the solution. By contrast, the controlled dynamics displayed in the bottom panel demonstrates that the SDRE-based feedback successfully stabilizes the unstable equilibrium $\overline y = 0$, in agreement with the behavior observed in the one-dimensional experiment. These results confirm that the proposed approach preserves its qualitative control properties when passing from one to two spatial dimensions, while remaining computationally feasible due to the structured Riccati solvers introduced in the preceding section.

\begin{table}[hbht]
\centering
\begin{tabular}{c|ccc}

$N$   & CPU per control & Controlled Total cost &  Uncontrolled Total cost
\\
\hline
2500    & 6.0 &   0.0961 &  3.8732\\
10000   & 75.7  &   0.0961  & 3.8846   \\
 \end{tabular}
  \caption{Computational cost and performance of the SDRE-based feedback control for the two-dimensional Allen--Cahn equation, for increasing problem sizes $N=n^2$. The table reports the average CPU time required to compute one feedback control, together with the total cost attained by the controlled and uncontrolled dynamics.}
 \label{table_AC2}
\end{table}

\begin{figure}[htbp]	
\centering
 \includegraphics[width=0.45\textwidth]{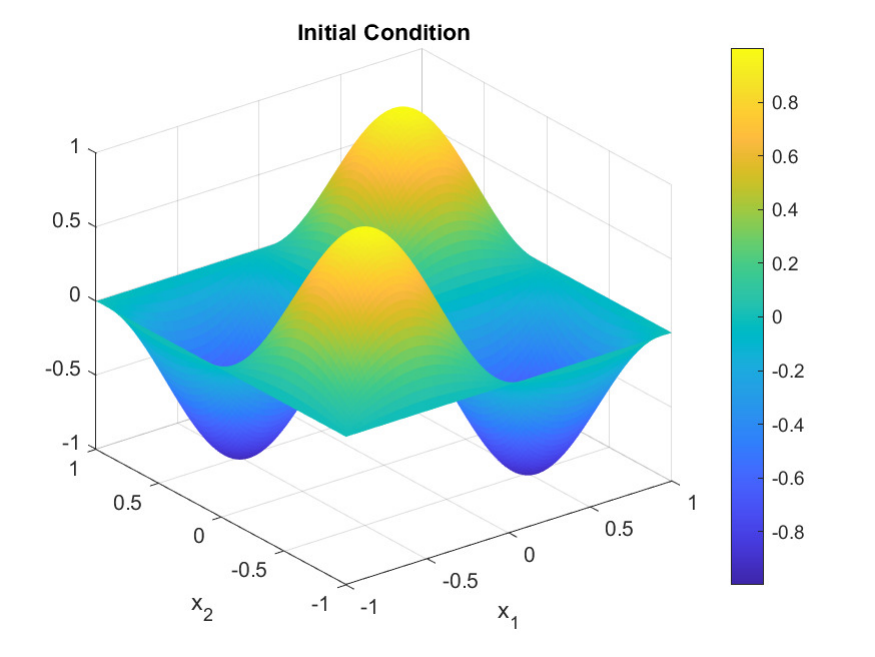}
	\includegraphics[width=0.45\textwidth]{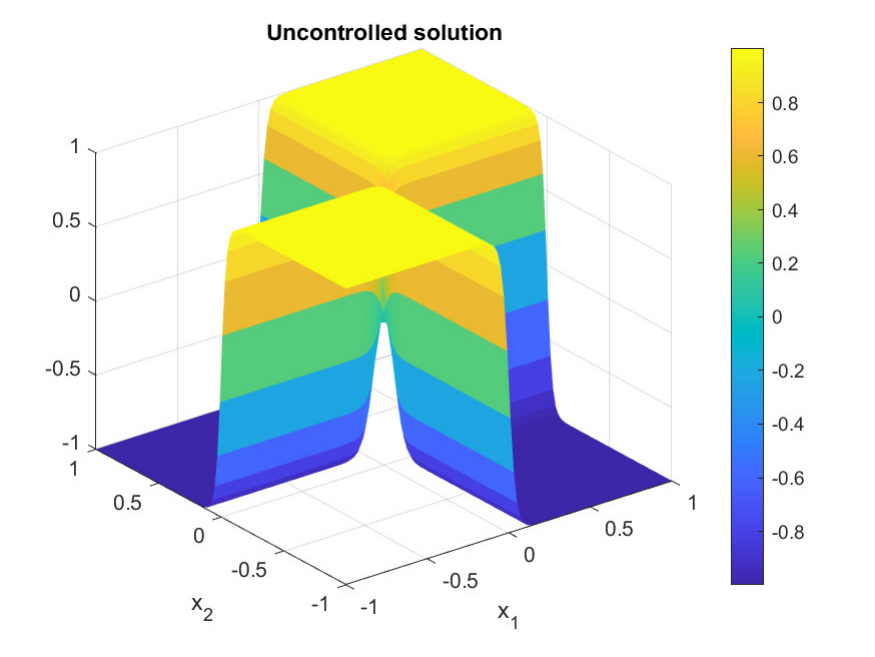}
		\includegraphics[width=0.45\textwidth]{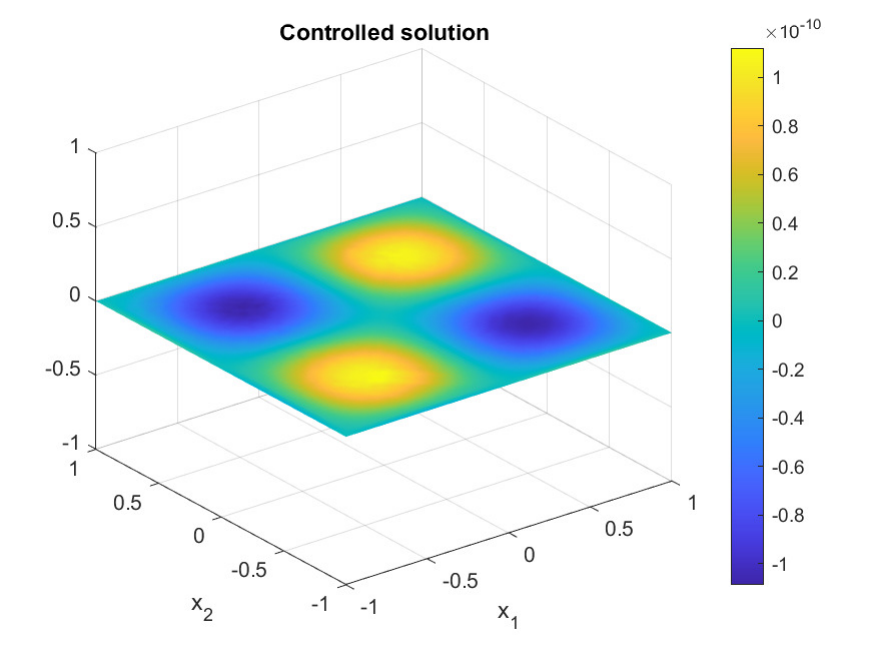}
	\caption{Initial condition (top left), uncontrolled solution (top right) and controlled solution via Algorithm \ref{alg:care_2d} with $n=100$ (bottom).} 
 \label{fig:AC_2d}
\end{figure}

}
\section{Concluding remarks}\label{sec:conclusion}
In this paper, we have demonstrated that under reasonable assumptions for control problems, continuous-time Riccati equations with quasiseparable coefficients have numerically quasiseparable solutions. Our decay bounds establish a connection between the off-diagonal singular values and Zolotarev numbers, which are linked either to the set $\mathcal{W}(L^{-1}AL)$ or to the spectra of the closed-loop matrices associated with the principal submatrices of the Hamiltonian. Additionally, our analysis provides enhanced estimates for the tensor train ranks of the value function corresponding to the CARE solutions.

Building on this theoretical framework, we have introduced practical methodologies to reduce the computational complexity of solving CAREs with large-scale quasiseparable coefficients. Specifically, we proposed two algorithms: one designed for general quasiseparable coefficients and another tailored for banded coefficients. Numerical experiments validate the scalability of these methods across both synthetic and real-world scenarios, particularly in control theory applications involving partial differential equations and agent-based models.

We remark that there are a few points that might deserve further investigations. First, with additional assumptions the parameter $t$ in the bound of Theorem~\ref{thm:sing-decay} might be improved by following the ideas in Remark~\ref{rem:improv1} and Remark~\ref{rem:improv2}, and this would automatically improve the decay bounds given in the corollaries of this result. Second, our bounds ensure the quasiseparable structure of the solution when $F$ is low-rank (apply Theorem~\ref{thm:quasi-zol} as in the proof of Theorem~\ref{thm:ttranks}) or when $F$ \upd{has} full rank (Corollary~\ref{cor:sing-decay3}, and Corollary~\ref{cor:sing-decay4}); however, it seems that the solution inherits the structure also in an intermediate scenario where $F$ is rank  deficient but not low-rank. In addition, according to the numerical results in Figure~\ref{fig:sing-decay2}, the role of $\kappa(F)$ in our upper bounds seems to be an artefact of our proof. Lastly, while testing of the \textsc{tink} algorithm for the banded case, we observed instances where the truncated Newton-Kleinman iteration exhibited a transient phase with large bandwidths, even though the sought solution was approximately banded. It would be of interest to explore whether alternative thresholding techniques could restore computational efficiency in such cases.

%

\bibliography{references}

\begin{thebibliography}{10}

\bibitem{alla2023}
A.~Alla, D.~Kalise, and V.~Simoncini.
\newblock State-dependent {Riccati} equation feedback stabilization for
  nonlinear {PDEs}.
\newblock {\em Adv. Comput. Math.}, 49(1):32, 2023.
\newblock Id/No 9.

\bibitem{anderson2007optimal}
B.~D. Anderson and J.~B. Moore.
\newblock {\em Optimal control: linear quadratic methods}.
\newblock Courier Corporation, 2007.

\bibitem{bailo2018optimal}
R.~Bailo, M.~Bongini, J.~A. Carrillo, and D.~Kalise.
\newblock Optimal consensus control of the {C}ucker-{S}male model.
\newblock {\em IFAC-PapersOnLine}, 51(13):1--6, 2018.

\bibitem{Banks_Lewis_Tran_2007}
H.~T. Banks, B.~M. Lewis, and H.~T. Tran.
\newblock Nonlinear feedback controllers and compensators: a state-dependent
  {Riccati} equation approach.
\newblock {\em Computational Optimization and Applications}, 37(2):177–218,
  2007.

\bibitem{bardi1997optimal}
M.~Bardi, I.~C. Dolcetta, et~al.
\newblock {\em Optimal control and viscosity solutions of
  Hamilton-Jacobi-Bellman equations}, volume~12.
\newblock Springer, 1997.

\bibitem{beckerman06}
B.~Beckermann, S.~A. Goreinov, and E.~E. Tyrtyshnikov.
\newblock Some remarks on the {Elman} estimate for {GMRES}.
\newblock {\em SIAM J. Matrix Anal. Appl.}, 27(3):772--778, 2006.

\bibitem{beckermann17}
B.~Beckermann and A.~Townsend.
\newblock On the singular values of matrices with displacement structure.
\newblock {\em SIAM J. Matrix Anal. Appl.}, 38(4):1227--1248, 2017.

\bibitem{benner2009}
P.~Benner.
\newblock System-theoretic methods for model reduction of large-scale systems:
  Simulation, control, and inverse problems.
\newblock In {\em Proceedings of MATHMOD}, volume 2009, pages 126--145. Vienna,
  Austria: ARGESIM, 2009.

\bibitem{benner2015numerical}
P.~Benner, M.~Bollh{\"o}fer, D.~Kressner, C.~Mehl, and T.~Stykel.
\newblock Numerical algebra, matrix theory, differential-algebraic equations
  and control theory.
\newblock {\em Switzerland: Springer International Publishing}, 2015.

\bibitem{radi18}
P.~Benner, Z.~Bujanovi\'c, P.~K\"urschner, and J.~Saak.
\newblock R{ADI}: a low-rank {ADI}-type algorithm for large scale algebraic
  {R}iccati equations.
\newblock {\em Numer. Math.}, 138(2):301--330, 2018.

\bibitem{benner20}
P.~Benner, Z.~Bujanovi\'c, P.~K\"urschner, and J.~Saak.
\newblock A numerical comparison of different solvers for large-scale,
  continuous-time algebraic {R}iccati equations and {LQR} problems.
\newblock {\em SIAM J. Sci. Comput.}, 42(2):A957--A996, 2020.

\bibitem{benner2016inexact}
P.~Benner, M.~Heinkenschloss, J.~Saak, and H.~K. Weichelt.
\newblock An inexact low-rank {N}ewton--{ADI} method for large-scale algebraic
  {R}iccati equations.
\newblock {\em Applied Numerical Mathematics}, 108:125--142, 2016.

\bibitem{benzi-boito}
M.~Benzi, P.~Boito, and N.~Razouk.
\newblock Decay properties of spectral projectors with applications to
  electronic structure.
\newblock {\em SIAM Rev.}, 55(1):3--64, 2013.

\bibitem{bini2011numerical}
D.~A. Bini, B.~Iannazzo, and B.~Meini.
\newblock {\em Numerical solution of algebraic {R}iccati equations}.
\newblock SIAM, 2011.

\bibitem{breiten2015feedback}
T.~Breiten and K.~Kunisch.
\newblock Feedback stabilization of the schl{\"o}gl model by lqg-balanced
  truncation.
\newblock In {\em 2015 European Control Conference (ECC)}, pages 1171--1176.
  IEEE, 2015.

\bibitem{ccimen2008state}
T.~{\c{C}}imen.
\newblock State-dependent {R}iccati equation ({SDRE}) control: a survey.
\newblock {\em IFAC Proceedings Volumes}, 41(2):3761--3775, 2008.

\bibitem{crouzeix}
M.~Crouzeix and C.~Palencia.
\newblock The numerical range is a {{\((1+\sqrt{2})\)}}-spectral set.
\newblock {\em SIAM J. Matrix Anal. Appl.}, 38(2):649--655, 2017.

\bibitem{cucker2007emergent}
F.~Cucker and S.~Smale.
\newblock Emergent behavior in flocks.
\newblock {\em IEEE Transactions on automatic control}, 52(5):852--862, 2007.

\bibitem{dolgov21}
S.~Dolgov, D.~Kalise, and K.~K. Kunisch.
\newblock Tensor decomposition methods for high-dimensional
  {H}amilton-{J}acobi-{B}ellman equations.
\newblock {\em SIAM J. Sci. Comput.}, 43(3):A1625--A1650, 2021.

\bibitem{dolgov2023}
S.~Dolgov, D.~Kalise, and L.~Saluzzi.
\newblock Data-driven tensor train gradient cross approximation for
  {Hamilton}-{Jacobi}-{Bellman} equations.
\newblock {\em SIAM J. Sci. Comput.}, 45(5):a2153--a2184, 2023.

\bibitem{dolgov13}
S.~Dolgov and B.~Khoromskij.
\newblock Two-level {QTT}-{T}ucker format for optimized tensor calculus.
\newblock {\em SIAM J. Matrix Anal. Appl.}, 34(2):593--623, 2013.

\bibitem{chebfun}
T.~A. Driscoll, N.~Hale, and L.~N. Trefethen.
\newblock Chebfun guide, 2014.

\bibitem{eidelman14}
Y.~Eidelman, I.~Gohberg, and I.~Haimovici.
\newblock {\em Separable type representations of matrices and fast algorithms.
  {V}ol. 1}, volume 234 of {\em Operator Theory: Advances and Applications}.
\newblock Birkh\"{a}user/Springer, Basel, 2014.
\newblock Basics. Completion problems. Multiplication and inversion algorithms.

\bibitem{elman1982iterative}
H.~C. Elman.
\newblock {\em Iterative methods for large, sparse, nonsymmetric systems of
  linear equations}.
\newblock Yale University, 1982.

\bibitem{embree2023}
M.~Embree.
\newblock Extending {Elman}'s bound for {GMRES}.
\newblock {\em Linear Algebra and its Applications}, 726:54--70, 2025.

\bibitem{feitzinger2009inexact}
F.~Feitzinger, T.~Hylla, and E.~W. Sachs.
\newblock Inexact {K}leinman--{N}ewton method for {R}iccati equations.
\newblock {\em SIAM Journal on Matrix Analysis and Applications},
  31(2):272--288, 2009.

\bibitem{gawlik}
E.~S. Gawlik.
\newblock Zolotarev iterations for the matrix square root.
\newblock {\em SIAM J. Matrix Anal. Appl.}, 40(2):696--719, 2019.

\bibitem{grasedyck08}
L.~Grasedyck.
\newblock Nonlinear multigrid for the solution of large-scale {Riccati}
  equations in low-rank and {{\(\mathcal H\)}}-matrix format.
\newblock {\em Numer. Linear Algebra Appl.}, 15(9):779--807, 2008.

\bibitem{grasedyck03}
L.~Grasedyck, W.~Hackbusch, and B.~N. Khoromskij.
\newblock Solution of large scale algebraic matrix {Riccati} equations by use
  of hierarchical matrices.
\newblock {\em Computing}, 70(2):121--165, 2003.

\bibitem{haber18}
A.~Haber and M.~Verhaegen.
\newblock Sparsity preserving optimal control of discretized {PDE} systems.
\newblock {\em Comput. Methods Appl. Mech. Eng.}, 335:610--630, 2018.

\bibitem{hackbusch15}
W.~Hackbusch.
\newblock {\em Hierarchical matrices: algorithms and analysis}, volume~49 of
  {\em Springer Series in Computational Mathematics}.
\newblock Springer, Heidelberg, 2015.

\bibitem{halikias2024structured}
D.~Halikias and A.~Townsend.
\newblock Structured matrix recovery from matrix-vector products.
\newblock {\em Numerical Linear Algebra with Applications}, 31(1):e2531, 2024.

\bibitem{halko2011finding}
N.~Halko, P.-G. Martinsson, and J.~A. Tropp.
\newblock Finding structure with randomness: Probabilistic algorithms for
  constructing approximate matrix decompositions.
\newblock {\em SIAM review}, 53(2):217--288, 2011.

\bibitem{kazeev13}
V.~Kazeev, O.~Reichmann, and C.~Schwab.
\newblock Low-rank tensor structure of linear diffusion operators in the {TT}
  and {QTT} formats.
\newblock {\em Linear Algebra Appl.}, 438(11):4204--4221, 2013.

\bibitem{kleinman1968iterative}
D.~Kleinman.
\newblock On an iterative technique for {R}iccati equation computations.
\newblock {\em IEEE Transactions on Automatic Control}, 13(1):114--115, 1968.

\bibitem{kressner20}
D.~Kressner, P.~K\"{u}rschner, and S.~Massei.
\newblock Low-rank updates and divide-and-conquer methods for quadratic matrix
  equations.
\newblock {\em Numer. Algorithms}, 84(2):717--741, 2020.

\bibitem{kressner19}
D.~Kressner, S.~Massei, and L.~Robol.
\newblock Low-rank updates and a divide-and-conquer method for linear matrix
  equations.
\newblock {\em SIAM J. Sci. Comput.}, 41(2):A848--A876, 2019.

\bibitem{kressner23}
D.~Kressner, S.~Massei, and J.~Zhu.
\newblock Improved {P}ara{D}iag via low-rank updates and interpolation.
\newblock {\em Numer. Math.}, 155(1-2):175--209, 2023.

\bibitem{lancaster}
P.~Lancaster and L.~Rodman.
\newblock {\em Algebraic {Riccati} equations}.
\newblock Oxford: Clarendon Press, 1995.

\bibitem{levitt2024linear}
J.~Levitt and P.-G. Martinsson.
\newblock Linear-complexity black-box randomized compression of rank-structured
  matrices.
\newblock {\em SIAM Journal on Scientific Computing}, 46(3):A1747--A1763, 2024.

\bibitem{massei18}
S.~Massei, D.~Palitta, and L.~Robol.
\newblock Solving rank-structured {S}ylvester and {L}yapunov equations.
\newblock {\em SIAM J. Matrix Anal. Appl.}, 39(4):1564--1590, 2018.

\bibitem{massei20}
S.~Massei, L.~Robol, and D.~Kressner.
\newblock hm-toolbox: {MATLAB} software for {HODLR} and {HSS} matrices.
\newblock {\em SIAM J. Sci. Comput.}, 42(2):C43--C68, 2020.

\bibitem{melman2013generalization}
A.~Melman.
\newblock Generalization and variations of {P}ellet’s theorem for matrix
  polynomials.
\newblock {\em Linear algebra and its applications}, 439(5):1550--1567, 2013.

\bibitem{oseledets11}
I.~V. Oseledets.
\newblock Tensor-train decomposition.
\newblock {\em SIAM J. Sci. Comput.}, 33(5):2295--2317, 2011.

\bibitem{oseledets13}
I.~V. Oseledets.
\newblock Constructive representation of functions in low-rank tensor formats.
\newblock {\em Constr. Approx.}, 37(1):1--18, 2013.

\bibitem{silvester2000determinants}
J.~R. Silvester.
\newblock Determinants of block matrices.
\newblock {\em The Mathematical Gazette}, 84(501):460--467, 2000.

\bibitem{sima}
V.~Sima.
\newblock {\em Algorithms for linear-quadratic optimization}, volume 200 of
  {\em Pure Appl. Math., Marcel Dekker}.
\newblock New York, NY: Marcel Dekker, 1996.

\bibitem{simoncini07}
V.~Simoncini.
\newblock A new iterative method for solving large-scale {L}yapunov matrix
  equations.
\newblock {\em SIAM J. Sci. Comput.}, 29(3):1268--1288, 2007.

\bibitem{simoncini16b}
V.~Simoncini.
\newblock Analysis of the rational {K}rylov subspace projection method for
  large-scale algebraic {R}iccati equations.
\newblock {\em SIAM J. Matrix Anal. Appl.}, 37(4):1655--1674, 2016.

\bibitem{simoncini16}
V.~Simoncini.
\newblock Computational methods for linear matrix equations.
\newblock {\em SIAM Rev.}, 58(3):377--441, 2016.

\bibitem{simoncini14}
V.~Simoncini, D.~B. Szyld, and M.~Monsalve.
\newblock On two numerical methods for the solution of large-scale algebraic
  {R}iccati equations.
\newblock {\em IMA J. Numer. Anal.}, 34(3):904--920, 2014.

\bibitem{townsend18}
A.~Townsend and H.~Wilber.
\newblock On the singular values of matrices with high displacement rank.
\newblock {\em Linear Algebra Appl.}, 548:19--41, 2018.

\bibitem{vandebril08}
R.~Vandebril, M.~Van~Barel, and N.~Mastronardi.
\newblock {\em Matrix computations and semiseparable matrices. {V}ol. 1}.
\newblock Johns Hopkins University Press, Baltimore, MD, 2008.
\newblock Linear systems.

\bibitem{xia2012complexity}
J.~Xia.
\newblock On the complexity of some hierarchical structured matrix algorithms.
\newblock {\em SIAM Journal on Matrix Analysis and Applications},
  33(2):388--410, 2012.

\bibitem{xia10}
J.~Xia, S.~Chandrasekaran, M.~Gu, and X.~S. Li.
\newblock Fast algorithms for hierarchically semiseparable matrices.
\newblock {\em Numer. Linear Algebra Appl.}, 17(6):953--976, 2010.

\end{thebibliography}
\bibliographystyle{abbrv}
\end{document}